\documentclass[11pt,reqno]{amsart}

\numberwithin{equation}{section}
3
\usepackage{times}
\usepackage{amsmath,amsfonts,amstext,amssymb,amsbsy,amsopn,amsthm,eucal}
\usepackage{txfonts}
\usepackage{dsfont}
\usepackage{graphicx}   % for figures
\usepackage{hyperref}
\usepackage{accents}
\usepackage{enumerate}
\usepackage{xcolor}
\usepackage{verbatim}
\usepackage{esint}

\usepackage{mathtools,leftindex,tensor}
\usepackage[version=4]{mhchem}

\usepackage[normalem]{ulem}
\usepackage{cancel}

\newcommand{\RR}{\mathds{R}}

\newcommand{\tr}{\mathrm{tr}}

\newcommand{\Rm}{{\rm Rm}}
\newcommand{\Ric}{{\rm Ric}}

\newcommand{\Vol}{{\rm Vol}}
\newcommand{\diam}{{\rm diam}}

\newcommand{\Hess}{\mathrm{Hess}}

\newcommand{\R}{\mathrm{R}}
\newcommand{\E}{\mathrm{E}}

\newcommand{\cC}{\mathcal{C}}

\newcommand{\cH}{\mathcal{H}}
\newcommand{\cI}{\mathcal{I}}

\newcommand{\cL}{\mathcal{L}}

\newcommand{\cM}{\mathcal{M}}
\newcommand{\cN}{\mathcal{N}}
\newcommand{\cO}{\mathcal{O}}

\newcommand{\cP}{\mathcal{P}}

\newcommand{\cR}{\mathcal{R}}
\newcommand{\cS}{\mathcal{S}}

\newcommand{\cV}{\mathcal{V}}
\newcommand{\cW}{\mathcal{W}}

\newcommand{\supp}{\mathrm{supp }}

\newcommand{\pr}{\mathrm{pr}}

\newtheorem{thm}{Theorem}[section]

\newtheorem{lem}[thm]{Lemma}
\newtheorem{prop}[thm]{Proposition}
\newtheorem{cor}[thm]{Corollary}

\newtheorem{defn}[thm]{Definition}

\newtheorem{exmp}[thm]{Example}

\newtheorem{theorem}[thm]{Theorem}

\newtheorem{proposition}[thm]{Proposition}
\newtheorem{lemma}[thm]{Lemma}
\newtheorem{corollary}[thm]{Corollary}

\theoremstyle{definition}
\newtheorem{definition}[thm]{Definition}

\theoremstyle{remark}
\newtheorem{remark}[thm]{Remark}

\newtheorem{conjecture}[thm]{Conjecture}
\newtheorem{question}[thm]{Question}

\begin{document}

\thanks{}

\title[]{Regularity of Einstein $5$-manifolds via $4$-dimensional gap theorems}

\author{Yiqi Huang and Tristan Ozuch}
\address{Department of Mathematics, MIT, 77 Massachusetts Avenue, Cambridge, MA 02139-4307, USA}
 \email{yiqih777@mit.edu, ozuch@mit.edu}

%\date{\today}

\vspace{-0.3cm}

\begin{abstract}
    We refine the regularity of noncollapsed limits of $5$-dimensional manifolds with bounded Ricci curvature. In particular, for noncollapsed limits of Einstein $5$-manifolds, we prove that
    \begin{itemize}
        \item tangent cones are unique of the form $\mathbb{R}\times\mathbb{R}^4/\Gamma$ on the top stratum, hence outside a countable set of points; this follows from a new isolation result for cones of the form $\mathbb{R}\times\mathbb{R}^4/\Gamma$ among all tangent cones,
        %\item {\color{red} this uniqueness follows from a new isolation result for the cones $\mathbb{R}\times\mathbb{R}^4/\Gamma$ among all tangent cones, independent of \L{}ojasiewicz--Simon inequalities and other techniques in the existing literature,}
        \item the singular set is entirely contained in a countable union of Lipschitz curves and points,
        \item away from a nowhere dense subset, these Lipschitz curves consist of smooth geodesics,
        \item the interior of any geodesic is removable: limits of Einstein manifolds are real-analytic orbifolds with singularities along geodesic and bounded curvature away from their extreme points, and
        \item if an asymptotically Ricci-flat $5$-manifold with Euclidean volume growth has one tangent cone at infinity that splits off a line, then it is the unique tangent cone at infinity.
    \end{itemize}
    These results prompt the question of the orbifold regularity of noncollapsed limits of Einstein manifolds off a codimension $5$ set in arbitrary dimension. 
    
    The proofs rely on a new result of independent interest: \textit{all} spherical and hyperbolic $4$-orbifolds are isolated among Einstein $4$-orbifolds in the Gromov-Hausdorff sense. This yields various gap theorems for Einstein $4$-orbifolds, which do not extend to higher dimensions. The proofs of these gap theorems require a careful analysis of singular metrics and families of metrics that degenerate.
\end{abstract}
\maketitle
\vspace{-1.4cm}
\setcounter{tocdepth}1
\tableofcontents

\section{Introduction}

An \textit{Einstein manifold} $(M,g)$ satisfies the fundamental curvature condition
\begin{equation}\label{eq:Einstein}
\Ric(g) = \Lambda g, \quad \Lambda \in \mathbb{R}.
\end{equation} 
Such metrics are distinguished not only by having a curvature constant but also as stationary points of significant geometric functionals. Another curvature condition considered in this article is that of \textit{bounded Ricci curvature}, which we normalize in dimension $n$ as $$-(n-1)g\leqslant \Ric\leqslant (n-1)g.$$
In dimensions greater than or equal to four, compact manifolds with uniform diameter bounds and noncollapsing conditions, sequences of Einstein metrics may develop singularities. Understanding these singularities and their formation has been a central question since the 1970's. By the end of the 1980's, noncollapsed limits of $4$-manifolds were roughly understood. More recently, the $4$-dimensional picture has been further clarified and a basic structure is emerging in higher-dimensions.

In the present article, we refine our understanding of limits in dimension $4$ and deduce several gap theorems for spherical and hyperbolic orbifolds, only valid in this dimension. This leads to substantial applications to the regularity of limits of Einstein $5$-manifolds, towards an orbifold regularity theorem away from a conjectually codimension $5$ singular set.

\subsection{Background}

\subsubsection{Noncollapsed degeneration of Einstein $4$-manifolds}

A central objective in four-dimensional differential geometry is to understand the structure and topology of the moduli space of Einstein metrics on a given compact manifold $M^4$, i.e. the set of Einstein metrics modulo the action of the diffeomorphism group, denoted $\mathbf{E}(M^4)$. When equipped with the Gromov-Hausdorff metric $d_{GH}$, the completion of $\mathbf{E}(M^4)$ includes Einstein orbifolds, as proven in \cite{and,bkn,tia}. The need for a common background manifold $M^4$ was weakened to a noncollapsing assumption in \cite{cn}. The structure of this completion has recently been improved substantially in \cite{biq1,biq2,ozu1,ozu2,ozuthese}. It has been shown in particular that the moduli space of Einstein metrics $GH$-close to a given Einstein $4$-orbifold $(M_o,g_o)$ was embedded in finite-dimensional set obtained by a gluing construction. Its dimension is controlled by the noncollapsedness. 

It has long been an open question, explicitly posed in \cite{and94, andsurv}, whether Einstein orbifolds can be approximated or resolved by sequences of smooth Einstein metrics. Such approximations have been obtained by gluing-perturbation, see e.g. \cite{biq1,don,LS}. On the other hand, in \cite{ozu4}, it was proven that spherical and hyperbolic orbifolds with the simplest singularities, $\mathbb{R}^4/\pm$, could never be limits of smooth Einstein $4$-manifolds in the Gromov-Hausdorff sense. In the present article, we show that \textit{all} hyperbolic and spherical orbifolds are $GH$-isolated, even within the class of Einstein $4$-\textit{orbifolds}. This is not a mere extension of \cite{ozu4} and requires completely new ideas in order to deal with the deeply technical case of multiscale trees of singularities and singular metrics. The resulting isolation theorem implies a series of new surprising quantitative gap phenomena for Einstein $4$-orbifolds. These results are genuinely $4$-dimensional: they have no analogue in higher dimensions $n\geqslant5$, even with isolated singularities. This is the first instance of gap theorems that deals with orbifold degenerations, while previous ones such as \cite{hm25} relied on topological or geometric assumptions that prevented said degenerations.

\subsubsection{Regularity of noncollapsed Einstein $n$-manifolds for $n\geqslant5$}

The higher-dimensional Einstein manifolds have been extensively studied over the last four decades. The modern compactness theory originates with the Cheeger-Gromov program of pointed $C^{\infty}$ convergence under curvature control.
A foundational leap was the structure theory for spaces with lower Ricci curvature bounds
\cite{cc, col, cc97,cc00a,cc00b}, which introduced the stratification of the singular set and established the regularity of the limit space. 

\emph{Uniqueness of tangent cones.}
One of the key problems in the regularity theory is the uniqueness of the tangent cones. For a noncompact manifold with nonnegative Ricci curvature and Euclidean volume growth, it was proved in \cite{cc} that the tangent cone at infinity is a metric cone. A priori, however, different rescaling sequences may lead to different tangent cones.  A long-standing well-known question asks whether the cross-section of the tangent cone at infinity for noncollapsed Ricci-flat manifolds depends on the rescaling sequence of the blow-downs. It is known that uniqueness may fail if assuming merely the lower bound on the Ricci curvature \cite{per,cc,con13} and for collapsed Ricci-flat metrics \cite{hat1}.

%There is a substantial literature on this uniqueness problem of geometric PDEs. For minimal surface, the uniqueness was established assuming the integrability and smoothness of the cross-sections in \cite{aa}, and the integrability assumption was later removed in the seminal work of \cite{sim}. For more general nonsmooth cross-sections, the uniqueness of cones $\RR \times C(\mathbb{S}^p \times \mathbb{S}^q)$ for $p+q \ge7$ was proved in \cite{sim94}. Only very recently was the uniqueness of the cylindrical Simons cone, $\RR \times C(\mathbb{S}^3 \times \mathbb{S}^3)$ established in \cite{sze20}. The case $\RR \times C(\mathbb{S}^2 \times \mathbb{S}^4) $ was subsequently addressed in \cite{ftw}. See also the uniqueness for two-dimensional area-minimizing integral currents \cite{whi}.  Analogous uniqueness problems arise in other geometric PDEs, for example, mean curvature flow \cite{sch14,cim,cm15,cs21,cm23,lz}, and Ricci flow \cite{sw,cmz,lw,cm25,fla} and the references therein. 

Such uniqueness for noncollapsed Ricci-flat manifolds was first established by Cheeger-Tian \cite{ct} under the assumption that the cross-sections of all tangent cones are smooth and integrable. Two decades later, Colding-Minicozzi \cite{cm14} proved uniqueness assuming only that a single tangent cone is smooth, using the Łojasiewicz-Simon inequality. See also the related developments \cite{ksz24,yz25}. By contrast, when the cone has a nonsmooth cross-section -- most notably for cones that split off a line $\mathbb{R}\times C(Z)$ -- the theory remains far less developed. In the Kähler-Einstein setting, Donaldson-Sun \cite{ds17} established uniqueness of tangent cones by means of algebro-geometric methods.

As our first result in higher dimensions, we prove that in dimension five, if one tangent cone splits off a line, then it is the unique tangent cone. Our argument does not rely on the Łojasiewicz-Simon inequality as in \cite{cm14}, nor on algebro-geometric techniques as in \cite{ds17}. To the best of the authors’ knowledge, this is the first uniqueness result for a general Ricci-flat cone with a nonsmooth cross-section. As a consequence, we obtain a clearer description of the asymptotic geometry and new curvature decay estimates for Ricci-flat $5$-manifolds.

Analogous uniqueness questions arise in other geometric PDEs, including minimal surface theory \cite{aa,sim,whi,sim94,sze20b,cl,ftw}, mean curvature flow \cite{sch14,cim,cm15,cs21,cm23,lz,bl}, and Ricci flow \cite{sw,cmz,lw,cm25,fla}, where uniqueness has been obtained under various smoothness, integrability, and structural assumptions on the cross-sections. We refer to the literature for further developments.

\emph{Regularity of limit spaces and singular sets.} 
Through a series of works \cite{and,bkn,tia,cn}, it was shown that limits of Einstein $4$-manifolds with bounded Ricci curvature and uniform lower volume bounds are Riemannian orbifolds with locally finitely many singularities.  
A natural question is whether such structural results extend to higher dimensions.  
Key recent milestones include the convexity of the regular set \cite{con12}, the resolution of the codimension-four conjecture \cite{cn} and the \emph{rectifiability} of the singular set  \cite{cjn,jn}. Here, \emph{rectifiability} means that, up to an \emph{$\cH^{n-4}$-measure zero} subset, $\cS$ is contained in a countable union of bi-Lipschitz images of subsets of~$\mathbb{R}^{n-4}$.

Several issues remained largely open, notably the treatment of the exceptional \emph{$\cH^{n-4}$-measure zero} subset and the possible emergence of an orbifold structure in higher dimensions. In dimension~$5$ we prove:  

\begin{enumerate}
\item \emph{Refined rectifiability and convexity.}  
We prove the uniqueness of the tangent cone on the \emph{entire} top stratum: if one tangent cone is $\RR \times \RR^4/\Gamma$, then it is the unique tangent cone. The uniqueness yields a refined stratification of the singular set $\cS$ as a \emph{finite disjoint} union of sets $\cS^0$ and $\cS^1_{\Gamma}$ for finitely many $\Gamma\subset O(4)$, where $\cS^0$ consists of countably many points whose tangent cones do not split, and $\cS^1_{\Gamma}$ consists of points whose tangent cone is $\mathbb{R}\times \mathbb{R}^4/\Gamma$.
We show that each $\cS^1_{\Gamma}$ is locally contained in a bi-Lipschitz image of a subset of $\mathbb{R}$, thereby establishing a bi-Lipschitz manifold structure for the singular set without discarding any \emph{$\cH^1$-measure zero} subset. Moreover, we prove that each connected component of $\cS^1_{\Gamma}$ is either a single point or a geodesic, and hence \emph{locally convex}.\\

%We prove that the entire singular set is contained in a countable union of points and bi-Lipschitz images of subsets of $\mathbb{R}$, without discarding any \emph{$\cH^1$-measure zero} subset. This refinement relies on the uniqueness of local tangent cones: while \cite{cjn,jn} established uniqueness $\cH^{n-4}$-almost everywhere, in dimension~$5$ we prove uniqueness on the \emph{entire} top stratum, hence away from a countable collection of points. Moreover, the uniqueness implies a refined stratification of the singular set $\cS$, as finitely many disjoint unions of $\cS^0$ and $\cS^1_{\Gamma}$, where $\cS^0$ is a set of countably many $0$-symmetric points, and $\cS^1_{\Gamma}$ consists of the points whose tangent cone is $\RR \times \RR^4_{\Gamma}$. We further prove that each connected component of $\cS^1_{\Gamma}$ is locally convex. \\

\item \emph{Orbifold regularity.}  
We extend the orbifold regularity to dimension $5$ and formulate a \emph{codimension-five orbifold conjecture} for Einstein limits $X$.  
In particular, the singular curves in $\cS^1_{\Gamma}$ are geodesic and possess a real-analytic \textit{orbifold} structures with \textit{bounded curvature}. Together with the geodesics in the regular set, these exhaust all geodesics in the limit space.
This leads us to introduce the \emph{orbifold-regular-singular} decomposition $X = \cR_{\mathrm{orb}} \cup \cS_{\mathrm{orb}}$. We prove that the orbifold-regular set $\cR_{\mathrm{orb}}$ is an open \textit{convex} 5-dimensional orbifold, and the orbifold-singular set $\cS_{\mathrm{orb}}$ consists precisely of the \emph{extreme points} of $X$, namely those points that do not lie on any geodesic joining two other points. We further show that $\cS_{\mathrm{orb}}$ is contained in a \emph{nowhere dense} subset of a countable collection of bi-Lipschitz curves. Since Einstein metrics are real-analytic, we further conjecture -- by unique continuation -- that $\cS_{\mathrm{orb}}$ is countable. This will be discussed in detail in Section \ref{subs:conjectures}. 

%the singular curves, away from a \textit{nowhere-dense} subset $\cS_{\mathrm{orb}}$, are geodesic and possess a real-analytic \textit{orbifold} structures with \textit{bounded curvature}. By classifying all geodesics in the limit space $(X,d)$, we also obtain more metric geometry properties of the limit space, generalizing \cite{con12}. Namely, the orbifold-regular set $\cR_{\mathrm{orb}}:= X \setminus \cS_{\mathrm{orb}}$ is an open \textit{convex} orbifold, and the orbifold-singular set $\cS_{\mathrm{orb}}$ is the set of \emph{extreme points} of the limit space. Since Einstein metrics are real-analytic, we further conjecture -- by unique continuation -- that $\cS_{\mathrm{orb}}$ is countable. This will be discussed in detail in Section \ref{subs:conjectures}. 
\end{enumerate}

Many of the results above
%, except for the orbifold structure, 
require only bounded Ricci curvature, and less central statements hold in all dimensions $n\geqslant 5$. See \cite{col96a,col96b, cct, ct06,col12, chn13,bam17,bam18,bam20} for more developments in this direction.

\subsection{Main results}

\subsubsection{Gap theorems for Einstein $4$-manifolds}

We first show the isolation of all hyperbolic or spherical $4$-orbifolds, i.e. spaces with isolated flat conical singularities and sectional curvatures constant equal to $\pm1$.

\begin{theorem}[Isolation theorem]\label{thm:isolation orbifolds}
    For any compact hyperbolic or spherical $4$-\emph{orbifold} $(M_o,g_o)$ with isolated singularities, there exists $\delta_0 >0$ such that if an Einstein \emph{orbifold} $(M,g)$ with $ \Ric(g) = \pm 3g$ satisfies
    $$ d_{GH}\left((M,g),(M_o,g_o)\right) < \delta_0,$$
    then $(M,g)$ is \emph{isometric} to $(M_o,g_o)$.
\end{theorem}

This result extends that of \cite{ozu4} that was only dealing with smooth metrics close to orbifolds with simplest singularities $\mathbb{R}^4/\mathbb{Z}_2$. The present article deals with much more challenging general case which allows formation of \textit{trees of singularities} which substantially complicate the problem and require new ideas.

\begin{remark}
    There are Ricci-flat $4$-manifolds approaching flat orbifolds \cite{LS} on $K3$, but not on arbitrary topologies \cite{bk,ozu3}.
\end{remark}

Theorem \ref{thm:isolation orbifolds} yields gap theorems for Einstein $4$-orbifolds with no extension to higher dimensions.

\begin{corollary}[Gap theorems]\label{cor:gap thms}
For all $v>0$, there exists $\varepsilon_0=\varepsilon_0(v)>0$ such that if an Einstein $4$-orbifold $(M,g)$ with $\Ric(g)=3g$ and $\Vol(M,g)>v$ satisfies one of the following:
\begin{enumerate}
  \item \emph{(Almost Bonnet-Myers theorem)} $\diam(M,g)>\pi-\varepsilon_0$,
  \item \emph{(Almost Obata-Lichnerowicz theorem)} the first eigenvalue $\lambda_1$ of the Laplacian satisfies $\lambda_1<4+\varepsilon_0$,
  \item \emph{(Almost conformally flat)} $\|W_g\|_{L^1}<\varepsilon_0$,
  \item \emph{(Almost Lévy-Gromov theorem)} there exists some $t\in(17^{-17},1-17^{-17})$ with $\big|I_{(M,g)}(t)-I_{\mathbb S^4}(t)\big|\leqslant \varepsilon_0$, where $I$ denotes the isoperimetric profile $I_{(M,g)}(t):=\inf\big\{\mathrm{Per}(\Omega): \Omega\subset M, \Vol_g(\Omega)=t\,\Vol_g(M)\big\},$
  \item \emph{(Almost Bishop-Gromov theorem)} $(M,g)$ has a singularity modeled on $\mathbb{R}^4/\Gamma$ and $\Vol(M)>\frac{|\mathbb S^4|}{|\Gamma|}-\varepsilon_0$,
\end{enumerate}
then $(M,g)$ is a spherical $4$-orbifold. In all cases but (3), $(M,g)$ is isometric to a spherical suspension over $\mathbb S^3/\Gamma$ for some finite $\Gamma\subset O(4)$.
\end{corollary}
\begin{remark}
    From \cite{ctz}, spherical 4-orbifolds are spherical suspensions over $\mathbb S^3/\Gamma$ and their $\mathbb{Z}_2$-quotient. 
\end{remark}

\begin{remark}
Each gap theorem above strengthens well-known rigidity theorems for spherical suspensions. For (1) use the maximal-diameter theorem of Cheng \cite{che75}, for (2) use the Lichnerowicz-Obata theorem \cite{lic58,oba62}, for (4) use the Lévy-Gromov estimate \cite{gro} (see a survey \cite{ant25}), for (5) use Bishop-Gromov based at the singular point \cite{bc64,gro}. In the RCD setting, equality yields a spherical suspension; see for instance \cite{ket15,cam17}. Note also that the $L^1$ smallness in (3) is implied by the $L^p$ smallness with any $p\geqslant 1$ and that gap theorems for smooth Einstein $n$-manifolds with small $L^{\frac{n}{2}}$ norm of their Weyl tensors are proven in \cite{Sin92,She90,IS02}.
\end{remark}

%\begin{remark} {\color{red}This is the first instance of gap theorems that deals with orbifold degenerations, while previous ones such as \cite{hm25} relied on topological or geometric assumptions that prevented said degenerations.} \end{remark}

\begin{remark}
    Direct generalizations of Corollary \ref{cor:gap thms} fail in dimension $5$ and higher: in \cite{boh}, Böhm constructed a sequence of Einstein $5$-metrics $(S^5,g_i)$ with $\Ric(g_i) = 4g_i$ and volume uniformly bounded below such that $(S^5,g_i)_{i\in\mathbb{N} }$ converges in the Gromov-Hausdorff sense to a spherical suspension where the diameter and the first eigenvalue of the $(S^5,g_i)$ are arbitrarily close to $\pi$ and $5$ respectively, and supposedly, since the singularities forming along the sequence are conical, we should have: for $1\leqslant p<\frac{5}{2}$, $\|W\|_{L^p}\to 0$. See Question \ref{ques: gap higher dime} for a reasonable higher-dimensional generalization.
\end{remark}

\begin{remark}\label{rem: counterex bounded Ricci}
    Imposing loose bounds on $\Ric$ is not sufficient to obtain (almost) rigidity: by \cite[Remark 51]{ozu2}, there exist smooth metrics $(\mathbb{S}^2\times \mathbb{S}^2,g_t)_{t\in(0,1)}$ with $3g_t\leqslant\Ric(g_t)\leqslant 6.1g_t$ converging to a spherical orbifold.
    The metrics $g_t$ satisfy $\diam(\mathbb{S}^2\times \mathbb{S}^2,g_t) \to \pi$, one can show that the first eigenvalue tends to $4$, and for any $p\in [1,+\infty) $ the $L^p$-norm of the Weyl curvature converges to zero as $t\to 0$. 
\end{remark}

\begin{remark}
    The constants in our gap theorems of Corollary \ref{cor:gap thms} are not explicit since they are proven by contradiction. Gap theorems with explicit constant, such as \cite{gur00} are exceptionally rare; see \cite{aes}. 
\end{remark}

\subsubsection{Tangent cones of Einstein $5$-manifolds}

Our main new tool to study the regularity of noncollapsed limits of Einstein $5$-manifolds is the following result. 

\begin{theorem}[Isolation of $1$-symmetric Ricci-flat cones]\label{thm:isolation 1sym}
For all $v>0$, there exists $\varepsilon_0 = \varepsilon_0(v)>0$ such that the following holds.

Let $(C(Y),d_{C(Y)},y)$ and $(C(Z),d_{C(Z)},z)$ be two metric cones that are limits of sequences of pointed 5-manifolds $(M_i,g_i,p_i)$ and $(N_i,\tilde{g}_i,q_i)$ both satisfying $|\Ric_i| \to 0$ and that the balls centered at $p_i, q_i$ of radius $1$ have volume bounded below by $v>0$. Suppose that $(C(Z),d_{C(Z)},z)$ isometrically splits off a line and 
\begin{equation*}
    d_{GH}(Y,Z) < \varepsilon_0.
\end{equation*}
Then the cones are isometric: $(C(Y),d_{C(Y)},y) \cong (C(Z),d_{C(Z)},z) \cong (\RR \times C(\mathbb{S}^3/\Gamma), g_0, 0)$ for $\Gamma \subset O(4)$ finite.

\end{theorem}
\begin{remark}
    Once again, the higher-dimensional analogous statement cannot hold by considering cones over the examples of \cite{boh}. 
\end{remark}

In particular, if these two cones arise as tangent cones at some point, this quantitative isolation directly implies the uniqueness of tangent cones splitting off a line.

\begin{theorem}[Uniqueness of tangent cones]\label{t:local uniqueness}
    Let $(M_i,g_i, p_i)_{i\in\mathbb{N}}$ be a sequence of pointed $5$-manifolds with $|\Ric|< 4$ and $\Vol(B_1(p_i)) > v >0$ $d_{GH}$-converging to $(X,d,p)$. Assume that $(\mathbb{R}\times C(Y),d_{\mathbb{R}\times C(Y)})$ is a tangent cone at $x\in X$, then it is the \emph{unique} tangent cone at $x$. In particular, except for countably many points, the tangent cone is uniquely isometric to $\RR \times C(\mathbb{S}^3/\Gamma)$ where $\Gamma \subset O(4)$ acts freely. 
\end{theorem}

%\begin{remark} Note that this uniqueness does not rely on \L{}ojasiewicz--Simon inequalities as in \cite{cm14} or algebraic arguments as in \cite{DS}.\end{remark}

Moreover, we obtain quantitative regularity on a wedge region $\cW(x)$ around any point $x \in X$ whose tangent cone splits off a line.  
The wedge region, defined in \eqref{e:wedge region}, consists of a neighborhood of $x$ with a scaling-invariant tubular neighborhood of the splitting axis removed.

\begin{corollary}\label{c:local curvature est}
    If one tangent cone at $x \in X$ splits, then there exists some wedge region $\cW(x)$ around $x$ and constant $\rho>0$ such that the harmonic radius (see for instance \cite{and90} for a definition), $r_h(y)$, at the point $y \in \cW(x)$ has a lower bound $r_h(y) \ge \rho \cdot  d(x,y)$. If we further assume that the sequence $M_i$ is \emph{Einstein}, then we have the curvature estimate $|\Rm|(y) =o \big( d(x,y)^{-2} \big)$ in $\cW(x)$. 
\end{corollary}

We next describe the asymptotic geometry of complete, noncollapsed, asymptotically Ricci-flat $5$-manifolds with Euclidean volume growth.  
When a tangent cone at infinity splits off a line, then it is \emph{unique}.

\begin{theorem}\label{t:uniqueness at infinity}
    Let $(M,g,p)$ be a 5-manifold with  Euclidean volume growth and satisfy $\sup_{B_{2R}(p) \setminus B_R(p)}|\Ric| = o(R^{-2})$. If one of the tangent cones at infinity splits, then this is the \emph{unique} tangent cone at infinity.
\end{theorem}

Finally, we obtain curvature control near infinity.

\begin{corollary}\label{c:infinity est}
    Let $(M,g,p)$ be a 5-manifold with  Euclidean volume growth and satisfy $\sup_{B_{2R}(p) \setminus B_R(p)}|\Ric| = o(R^{-2})$. If one of the tangent cones at infinity splits, then there exists some wedge region around infinity $\cW(\infty)$ and constant $\rho >0$ such that the harmonic radius at the point $y \in \cW(\infty)$ has a lower bound $r_h(y) \ge \rho \cdot  d(p,y)$. If we further assume that the manifold $M$ is \emph{Ricci-flat}, then we have the curvature estimate $|\Rm|(y) = o\big( d(p,y)^{-2} \big)$ in $\cW(\infty)$. 
\end{corollary}

\subsubsection{Singular set of noncollapsed limits of Einstein $5$-manifolds, and orbifold structure}

We refine the structure of the singular set $\cS$, and the regularity of the Gromov-Hausdorff limit space $X$. By the uniqueness of $1$-symmetric tangent cones \ref{t:local uniqueness}, we can further stratify the singular set $\cS(X)$ as follows: 
\begin{theorem}\label{t:new stratification}
    There exists some $N_0 = N_0(v)$ depending on the noncollapsing constant $v$ such that
\begin{equation}
    \cS(X) = \cS^0 \sqcup \bigsqcup_{2 \leqslant |\Gamma| \leqslant N_0} \cS^1_{\Gamma}, 
\end{equation}
where $\cS_{\Gamma}^1$ is the set of points whose unique tangent cone is $\RR \times \RR^4/\Gamma$, $\Gamma\subset O(4)$ and $\cS^0$ is a set of countably many points where no tangent cone splits. 
\end{theorem}

Our first result shows that locally each $\cS^1_{\Gamma}$ is contained in the bi-Lipschitz image of some subset of $\RR$. In particular, this refines the classical rectifiability statement of \cite{cjn,jn}, eliminating the need to discard an \emph{$\mathcal{H}^{1}$-measure zero} subset:

\begin{theorem}[Singular set contained in Lipschitz manifolds]\label{thm:lips curve sing}
    Let $(M_i,g_i, p_i)_{i\in\mathbb{N}}$ be a sequence of pointed $5$-manifolds with $|\Ric|< 4$ and $\Vol(B_1(p_i)) > v >0$. Let $(X,d,p)$ be a Gromov-Hausdorff limit of the sequence. For any $x\in \cS^1_{\Gamma}$, there exists some $r>0$ and some function $\phi$ such that $\phi:B_{r}(x) \cap \cS^1_{\Gamma} \to \RR$ is bi-Lipschitz. In particular, the singular set $\cS$ satisfies: There exist countably many points $q_j$ and measurable subsets $Z_i \subset \cS^1_{\Gamma_i}$  with bi-Lipschitz maps $\phi_i: Z_i \to \RR$ such that 
    \begin{equation}
        \cS = \bigcup_{i=1}^{\infty} Z_i \,\,\cup\,\, \bigcup_{j=0}^{\infty} \{q_j\}. 
    \end{equation}
\end{theorem}

{
\begin{remark}
    This resolves, in dimension $5$ and under a two-sided Ricci bound, the analogue of \cite[Conjecture 2.15]{nab20}. The original conjecture concerns lower Ricci bounds.
\end{remark}
}

%{\color{red}

%From our last section, we also probably have the set of points whose tangent cones have at most orbifold singularities $\mathbb{R}^4/\Gamma$ with $|\Gamma|\leqslant k$ is also convex? 

%}

The points $\{q_j\}$ in fact coincide with the set of $0$-symmetric singularities $\cS^0$. Theorem \ref{thm:lips curve sing} says that all of the $1$-symmetric singular points lie on Lipschitz curves. It turns out that we understand the singularities completely when they exhaust such a curve in $\cS^1_{\Gamma}$: they are orbifold singularities along geodesics and the curvature stays bounded. In other words: the curve of singularities in $\cS^1_{\Gamma}$ is \emph{removable}. More generally, given \emph{any} curve inside the singular set $\cS$, we have the following structure theorem:

\begin{theorem}[Structure of curves of singularities]\label{mthm: orbifold regularity} 
Let $(M_i,g_i,p_i)$ be a sequence of pointed \emph{Einstein} $5$-manifolds with $|\Ric|<4$ and $\Vol(B_1(p_i))>v>0$. Let $(X,d,p)$ be a Gromov-Hausdorff limit of the sequence. Assume that there exists a unit-speed curve $\gamma:I = (-\varepsilon,\varepsilon)\to X$ contained in the singular set, $\gamma(I)\subset\mathcal{S}(X)$, for some $\varepsilon>0$. Then, there exists a relatively closed \emph{nowhere dense} subset $\cI \subset I$ such that $I = \cI \sqcup \bigsqcup_{j =1}^{\infty} (s_j, s_{j+1})$ satisfies:
\begin{itemize}
    \item on each $(s_j,s_{j+1})$, there exists a finite group $\Gamma_j\subset O(4)$ such that at all of the points in $\gamma(s_j,s_{j+1})$, the tangent cone  is unique and isometric to $\mathbb{R}\times \mathbb{R}^4/\Gamma_j$,
    \item there exist an open sets $U_j\subset X$ with $\gamma\big((s_j,s_{j+1})\big)\subset U_j$ such that on $U_j\setminus\gamma$ the Einstein metric extends \emph{real-analytically} across $\gamma$ as a real-analytic $5$-orbifold metric with singularities $\mathbb{R}\times \mathbb{R}^4/\Gamma_j$ along $\gamma$. In particular, the curvature stays bounded.
    \item $\gamma_{|(s_j,s_{j+1})}$ is a unit-speed geodesic of $X$. 
\end{itemize}
If $\gamma(I)\subset\mathcal{S}^1_\Gamma$, then $\cI=\emptyset$, $\gamma$ is a geodesic and the metric extends as a real-analytic orbifold across it.
\end{theorem}

\begin{corollary}
    Each connected component $\Omega$ of $\cS^1_{\Gamma}$ is either a single point or a geodesic, and thus \emph{locally convex}. That is, for every $x\in \Omega$ there exists $r>0$ such that for any $y\in \Omega\cap B_{r}(x)$, every minimizing geodesic $\gamma$ joining $x$ and $y$ is contained in $\Omega$.
\end{corollary}

\begin{remark}
    In general, the subset $\cI \subset I$ in Theorem \ref{mthm: orbifold regularity} is non-empty and one cannot expect convexity of the entire singular set $\cS$. For example, let $C(Z)$ be a Ricci-flat cone that does not split, where $Z$ is an Einstein $4$-orbifold with at least two singular points. Then the singular set is not locally convex around the vertex. 
\end{remark}

Theorem \ref{mthm: orbifold regularity} characterizes any curve of singularities up to a nowhere dense subset: they are open geodesic segments. We now turn to the converse and completely characterize neighborhoods of open geodesics of our limit spaces.

By establishing the sharp H\"older continuity of tangent cones along geodesics in a seminal work \cite{con12}, Colding-Naber, and later generalized by Deng \cite{deng25}, classified the geodesics (meaning locally minimizing, in the metric sense) with one regular point: the whole interior part is regular. Combining Theorems \ref{t:local uniqueness}, \ref{mthm: orbifold regularity} and their estimates, we can now fully classify the geodesics by one of their interior points: 
\begin{theorem}[Classification of geodesics]\label{t:limit geodesic} 
    Let $\sigma:(-\varepsilon, \varepsilon)\to X$ be a geodesic in a noncollapsed Einstein $5$-dimensional limit space $(X,d)$ with $|\Ric_{M_i}| \leqslant 4$. Then either $\sigma\big((-\varepsilon, \varepsilon) \big)$ is regular, or the entire $\sigma\big((-\varepsilon, \varepsilon)\big)\subset\cS^1_{\Gamma}$ for some $\Gamma \subset O(4)$ and the interior of the curve is removable in the sense of Theorem \ref{mthm: orbifold regularity}.  
\end{theorem}

This implies that the entire interior of a geodesic has isometric tangent cones $\RR \times \RR^4/\Gamma$ for some fixed $\Gamma \subset O(4)$ (with $\Gamma=\{id\}$ corresponding to regular points). In particular, $\cS^0$ cannot lie on the interior of any geodesic. This yields a dimension-five Einstein-manifold version of the uniqueness of tangent cones along geodesics in the spirit of the conjecture of Burago proved by Petrunin \cite{bgp92, petrunin98} for Alexandrov spaces. 

Finally we discuss the orbifold regularity of the limit space $X$. Theorem \ref{mthm: orbifold regularity} says the ``interior" of each $\cS^1_{\Gamma}$ is locally a geodesic along which the metric extends real-analytically. Then it is natural to introduce the \emph{orbifold-regular-singular} decomposition as follows: 
\begin{equation}\label{e:orbifold-reg-sing}
    \mathcal{R}_{\mathrm{orb}} = \mathcal{R}\,\sqcup\bigsqcup_{{2\leqslant|\Gamma|\leqslant N_0}}\mathring{\mathcal{S}}^1_{\Gamma} \quad \text{ and } \quad \cS_{\mathrm{orb}} = X \setminus \cR_{\mathrm{orb}},
\end{equation}
where $\cR$ is the regular set, and $\mathring{\mathcal{S}}^1_{\Gamma} \equiv \{ x\in \cS^1_{\Gamma}~|~ $there exists some curve $\sigma:(-\varepsilon, \varepsilon) \to \cS^1_{\Gamma}$ with $\sigma(0) = x \}.$ 

Combining the refined rectifiability Theorem \ref{thm:lips curve sing}, the orbifold structure Theorem \ref{mthm: orbifold regularity} and the classification of the geodesics Theorem \ref{t:limit geodesic}, we can prove the following partial orbifold-regularity of $X$:
\begin{corollary}
The orbifold-regular subset $\cR_{\mathrm{orb}}$ and the orbifold-singular subset $\cS_{\mathrm{orb}}$ defined in \eqref{e:orbifold-reg-sing} satisfy the following:
\begin{itemize}
    \item \emph{(Orbifold structure) } $\cR_{\mathrm{orb}}$ is a real-analytic Einstein $5$-orbifold. 
    \item \emph{(Convexity) } $\cR_{\mathrm{orb}}$ is convex: any geodesic with endpoints in $\mathcal{R}_{\mathrm{orb}}$ is contained in $\mathcal{R}_{\mathrm{orb}}$.
    \item \emph{(Extreme points) } $\cS_{\mathrm{orb}}$ is exactly the set of \emph{extreme points of $(X,d)$}, i.e. the set of points that do not lie on a geodesic between two other points. In particular, the curvature of $(X,d)$ is bounded around every non-extreme point (i.e. points in $X\setminus\cS_{\mathrm{orb}}=\mathcal{R}_{\mathrm{orb}}$).
    \item \emph{(Nowhere denseness) } There exist countably many points $\{q_j \}$, measurable subsets $Z_i \subset X$ and bi-Lipschitz maps $\phi_i:Z_i \to \RR$ with $\phi_i(Z_i)$ \emph{nowhere dense} in $\RR$, such that $\cS_{\mathrm{orb}} = \bigcup_i Z_i \cup \bigcup_j\{q_j\}$. 
\end{itemize}
\end{corollary}

In particular, we show that the orbifold-regular set $\cR_{\mathrm{orb}}$ is homeomorphic to a real-analytic variety, thus confirming a conjecture of Naber \cite[Conjecture 6.4]{nab14} away from extreme points in dimension $5$.

\subsubsection{Singular set of higher dimensional noncollapsed limits}

In higher dimensions we can
\emph{distinguish} the $(n-4)$-stratum from the $(n-5)$-stratum, and this separation already
forces a uniqueness phenomenon.

\begin{theorem}\label{thm:higher dimension unique}
Let $(M_i,g_i,p_i)_{i\in\mathbb{N}}$ be a sequence of pointed $n$-manifolds with
$| \Ric_{g_i}| \leqslant n-1$ and $\Vol(B_1(p_i))\geqslant v>0$, and let $(X,d,p)$ be a
Gromov-Hausdorff limit. Fix $x\in X$. Assume that every tangent cone at $x$ splits off an $\mathbb{R}^{n-5}$-factor and that at least one
tangent cone at $x$ splits off $\mathbb{R}^{n-4}$. Then the tangent cone at $x$ is unique isometric to $\mathbb{R}^{n-4}\times \mathbb{R}^4/\Gamma$ for some $\Gamma\subset O(4)$.
\end{theorem}

\begin{cor}

    Let $(X,d)$ be a noncollapsed limit of  $6$-manifolds with $|\Ric|\leqslant 5$. Let $p = \gamma(0)$ be a point lying on an open unit-speed geodesic $\gamma: (-\varepsilon,\varepsilon)\to X$ for some $\varepsilon>0$. Assume that \emph{one} tangent cone at $p$ splits off $\mathbb{R}^{n-4}$. Then, there exists $\Gamma\subset O(4)$ such that for any $s\in(-\varepsilon,\varepsilon)$, the tangent cone at $\gamma(s)$ is unique and isometric to $\mathbb{R}^2\times\mathbb{R}^4/\Gamma$. 
\end{cor}

\subsection{Questions and conjectures}\label{subs:conjectures}

Motivated by the new $5$-dimensional results of the present articles, we propose a program towards an orbifold regularity theory off a \textit{codimension $5$} set. This is analogous to the energy minimizing  harmonic maps: while in this context, the maps are \textit{controllably} smooth away from a codimension $2$ set, it was proven in \cite{su82} that the limits are in fact smooth away from a codimension $3$ set.

\subsubsection{Dimension 5}

 Einstein metrics being real-analytic, one might conjecture by unique continuation that $\mathcal{S}_{\mathrm{orb}}$ cannot have accumulation points. 

\begin{question}
    Are there noncollapsed limits of Einstein $5$-manifolds with isolated singular points in $\mathcal{S}^1$? If singular points in $\mathcal{S}^1$ accumulate at a point, do they all lie on a common singular curve?
\end{question}

\begin{remark}
Motivated by the Calabi-Yau singular models \cite{hn,sze19}, it is reasonable to expect that singularities of the form $\mathbb{R}^{2k}\times\mathbb{R}^4/\Gamma$ in Einstein limits can be isolated; the same may plausibly hold for $\mathbb{R}\times\mathbb{R}^4/\Gamma$. The question is whether they can accumulate without forming curves. See also \cite{sze21} in the context of minimal surfaces.
\end{remark}

Ruling out the accumulation of singular points, possibly by using the real-analyticity of the Einstein equation, that do not lie on curves of $\mathcal{S}$ would essentially prove the following conjecture.

\begin{conjecture}[Codimension-$5$ conjecture in dimension $5$]\label{conj:codim5 dim 5}
    Let $(X,d)$ be a Gromov-Hausdorff limit of noncollapsed Einstein $5$-manifolds. Then, the set $\mathcal{S}_{\mathrm{orb}}=X\setminus\mathcal{R}_{\mathrm{orb}}$ is countable. In particular, away from a countable set of points, $(X,d)$ has the structure of a smooth Einstein $5$-orbifold with singularities of the form $\mathbb{R}\times \mathbb{R}^4/\Gamma$ along countably many \emph{geodesics}. 
\end{conjecture}

\begin{remark}
     In dimension $4$, while the control of the degeneration cannot be better than codimension $4$, the limit spaces are \textit{real-analytic orbifolds} without singularities up to taking local finite covers. For $6$-dimensional Kähler-Einstein metrics, Section 5 of \cite{DS} proves that the link of any tangent cone is a 5D Sasaki-Einstein orbifold, whose singular set is finitely many periodic Reeb orbits, and the local quotient by the Reeb action is a Kähler-Einstein 4-orbifold.
\end{remark}

Motivated by the Calabi-Yau models \cite{li19,cr23,sze19,sze20}, with tangent cones of the form $\mathbb{R}^{2k}\times\mathbb{R}^4/\Gamma$, a natural question is the following one.

\begin{question}
    Are there complete $5$-dimensional Ricci-flat metrics whose asymptotic cones split a line at infinity while the interior metrics do not split a line?
\end{question}

\subsubsection{Higher dimensions}

In higher dimensions, we have less evidence, but based on our $5$-dimensional results, we ask natural questions of isolation and gap theorems. 

\begin{question}[Higher-dimensional isolation of spherical orbifolds and gap theorem]
    Given $(M_o^n,g_o)$ a spherical orbifold with codimension $4$ singularities. Is there a sequence of smooth Einstein $n$-metrics converging to $(M_o^n,g_o)$ in the Gromov-Hausdorff sense? 
\end{question}

\begin{question}\label{ques: gap higher dime}
    For every $n \ge 4, v>0$, does there exist $\varepsilon_0 = \varepsilon_0(n,v)>0$ such that if $(M_o^n,g_o)$ with $\Vol(g_o)>v$ and $\Ric(g_o) = (n-1)g_o$ satisfies $\lambda_{n-3}<n+\varepsilon_0$ with $\lambda_{k}$ the $k^{th}$ eigenvalue of the Laplace-Beltrami operator?
\end{question}

\begin{exmp}
    The simplest example of such an orbifold is $(M_o^5,g_o)$ the quotient of $\mathbb{S}^5\subset \mathbb{R}^4\times \mathbb{R}^2$ by the involution $(-\operatorname{Id},\operatorname{Id})$ with a circle of $\mathbb{R}\times\mathbb{R}^4/\mathbb{Z}_2$ singularities. The orbifold $(M_o^5,g_o)$ is a $5$-dimensional example, but it is also $\mathbb{S}^1$-invariant. 
\end{exmp}

\begin{remark}
    A natural tentative would be to glue in a rescaled cylinder $\mathbb{S}^1$ times the Eguchi-Hanson metric and perturb in an $\mathbb{S}^1$-invariant way. This is however impossible by our isolation theorem \ref{thm:isolation orbifolds} since the $\mathbb{S}^1$-invariance reduces the problem to that of desingularizing a spherical $4$-orbifold by Einstein $4$-orbifolds.
\end{remark}

\begin{question}[Codimension-$5$ conjecture]\label{ques:codim 5}
Let $(X,d)$ be a Gromov-Hausdorff limit of noncollapsed Einstein $n$-manifolds. Then there exists a closed set $\mathcal{S}_{\mathrm{orb}}\subset X$ of Hausdorff codimension at least $5$ such that $\mathcal{R}_{\mathrm{orb}}:=X\setminus\mathcal{S}_{\mathrm{orb}}$ is a smooth Einstein $n$-orbifold. Moreover, the singular set in $\mathcal{R}_{\mathrm{orb}}$ is a countable union of totally geodesic codimension $4$ submanifolds, and each point on this set admits a neighborhood modeled on
$\mathbb{R}^{n-4}\times \mathbb{R}^4/\Gamma$ for some finite $\Gamma\subset O(4)$.
\end{question}

\begin{remark}
The full singular locus of $X$ is still expected to have (optimal) codimension $4$, but outside the codimension-$5$ “bad set” $\mathcal{S}_{\mathrm{orb}}$ all singularities are \emph{removable} in the sense that the metric extends as a real-analytic Einstein orbifold metric with bounded curvature across the codimension-$4$ strata. See, the analytic removability and regularity mechanisms in \cite{su82,bkn,and}.
\end{remark}

\emph{Towards a proof in all dimensions.}
The present paper essentially proves the iteration from $n-1=4$ to $n=5$ of a plausible, long-term strategy in order to answer Question \ref{ques:codim 5} affirmatively by an iteration on the dimension. Assume that noncollapsed limits of $(n-1)$-dimensional Einstein metrics have codimension $5$ regularity in the sense of Question \ref{ques:codim 5}.
\begin{enumerate} 
    \item \emph{Isolation of higher-dimensional spherical orbifolds.} Show that $(n-1)$-dimensional spherical suspensions over $\mathbb{S}^3/\Gamma$ are quantitatively GH-isolated among noncollapsed limits of $(n-1)$-dimensional Einstein metrics with $\Ric = (n-2)$. 
  \item \emph{Isolation of $(n - 4)$-symmetric cones.} Deduce that cones whose links are spherical suspensions over $\mathbb{S}^3/\Gamma$ (equivalently, local models $\mathbb{R}^{n-4}\times \mathbb{R}^4/\Gamma$) are quantitatively isolated among Ricci-limit cones.
  \item \emph{Singular set contained in Lipschitz submanifolds.} Show that $\mathcal{S}^{n-4}$ is contained in a union of Lipschitz codimension $4$ submanifolds. 
  \item \emph{Removable codimension $4$ singularities.} Show that subsets of $\mathcal{S}^{n-4}$ that exhaust $(n-4)$-dimensional subsets of the Lipschitz maps must be removable.
  \item \emph{Codimension $5$ singularities.} Show that the remaining singularities in $\mathcal{S}^{n-4}$ form a subset of Hausdorff dimension $(n-5)$.
\end{enumerate}

\subsection{Strategy and Organization}

The article is divided into two parts, the first one concerns gap theorems for Einstein orbifolds, and the second the regularity of noncollapsed Einstein $5$-manifolds.

\subsubsection{Isolation theorem for Einstein $4$-manifolds}

We first outline the proof of the isolation theorem for spherical and hyperbolic Einstein $4$-orbifolds, Theorem \ref{thm:isolation orbifolds}.

\noindent\textbf{Setup and refined approximation.}
Assume towards a contradiction that there exists a sequence of Einstein manifolds $(M_i,g_i)_i$ converging in the Gromov-Hausdorff (GH) sense to a spherical or hyperbolic Einstein $4$-orbifold $(M_o,g_o)$. By the GH-to-smooth theory of \cite{ozu1,ozu2}, each $g_i$ is smoothly approximated by a naïve gluing metric $g_i^D$ capturing all bubble trees of ALE Ricci-flat pieces. A key new input is a \emph{refined} obstruction space adapted to these trees. We first glue and perturb every bubble tree to a \emph{Ricci-flat modulo obstructions} configuration, and then upgrade $g_i^D$ to an \emph{Einstein modulo obstructions} approximation $g_i^A$ with an explicit refined obstruction $w_i^A$ supported on the bubble region.

\noindent\textbf{Quantitative control.}
In appropriate, divergence-free gauges, the linearization of the Einstein (mod. obstructions) map is invertible between specific weighted Hölder spaces, hence, by inverse function theorem,
\begin{equation}\label{eq:strat inv fct thm}
  \|g_i^A-g_i\|+\|w_i^A\|\;\lesssim\;\|\Ric(g_i^A)-\Lambda g_i^A + w_i^A\|,
\end{equation}
where $\Lambda\in\{\pm3\}$ is the Einstein constant of $(M_o,g_o)$. Thus, to reach a contradiction, it suffices to prove that $\|w_i^A\|$ cannot be dominated by the right-hand side.

\noindent\textbf{Decomposition of the obstruction.}
A precise analysis of the refined obstruction shows a canonical splitting
  $$w_i^A = u_i^B + u_i^o,$$
where $u_i^B$ is the obstruction to perturbing the bubble trees toward Ricci-flat metrics, and $u_i^o$ is the obstruction created when matching the trees to the core orbifold $(M_o,g_o)$. When $(M_o,g_o)$ is spherical or hyperbolic, then $u_i^o$ is nonvanishing. The main conceptual step of the proof then consists in proving that if $u_i^B$ compensates $u_i^o$, then it contains a much larger piece orthogonal to $u_i^o$. In either case this contradicts \eqref{eq:strat inv fct thm}:
  $$\|w_i^A\|\  \gg\ \|\Ric(g_i^A)-\Lambda g_i^A + w_i^A\|.$$ 

\noindent\textbf{Conclusion.} No Einstein orbifold can be GH-close to $(M_o,g_o)$ unless it is isometric to it.

\subsubsection{Regularity of Einstein $5$-manifolds}

We then focus on the regularity of Einstein $5$-manifolds. 

\noindent\textbf{Uniqueness of tangent cones.}
In the 5-dimensional limit space, we first establish an isolation theorem for $1$-symmetric cones, building on the gap theorems for Einstein $4$-orbifolds.  
This, in particular, yields the uniqueness of $1$-symmetric tangent cones.  
Combined with $\varepsilon$-regularity, we then obtain quantitative curvature estimates on wedge regions around such points, and analogous results hold for tangent cones at infinity and their asymptotic curvature behavior. These results are proved in Section \ref{subs:uniqueness}.

\noindent \textbf{Refined rectifiability.} 
The isolation theorem implies that $1$-symmetry at a given scale at some point $x\in \cS^1_{\Gamma}$ propagates to \emph{uniform symmetry} across all smaller scales.  
Combined with the neck decomposition developed in \cite{jn,cjn}, this uniform symmetry allows one to identify the local $1$-stratum $S^1_{\Gamma}$ near $x$ with the center $\cC_0$ of an associated neck region. A bi-Lipschitz coordinate is then constructed by using a splitting function to order points along 
$\cC_0$ and pushing forward the Ahlfors regular measure associated to the neck region onto a subset of $\RR$ via the ordering.   
A countable selection of such parametrizations yields a covering of $\cS^1$ by bi-Lipschitz images of subsets of $\mathbb{R}$, thereby establishing a \emph{refined rectifiability} of the singular set without discarding any exceptional subsets.  
Moreover, we show that along each curve with constant volume density, the tangent cones remain identical. The proof is given in Section \ref{subs:rectifiable}.

\noindent \textbf{Orbifold structure.}
Along a Lipschitz curve of singularity type $\RR \times \RR^4/\Gamma$, GH-close model charts to $\mathbb{R}\times \mathbb{R}^4/\Gamma$ can be \emph{glued horizontally} along the direction parallel to the singular curve. Controlling the gluing error uniformly is subtle: it requires compatibility of the local $\Gamma$-actions on adjacent charts. Crucially, the transition maps are constrained to be approximated only by the reflection or translation along the singular axis $\RR$ and rotation preserving the $\Gamma$-action on $\RR^4/\Gamma$ -- a property that fails in flat $\RR^5$. These glued charts then yield $C^0$ orbifold-type coordinates on punctured neighborhoods. 

To control the curvature in these necks, we apply Hardy-type inequalities and a Moser iteration for subsolutions of the following equation
\begin{equation}\label{eq: strat curvature eq}
    \Delta u \gtrsim -|y|^{-2}u
\end{equation}
with $|y|$ the distance to the singular curve. Equation \eqref{eq: strat curvature eq} is satisfied by $u=|\Rm|^{1-\alpha}$ for some $\alpha>0$ by 1) $\varepsilon$-regularity and our isolation theorem for $\mathbb{R}\times\mathbb{R}^4/\Gamma$ cones, 2) the differential equation satisfied by curvature on Einstein metrics, 3) an improved Kato inequality. This gives uniform \emph{pointwise} curvature bounds. Once curvature is uniformly bounded, we upgrade the coordinate regularity to $C^{1,\alpha}$, and then to analyticity via the Einstein equations.  
This establishes the analytic orbifold structure along the singular curves. Thanks to these coordinates, the singular curve, which is the fixed-point set of the $\Gamma$-action, is a geodesic. The details are presented in Section \ref{s:orbifold structure}.

\subsection{Acknowledgements}

The first author is grateful to Tobias Colding for suggesting the uniqueness problem and for his patient guidance and many insightful discussions. The first author also thanks Wenshuai Jiang for his support and helpful conversations. The authors would like to thank Max Hallgren and Wenshuai Jiang for very valuable comments that helped us improve a step of the proof of Theorem \ref{thm:lips curve sing} in Section \ref{subs:rectifiable}.

During this project, Yiqi Huang was partially supported by a Simons Dissertation Fellowship and by NSF Grant DMS--2405393, and Tristan Ozuch was partially supported by the National Science Foundation under grant DMS--2405328.

\part{Isolation and gap theorems for Einstein $4$-orbifolds}

\section{Preliminaries on the degeneration of Einstein $4$-manifolds}

\subsection{Einstein orbifolds and their infinitesimal deformations}

For $\Gamma$ a finite subgroup of $O(4)$ acting freely on $\mathbb{S}^3$, let us denote $(\mathbb{R}^4\slash\Gamma,e)$ the flat cone obtained by the quotient by the action of $\Gamma$, and $r_e:= d_e(.,0)$.

\begin{defn}[$4$-orbifold (with isolated singularities)]\label{orb Ein}
    We will say that a metric space $(M_o,g_o)$ is an orbifold of dimension $4$ if there exists $\varepsilon_0>0$ and a finite number of points $(p_j)_j$ of $M_o$ called \emph{singular} such that we have the following properties:
    \begin{enumerate}
        \item the space $(M_o\backslash\{p_j\}_j,g_o)$ is a Riemannian manifold of dimension $4$,
        \item for each singular point $p_j$ of $M_o$, there exists a neighborhood of $p_j$, $ U_j\subset M_o$, a finite subgroup acting freely on $\mathbb{S}^3$, $\Gamma_j\subset O(4)$, and a diffeomorphism $ \Phi_j: B_e(0,\varepsilon_0)\subset\mathbb{R}^4\slash\Gamma_j \to U_j\subset M_o $ for which, the pull-back of $\Phi_j^*g_o$  on the covering $\mathbb{R}^4$ is smooth.
    \end{enumerate}
\end{defn}

\begin{defn}[The function $r_o$ on an orbifold]\label{ro}
    Given $S_o\subset M_o$ a subset of the singularities, we define $r_o$, a function on $M_o$ { smooth away from $S_o$} satisfying $r_o:= (\Phi_j)_* r_e$ on each $U_j$ with $U_j\cap S_o\neq \emptyset$, and such that on $M_o\backslash U_j$, we have $\varepsilon_0\leqslant r_o<1$ (the different choices will be equivalent for our applications). We will denote, for $0<\varepsilon\leqslant\varepsilon_0$, $$M_o(\varepsilon):= \{r_o>\varepsilon\} = M_o\backslash  \Big(\bigcup_j \Phi_j\Big(\overline{B_e(0,\varepsilon)}\Big) \Big).$$
\end{defn}
We also define $\mathbf{O}(g_o)$ the space of traceless and divergence-free infinitesimal Einstein deformations of $g_o$. 

\subsection{Ricci-flat ALE spaces and their deformations}

Let us now turn to ALE Ricci-flat metrics.

\begin{defn}[ALE orbifold (with isolated singularities)]\label{def orb ale}
    An ALE orbifold of dimension $4$, $(N,g_{b})$ is a metric space for which there exists $\varepsilon_0>0$, singular points $(p_j)_j$ and a compact $K\subset N$ for which we have:
    \begin{enumerate}
        \item $(N,g_b)$ is an orbifold of dimension $4$,
        \item there exists a compact subset $K\subset N$ and a diffeomorphism $\Psi_\infty: (\mathbb{R}^4\slash\Gamma_\infty)\backslash \overline{B_e(0,\varepsilon_0^{-1})} \to N\backslash K$, for $\Gamma_\infty\subset O(4)$ such that $$r_e^l|\nabla^l(\Psi_\infty^* g_b - {e})|_{{e}}\leqslant C_l r_e^{-4}.$$
    \end{enumerate}
\end{defn}

\begin{defn}[The function $r_{b}$ on an ALE orbifold]\label{rb}
Given $S_b\subset N$ a subset of the singularities, we define $r_b$, a function on $N$ { smooth away from $S_b$} satisfying $r_b:= (\Psi_j)_* r_e$ on each $U_j$ with $U_j\cap S_b\neq \emptyset$, and $r_{b}:= (\Psi_\infty)_* r_e$ on $U_\infty = N\backslash K$, and such that $\varepsilon_0\leqslant  r_{b}\leqslant \varepsilon_0^{-1}$ on the rest of $N$ (the different choices are equivalent for our applications).

For $0<\varepsilon\leqslant\varepsilon_0$, we will denote $$N(\varepsilon):= \{\varepsilon<r_b<\varepsilon^{-1}\} = N\backslash  \Big(\bigcup_j \Psi_j\Big(\overline{B_e(0,\varepsilon)}\Big) \cup \Psi_\infty \Big((\mathbb{R}^4\slash\Gamma_\infty)\backslash B_e(0,\varepsilon^{-1})\Big)\Big).$$
\end{defn}

We define $\mathbf{O}(g_b)$ to be the set of infinitesimal Ricci-flat deformations which are transverse and traceless.

\subsection{Naïve desingularizations}\label{reecriture controle}

We recall the definition of a naïve desingularization of an orbifold, \cite{ozu1}.

\subsubsection{Gluing of ALE spaces to orbifold singularities}

    Let $0<2\varepsilon<\varepsilon_0$ be a fixed constant, $t>0$ a scale of gluing, $(M_o,g_o)$ an orbifold and $\Phi: B_e(0,\varepsilon_0)\subset\mathbb{R}^4\slash\Gamma \to U$ a local chart of Definition \ref{orb Ein} around a singular point $p\in M_o$. Let also $(N,g_b)$ be an ALE orbifold asymptotic to $\mathbb{R}^4\slash\Gamma$, and $\Psi_\infty: (\mathbb{R}^4\slash\Gamma)\backslash B_e(0,\varepsilon_0^{-1}) \to N\backslash K$ a chart at infinity of Definition \ref{def orb ale}.
    
    For $s>0$, define $\phi_s: x\in \mathbb{R}^4\slash\Gamma\to sx \in \mathbb{R}^4\slash\Gamma$. For $t<\varepsilon_0^4$, we define $M_o\#N$ as $N$ glued to $M_o$ thanks to the diffeomorphism $$ \Phi\circ\phi_{\sqrt{t}}\circ\Psi^{-1} : \Psi(A_e(\varepsilon_0^{-1},\varepsilon_0t^{-1/2}))\to \Phi(A_e(\varepsilon_0^{-1}\sqrt{t},\varepsilon_0)).$$ Consider moreover $\chi:\mathbb{R}^+\to\mathbb{R}^+$, a $C^\infty$ cut-off function supported on $[0,2]$ and equal to $1$ on $[0,1]$.
    
\begin{defn}[Naïve gluing of an ALE space to an orbifold]\label{def naive desing}
    We define a \emph{naïve gluing} of $(N,g_b)$ at scale $0<t<\varepsilon^4$ to $(M_o,g_o)$ at the singular point $p$, which we will denote $(M_o\#N,g_o\#_{p,t}g_b)$ by putting $g_o\#_{p,t}g_b=g_o$ on $M_o(\varepsilon)=M_o\backslash U$, $g_o\#_{p,t}g_b=tg_b$ on $N(\varepsilon)=K$, and on $\mathcal{A}(t,\varepsilon):=A_e(\varepsilon^{-1}\sqrt{t},2\varepsilon)$,
    \begin{equation}\label{naive gluing metrics}
        g_o\#_{p,t}g_b :=  t\chi(t^{-\frac14}r_e)\phi_{t^{-1/2}}^*(\Psi_\infty^*g_b) + \Big(1-\chi(t^{-\frac14}r_e)\Big)\Phi^*g_o.
    \end{equation}
    Given two tensors $h_o$ and $h_b$ on $M_o$ and $N$, we similarly define the \textit{naïve gluing}:
    \begin{equation}\label{naive gluing tensors}
        h_o\#_{p,t}h_b := t \chi(t^{-\frac14}r_e)\phi_{t^{-1/2}}^*\Psi_\infty^*h_b + \Big(1-\chi(t^{-\frac14}r_e)\Big)\Phi^*h_o
    \end{equation}
\end{defn}

In general, one has to desingularize by \textit{trees} of Ricci-flat ALE orbifolds. Consider $(M_o,g_o)$ an Einstein orbifold (the index $o$ stands for orbifold) and $S_o$ a \textit{subset} of its singular points, and $(N_j,g_{b_j})_j$ (the index $b_j$ stands for $j$-th bubble) a family of Ricci-flat ALE spaces asymptotic at infinity to $\mathbb{R}^4\slash\Gamma_j$ and $(S_{b_j})_j$ subsets of their singular points. Now, assume that there is a one-to-one map  $p:j\mapsto p_j\in S_o\cup \bigcup_k S_{b_k}$, where the singularity at $p_j$ is $\mathbb{R}^4\slash\Gamma_j$. We will call $D:= \big((M_o,g_o,S_o),(N_j,g_{b_j},S_{b_j})_j, p\big)$ a \emph{desingularization pattern}.

\begin{defn}[Naïve (partial) desingularization by a tree of singularities]\label{def total desing}
    Let $0<2\varepsilon<\varepsilon_0$, $D$ be a desingularization pattern for $(M_o,g_o)$, and let $0<t_j<\varepsilon^4$ be relative gluing scales. The metric $g^D_t$ is then the result of the following finite iteration: 
    \begin{enumerate}
        \item start with a deepest bubble $(N_j,g_{b_j})$, that is, $j$ such that $S_{b_j}= \emptyset$,
        \item if $p_j\in N_k$, replace $(N_k,g_{b_k},S_{b_k})$ and $(N_j,g_{b_j},\emptyset)$ in $D$ by $(N_k\#N_j,g_{b_k}\#_{p_j,t_j}g_{b_j},S_{b_k}\backslash\{p_j\})$ and restrict $p$ as $l\to p_l$ for $l\neq j$ in $D$ and consider another deepest bubble. The same works if $p_j\in M_o$.
        \item choose another deepest bubble and do the same.
    \end{enumerate}
    
    For $t = (t_j)_j$, if $N_j$ is glued to $p_j\in N_{j_1}$, and $N_{j_1}$ is glued to $p_{j_1}\in N_{j_2}$, ..., $N_{j_{k-1}}$ is glued to $p_{j_{k-1}}\in N_{j_k}$, which is glued to $M_o$, we define the absolute scale $T_j:= t_jt_{j_1}t_{j_2}...t_{j_k}$. This way, on each $N_j(\varepsilon)$,  $g^D_t|_{N_j(\varepsilon)}=T_jg_{b_j}$.
\end{defn}

\begin{remark}
    The above definition considers \textit{partial desingularizations} by considering \textit{subsets} $S_o$ and $S_{b_j}$ of all singularities. The resulting metric is only smooth if $S_o$ and the $S_{b_j}$ are \textit{all} of the singularities.
\end{remark}

\begin{defn}[Tree of singularities forming at a given singular point]\label{def:tree of sing at po}
    Let $p_1\in S_o$ be a singular point of $M_o$ (resp. $p_1\in S_{b_j}$), and consider $J_{p_1}$ the set of indices $j_0$ such that in the above desingularization pattern, there is a \textit{path} from $j_0$ to $p_1$, that is a set $j_1,\dots,j_{m}$ such that $p(j_k)\in S_{b_{j_{k+1}}}$ for $0\leqslant k\leqslant m$, and $p(j_m)=p_1$. 
    
    We call $\mathcal{B}_{p_1} = \left((N_j,g_{b_j},S_{b_j},t_j)_j,p_1\right)$ the \emph{tree of singularities} at $p_1$. It yields a topology $N_{J_{p_1}}$ given by the above gluing and admits an ALE metric $g^B_{t}$ equal to $t^{-1}_{p^{-1}(p_1)}g^D_{t}$ where both metrics are defined.

    We additionally call $(N_{j_m},g_{b_{j_m}})$ the core orbifold of the tree. It plays the role of $(M_o,g_o)$ in all of the previous constructions.
\end{defn}

 Let $(M,g^D)$ a naïve desingularization of an Einstein orbifold $(M_o,g_o)$ by a tree of ALE Ricci-flat orbifolds $(N_j,g_{b_j})_j$ with relative scales $t_j>0$. The manifold $M$ is covered as $M = M_o^t\cup \bigcup_jN_j^t$, where
\begin{equation}\label{def Mot}
    M_o^t: = M_o \backslash \Big(\bigcup_j \Phi_j(B_e(0,t_j^{\frac14})) \Big),
\end{equation}
 where $t_j>0$ is the relative gluing scale of $N_k$ at the singular point $p_k\in M_o$, and where 
 \begin{equation}\label{def Njt}
     N_j^t:= \Big(N_j\backslash \Psi_\infty \Big((\mathbb{R}^4\slash \Gamma_\infty)\backslash B_e(0, 2t_j^{-\frac14}) \Big)\Big) \backslash \Big(\bigcup_k \Psi_k(B_e(0,t_k^{\frac14}) \Big).
 \end{equation}
On $M_o^{16t}\subset M_o^t$, we have $g^D = g_o$ and on each $N_j^{16t}\subset N_j^t$, we have $g^D = T_j g_{b_j}$. By Definition \ref{def naive desing}, on the intersection $N_j^t\cap M_o^t$ we then have $\sqrt{T_j}r_{b_j} = r_o $, and on the intersection $N_j^t\cap N_k^t$, we have $\sqrt{T_j}r_{b_j} = \sqrt{T_k}r_{b_k}$.

\begin{defn}[Function $r_D$ on a naïve desingularization]\label{definition rD}
    On a naïve desingularization $(M,g^D)$, we define a smooth function $r_D$ in the following way:
    \begin{enumerate}
        \item $r_D = r_o$ on $M_o^t$,
        \item $r_D = \sqrt{T_j}r_{b_j}$ on each $N_j^t$.
    \end{enumerate}
\end{defn}

{
As in \cite[Definitions 1.11, 1.12]{ozu2}, we decompose our orbifold $M$ into particular regions and define associated cut-off functions.

\begin{defn}[{Neck regions, $\mathcal{A}_k(t,\varepsilon)$}]\label{def neck region}
    Let $(N_k,g_{b_k})$ be a Ricci-flat ALE orbifold of the above tree of singularities. We define $\mathcal{A}_k(t,\varepsilon)$ as the connected region with $\varepsilon^{-1}\sqrt{T_k}<\sqrt{T_k}r_{b_k}=r_D<\varepsilon t_k^{-\frac1{2}}\sqrt{T_k}$ with a nonempty intersection with $N_k^t$.
\end{defn}
\vspace{-0.4cm}
\begin{figure}[h]
\centering\includegraphics[width=7cm]{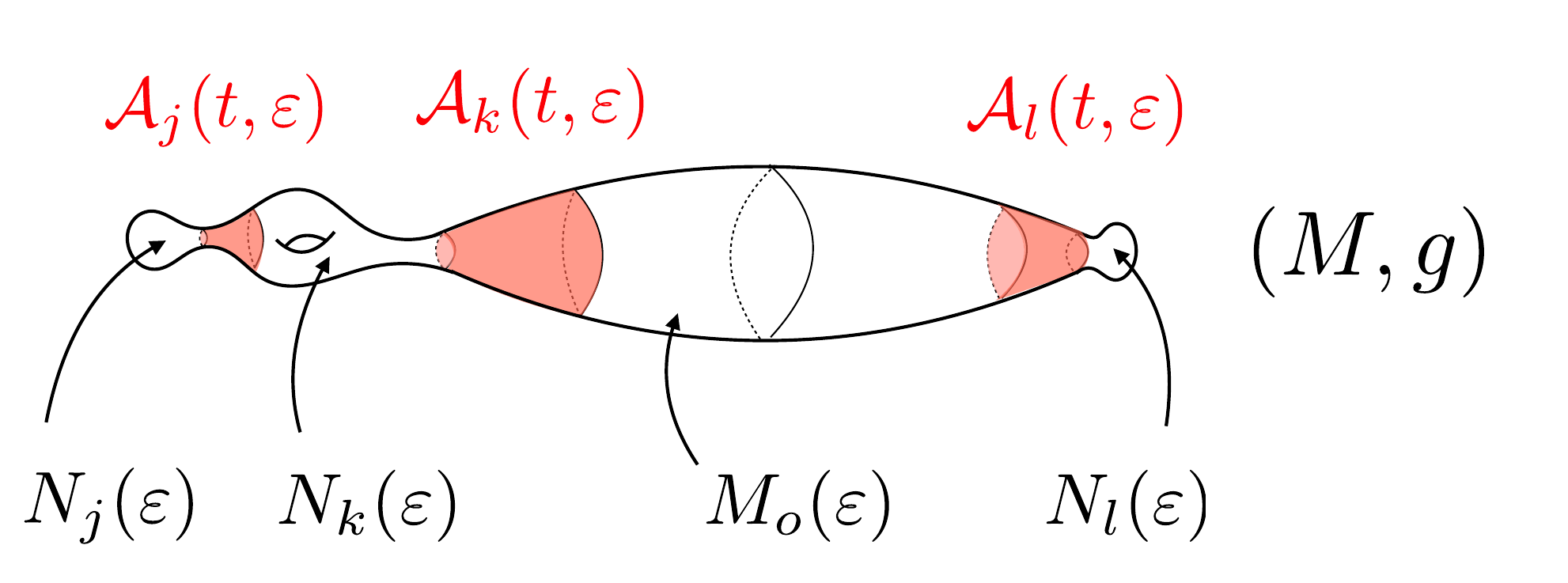}
\end{figure}

\vspace{-0.4cm}
\begin{defn}[Cut-off functions $\chi_{M_o^t}$, $\chi_{N_j^t}$, $\chi_{\mathcal{A}_k(t,\varepsilon)}$ and $\chi_{B(p_k,\varepsilon)}$]\label{def cutoffs all}

    We define the following cut-off functions thanks to the cut-off function $\chi$ used in Definition \ref{def naive desing}.
    \begin{itemize}
        \item $\chi_{M_o^t}$, equal to $1$ on $M_o^{16t}$ and equal to $1-\chi(t_k^{-\frac14}r_o)$ on each annulus $\mathcal{A}_k(t,\varepsilon)$. It is supported on $M_o^t$.
        \item $\chi_{N_j^t}$, equal to $1$ on $N_j^{16t}$ and equal to $1-\chi(t_k^{-\frac14}r_{b_j})$ on each annulus $\mathcal{A}_k(t,\varepsilon)$ at its singular points and $\chi(t_j^{\frac14}r_{b_j})$ in a neighborhood of infinity. It is supported on $N_j^t$.
        \item $\chi_{\mathcal{A}_k(t,\varepsilon)}$ equal to $1$ on $\mathcal{A}_k(t,\frac1{2}\varepsilon)$, and equal to $\chi(\varepsilon^{-1}t_k^{\frac1{2}}r_{b_k}) - \chi(\varepsilon r_{b_k})$. It is supported on $\mathcal{A}_k(t,\varepsilon)$.
        \item $\chi_{B(p_k,\varepsilon)}$ for $p_k\in M_o$ equal to $1$ on $r_o<\varepsilon$ equal to $\chi(\varepsilon^{-1}r_o)$ around $p_k$. It is supported in  supported in $r_o<2\varepsilon$ around $p_k$.
    \end{itemize}
\end{defn}}

\subsubsection{Refined approximate kernel of the linearization}\label{sec: refined obst}

One of the subtleties of our proof is that we need to carefully estimate the divergence of our tensors at several stages. This requires a refined notion of \textit{obstructions} from that of \cite[Definition 4.5]{ozu2}.

\textbf{Extending elements of $\mathbf{O}(g_o)$ or of $\mathbf{O}(g_{b_j})$ to deeper bubbles.}

\begin{defn}[Extension of elements of $\mathbf{O}(g_o)$ or of $\mathbf{O}(g_{b_j})$]\label{def: ext obst}
\,
\begin{enumerate}
    \item Consider an element of $\mathbf{o}\in\mathbf{O}(g_o)$ (resp. of $\mathbf{o}\in\mathbf{O}(g_{b_j})$).
    \item For each $p_{[0]}\in S_o$ (resp. $p_{[0]}\in S_{b_j}$) consider the constant $2$-tensor $H_{[0]}:=\mathbf{o}(p_{[0]})$ on the cone.
    \item Consider $j_{[0]} = p^{-1}(p_{[0]})$ and the Ricci-flat ALE metric $(N_{j_{[0]}},g_{b_{j_{[0]}}})$ and consider the unique symmetric $2$-tensor $ \mathbf{h}_{j_{[0]}}$ asymptotic to $H_{[0]}$ (obtained as in \cite[Lemma 2.7] {ozu3}) and satisfying:
    
    \begin{equation}
        P_{g_{b_{j_{[0]}}}}\mathbf{h}_{j_{[0]}} = 0,\qquad\delta_{g_{b_{j_{[0]}}}}\mathbf{h}_{j_{[0]}} = 0,\qquad\operatorname{tr}_{g_{b_{j_{[0]}}}}\mathbf{h}_{j_{[0]}} = 0.
    \end{equation}
    \item We then iterate: if for $k\geqslant 0$, $\mathbf{h}_{j_{[k]}}$ is defined, consider $p_{[k]}\in S_{b_{j_{[k]}}}$, and $j_{[k+1]} = p^{-1}(p_{[k]})$ and the constant $2$-tensor $H_{[k]}:=\mathbf{h}_{j_{[k]}}(p_{[k]})$. We then define $\mathbf{h}_{j_{[k+1]}}$ satisfying
    \begin{equation}
        P_{g_{b_{j_{[k+1]}}}}\mathbf{h}_{j_{[k+1]}} = 0,\qquad\delta_{g_{b_{j_{[k+1]}}}}\mathbf{h}_{j_{[k+1]}} = 0,\qquad\operatorname{tr}_{g_{b_{j_{[k+1]}}}}\mathbf{h}_{j_{[k+1]}} = 0.
    \end{equation}
\end{enumerate}
We therefore have tensors $\mathbf{h}_k$ defined on each $(N_k,g_{b_k})$ for $k\in J_{p(j)}$.

\end{defn}

\textbf{Projecting elements of $\mathbf{O}(g_o)$ or of $\mathbf{O}(g_{b_j})$ to the naïve desingularization $(M,g^D)$.}

We define the projection of an element $\mathbf{o}_o\in \mathbf{O}(g_o)$ on the naïve desingularization:
$$\pr(\mathbf{o}_o) := \tilde{\mathbf{o}}_{o,t}: = \chi_{M_o^t}\mathbf{o}_o+\sum_{k} \chi_{N_k^t}\mathbf{h}_k, $$
where the $\mathbf{h}_k$ are defined in Definition \ref{def: ext obst}. Similarly, for $\mathbf{o}_{j}\in \mathbf{O}(g_{b_j})$,  
$$\pr(\mathbf{o}_j) := \tilde{\mathbf{o}}_{j,t} := \chi_{N_j^t}\mathbf{o}_j+\sum_{k\in J_{p(j)}} \chi_{N_k^t}\mathbf{h}_k. $$

\begin{defn}[Space of refined obstructions]\label{def obst tronq}
    Let $(M,g^D_{t})$ be a naïve desingularization of a (compact or ALE) Einstein orbifold $(M_o,g_o)$. On $M$, we will denote 
    $$\tilde{\mathbf{O}}(g^D):=\Big\{\tilde{\mathbf{o}}_{o,t} + \sum_j\tilde{\mathbf{o}}_{j,t}, \;\mathbf{o}_o\in \mathbf{O}(g_o), \; \mathbf{o}_j\in \mathbf{O}(g_{b_j})\Big\}.$$ 
    We have also defined a one-to-one map $\pr: \mathbf{O}(g_o)\sqcup\bigsqcup_j\mathbf{O}(g_{b_j})\to \tilde{\mathbf{O}}(g^D)$.
\end{defn}

The reason to refine our obstructions is to improve how divergence-free they are. This also makes them better approximations of infinitesimal Einstein deformations, i.e. their image by the linearization operator $P_{g^D}$, {defined in \eqref{eq:def bar P}}, are smaller than in \cite[Lemma 4.5]{ozu2}. In particular, the proofs of \cite{ozu2} work with the above Definition \ref{def obst tronq} for $\tilde{\mathbf{O}}(g^D)$.

\begin{lemma}\label{lem: div P obst}
    Whether $g_o$ is compact or ALE, we have the following controls on the elements of $\tilde{\mathbf{o}}\in\tilde{\mathbf{O}}(g^D)$: for $C>0$ depending on $g_o$ and $g_{b_j}$, with $t_{\max} = \max_jt_j$ and with the function spaces of Section \ref{sec:function spaces},
\begin{equation}\label{eq:control operators on tilde o}
\| \delta_{g^D} \tilde{\mathbf{o}} \|_{r_D^{-1}C^{1,\alpha}_\beta}< C t_{\max}^{\frac{2-\beta}4} \| \tilde{\mathbf{o}} \|_{C^{2,\alpha}_{\beta,*}},\quad \| \delta_{g^D} \tilde{\mathbf{o}} \|_{r_D^{-3}C^{1,\alpha}_\beta}< C t_{\max}^{\frac{2-\beta}4} \|\tilde{\mathbf{o}} \|_{L^2(g^B)}\quad\text{and }\quad
\| P_{g^D} \tilde{\mathbf{o}} \|_{r_D^{-2}C^{0,\alpha}_\beta}< C t_{\max}^{\frac{2-\beta}4} \| \tilde{\mathbf{o}} \|_{C^{2,\alpha}_{\beta,*}}.
\end{equation}
\end{lemma}
\begin{proof}
    First note that on $M^{16t}_o$ and each $N_j^{16t}$, one has $ \delta_{g^D} \tilde{\mathbf{o}} = 0 $ and $ P_{g^D} \tilde{\mathbf{o}} = 0 $. There remains to understand the element in the annular regions $M^t_o\backslash M^{16t}_o$ and $N^t_j\backslash N^{16t}_j$.

    In such regions, where $\frac1{2}\sqrt{T_k}t_k^{-\frac1{4}}<r_D<2\sqrt{T_k}t_k^{-\frac1{4}}$ one has 
    \begin{equation}\label{eq control metric annulus}
        |\nabla^l_{g_{e}}(g^D - g_{e})| \leqslant C r_D^{-l}t_k^{\frac1{2}}
    \end{equation}
 for $l\leqslant 4$ and $C>0$. At the same time, by construction, for a $g_{e}$-parallel symmetric $2$-tensor $H$, we have 
 \begin{equation}\label{eq control tensor annulus}
     |\nabla_{e}^l(\tilde{\mathbf{o}}_{o} - H)| \leqslant C r_D^{-l} t_k^{\frac1{2}},
 \end{equation}
for $l\leqslant 4$ and $C>0$ on these regions $M^t_o\backslash M^{16t}_o$ and $N^t_j\backslash N^{16t}_j$ where 
 $\frac1{2}\sqrt{T_k}t_k^{-\frac1{4}}<r_D<2\sqrt{T_k}t_k^{-\frac1{4}}$.

Combining \eqref{eq control metric annulus} and \eqref{eq control tensor annulus}, we find for $C>0$ independent of $\tilde{\mathbf{o}}$ and $l\leqslant 2$
$$ \nabla^l(\delta_{g^D} \tilde{\mathbf{o}}) \leqslant C r_D^{-1-l} t_k^{\frac1{2}}, \qquad \text{ and } \qquad \nabla^l(P_{g^D} \tilde{\mathbf{o}}) \leqslant C r_D^{-2-l} t_k^{\frac1{2}},$$
which yields the result. The control by the $L^2$ norm comes from the fact that the $L^2(g^B)$ norm and the $r_B^{-2}C^{l,\alpha}_{\beta}(g^B)$ for $l\geqslant 0$ norms are equivalent independently of $t$, see \cite[Remark 4.45]{ozu2}.
\end{proof}

\subsection{Refined compactness of noncollapsed Einstein $4$-manifolds}\label{sec:compactness dim 4}

Let us explain the results from \cite{ozu1,ozu2,ozuthese} which we will need in the rest of the article. Consider for $t= (t_j)_j$, a naïve gluing $g^D_t$ as in Definition \ref{def naive desing} and its associated space of obstructions $\tilde{\mathbf{O}}(g^D_t)$.

\subsubsection{Einstein modulo obstruction perturbation of a naïve gluing}

Define $\tilde{\delta}_{g} = \pi_{\tilde{K}_o^\perp} \delta_g$ is the $L^2(g^D_{t})$-projection of the divergence orthogonally to the set $\tilde{K}_o^\perp$ of cut-offs of the Killing vector fields of $(M_o,g_o)$, whose definition we skip, see \cite[Section 3.2]{ozu2}. The operator $\tilde{\delta}_g\delta_g^*$ is invertible in the relevant weighted space by \cite[Proposition 3.3]{ozu2}. We define $\pi_g:= 1-\delta_g^*(\delta_g\delta_g^*)^{-1}\tilde{\delta}_g$, the $L^2(dv_g)$-projection on $\ker \tilde{\delta}_g$. Thanks to the Einstein operator $\E(g) = \Ric(g)-\frac{\R(g)}{2}g$, we introduce the $g^D_t$-\textit{gauged Einstein} operator 
$$\mathbf{\Phi}_{g^D_{t}}(g) = \E(g) + \delta_g^*\tilde{\delta}_{g^D_{t}}g.$$

\begin{thm}[{\cite[Theorem 4.6 and Remark 4.56]{ozu2}}]\label{thm: pert mod obst}
    There exists $\varepsilon_0>0$ such that for any $0<t<\varepsilon_0^4$ and any $v\in\tilde{\mathbf{O}}(g^D_t)$ with $|v|\leqslant \varepsilon_0$, there is a \emph{unique} solution $(\hat g,\hat w) = (\hat g_{t,v},\hat w_{t,v}) \in C^{2,\alpha}_{\beta,*}(g^D_{t })\times \tilde{\mathbf{O}}(g^D_{t }) $ with $\|\hat g-g^D_t\|_{C^{2,\alpha}_{\beta,**}} {< \varepsilon_0}$ (the norm ${C^{2,\alpha}_{\beta,**}}$ is recalled in Definition \ref{def: Ckalphabeta*}) and to the \textit{Einstein modulo obstructions} equation close to $g^D_{t}$: satisfying $\hat g-(g^D_{t}+v) \perp_{L^2(g^D_{t})}\tilde{\mathbf{O}}(g^D_t)$ and 
\begin{equation}\label{eq:RF mod obst}
    \mathbf{\Phi}_{g^D_{t}}(\hat g) + \pi_{\hat g} \hat w =0.
\end{equation}
The resulting metric $\hat{g}$ is Einstein \emph{if and only if} the obstruction $\hat{w}$ vanishes. 
\end{thm}

\begin{remark}
    By the Bianchi identity, taking $\tilde{\delta}_{\hat g}$ of \eqref{eq:RF mod obst} we find:
$(\tilde{\delta}_{\hat g}\delta_{\hat g}^*)\tilde{\delta}_{g^D_{t}}{\hat g} = 0,$
hence ${\hat g}$ automatically solves the gauge condition $\tilde{\delta}_{g^D_{t}}{\hat g} = 0$ by the invertibility of $\delta_{\hat g}\delta_{\hat g}^*$. We could therefore replace $\mathbf{\Phi}_{g^D_{t}}({\hat g}) + \pi_{\hat g} \hat{w} =0$ by $\E({\hat g}) + \pi_{\hat g} \hat{w} =0$ and $\tilde{\delta}_{g^D_{t}}{\hat g} = 0$.
\end{remark}

\begin{remark}\label{rem:gauged modulo obst}
    Here, we considered the gauged notion of Einstein modulo obstructions metric of \cite[Section 6]{biq1} and \cite[Remark 4.56]{ozuthese}, but different from \cite{ozu2}. 
\end{remark}

 The linearization of
$g\mapsto\mathbf{\Phi}_{g^D}(g)$ at $g^D$is
\begin{align}\label{eq:def bar P}
        \bar{P}_{{g^D}}(h):= \;d_{g^D}\mathbf{\Phi}_{g^D}(h) =& \frac1{2}\Big(\nabla^*_{g^D}\nabla_{g^D} h-2\mathring{\R}_{g^D}(h)+\Ric_{g^D}\circ h+h\circ \Ric_{g^D}-\R_{g^D} h+ \langle \Ric_{g^D}, h \rangle_{g^D} {g^D} \Big)\nonumber\\
        &-2 \delta_{g^D}^*\delta_{g^D}h+2\delta^*_{{g^D}}\tilde{\delta}_{g^D} h-(\delta_{g^D}\delta_{g^D} h) {g^D} + (\Delta_{g^D} \textup{tr}_{g^D}h) {g^D} - \nabla_{g^D}^2 \textup{tr}_{g^D} h,
\end{align}
in dimension $4$. If $g^D$ were an Einstein metric and $h$ a divergence-free and traceless symmetric $2$-tensor, then the operator $\bar{P}_{g^D}$ would reduce to $P_{g^D}:= \frac1{2}\nabla_{g^D}^*\nabla_{g^D} - \mathring{\R}_{g^D}.$ The proof of Theorem \ref{thm: pert mod obst} relies on a Lyapunov-Schmidt reduction and an application of the implicit function theorem between Banach spaces $C^{2,\alpha}_{\beta,*}$ and $r_D^{-2}C^{\alpha}_{\beta}$. There is a \textit{quantitative} uniqueness associated with Theorem \ref{thm: pert mod obst} as it comes from an application of the implicit function theorem, see \cite[Section 9]{biq1} and \cite[Section 4]{ozu2}.

\begin{cor}\label{cor: quantitative uniqueness desing}
    There exists $C >0$ independent on $t,v$ so that if an \textit{approximation} $(g^A_{t ,v },w^A_{t ,v })\in C^{2,\alpha}_{\beta,*}(g^D_{t })\times \tilde{\mathbf{O}}(g^D_{t })$ satisfies:  $\|g^A_{t ,v }-g^D_{t }\|_{C^{2,\alpha}_{\beta,*}(g^D_{t })}<\frac1{C},$  and $g^A_{t ,v }-g^D_{t }+v \perp_{L^2(g^D_{t })} \tilde{\mathbf{O}}(g^D_{t })$, then, one has:
\begin{equation}\label{eq:control approx obst thm}
    \|g^A_{t ,v }-\hat g_{t ,v } \|_{C^{2,\alpha}_{\beta,*}} + \|w^A_{t ,v }-\hat w_{t ,v }\|_{L^2(g^D_t)} < C \big\|\mathbf{\Phi}_{g^D_{t }}(g^A_{t ,v }) + \pi_{g^A_{t,v}}w^A_{t ,v } \big\|_{r_D^{-2}C^\alpha_{\beta}(g^D_{t })}.
\end{equation}
\end{cor}

\begin{remark}
    Applied to $(g^A_{t ,0 },w^A)=(g^D_t,0)$, \eqref{eq:control approx obst thm} implies that $\|g^D_t-\hat g_{t ,0 } \|_{C^{2,\alpha}_{\beta,*}}\leqslant Ct_{\max}^{\frac{2-\beta}{4}}$.
\end{remark}
\subsubsection{Einstein metric Gromov-Hausdorff close to an Einstein orbifold}

Theorem \ref{thm: pert mod obst} is not merely a gluing-perturbation theorem: the tedious norms involved are not chosen for their convenient analytical properties but from the Gromov-Hausdorff convergence of Einstein metrics. Any Einstein metric sufficiently close to an orbifold is reached by Theorem \ref{thm: pert mod obst}.

\begin{thm}[{\cite{ozu1,ozu2}}]\label{thm: exhaustion neighb orb}
   Let $(M_o,g_o)$ be an Einstein orbifold. There exists $\delta>0$ such that if $(M,g)$ is an Einstein $4$-\textit{orbifold} satisfying 
$$d_{GH}\big((M,g),(M_o,g_o)\big)<\delta,$$ 
Then, there exists a naïve gluing $g^D_{t }$ and $ v \in \tilde{\mathbf{O}}(g^D_{t  })$ such that:
$$\hat{g}_{t,v} = \phi^*g\qquad \text{and}\qquad \hat{w}_{t,v} = 0.$$
\end{thm}

\section{Gromov-Hausdorff isolation of spherical orbifolds}\label{sec:GH isolation}

Let us assume that an Einstein metric $(M,g)$ is sufficiently close to an Einstein orbifold $(M_o,g_o)$. Then, with $\phi:M\to M$ from Theorem \ref{thm: exhaustion neighb orb}, since $(\phi^*g , 0)$ is the only solution of the equation, one has by  \eqref{eq:control approx obst thm}:
\begin{equation}\label{eq:control approx obst}
    \|g^A_{t ,v }-\phi ^*g \|_{C^{2,\alpha}_{\beta,*}} + \|w^A\|_{L^2(g^D)} < C \big\|\mathbf{\Phi}_{g^D_{t }}(g^A) + \pi_{g^A_{t,v}}w^A \big\|_{r_D^{-2}C^\alpha_{\beta}(g^D_{t })}.
\end{equation}
In particular, if $\|w^A\|_{L^2(g^D)}$ is much larger than $C\big\|\mathbf{\Phi}_{g^D_{t }}(g^A) +\pi_{g^A} w^A \big\|_{r_D^{-2}C^\alpha_{\beta}(g^D_{t })}$, we reach a contradiction as used in \cite[Section 5.3]{ozu2}. Our strategy is therefore to construct such an approximation $(g^A,w^A)=(g^A_{t ,v },w^A_{t ,v })$ and to estimate $w^A$ carefully. Constructing this approximation will be done following these steps:
\begin{enumerate}
    \item For each singular point $p_o$ of $M_o$, consider the \textit{bubble tree} composed of the Ricci-flat ALE metrics collapsing at $p_o$ and $g^B=g^B_{t }$ the glued metric so that $g^B = g^D_{t }$ as in Definition \ref{def:tree of sing at po}. 
    
    \item We start by perturbing $(g^B,0)$ to a Ricci-flat modulo obstruction metric $(\bar{g}_B,u_B)$ and show that $\Ric(\bar{g}_B)$ decays fast at infinity.  This is the content of Proposition \ref{prop: gluing trees}.
    \item We then show in Corollary \ref{cor:estimate u_B} that either $u_B$ is orthogonal to a specific \textit{rescaling obstruction} $\mathbf{o}_1$ or it is much larger than its projection on $\mathbf{o}_1$.  
    \item In Section \ref{sec:approx metric}, we then construct our approximate metric $g^A$ by gluing the various Ricci-flat modulo obstructions gluings of bubble trees, $\bar{g}_B$, to a hyperbolic or spherical orbifold $(M_o,g_o)$ while keeping track of additional obstructions $u_o\approx \lambda_1\mathbf{o}_1$ with $\lambda_1\neq 0$ this gluing induces.  
    \item On the bubble tree region, the resulting approximate obstruction is $w^A = u_B + \lambda_1\mathbf{o}_1$. The end of the proof of Theorem \ref{thm:isolation orbifolds} is in Section \ref{sec: isolation hy sph orb}. The argument is: if the projection of $w^A = u_B+\lambda_1\mathbf{o}_1$ onto $\mathbf{o}_1$ were small enough not to contradict \eqref{eq:control approx obst}, then $u_B$ would in fact be much larger than $\lambda_1\mathbf{o}_1$ and the projection of $w^A$ on the orthogonal of $\mathbf{o}_1$ would contradict \eqref{eq:control approx obst}. 
\end{enumerate}
\subsection{Gluing bubble trees and controls}\label{sec:gluing trees}

In our context, $(M,g^D_{t })$ is composed of the core orbifold $(M_o,g_o)$ and of a bubble tree at each singular point. Let us pick one such singular point $p_o\in M_o$ and denote $g^B_{t }$ the associated bubble tree metric. The Ricci-flat ALE metrics in this tree are $(N_j,g_{b_j})_{j\in J_{p_o}}$. We similarly keep track of $v ^B\in \tilde{\mathbf{O}}(g^B_{t })$, the piece of $v \in\tilde{\mathbf{O}}(g^D_{t })$, supported in the bubble tree associated with $g^B_{t }$. 

We first extend \cite[Theorem 4.6]{ozu2} and show that $g^B_{t }$ can be perturbed to a Ricci-flat modulo obstruction metric in suitable spaces. There is no need to project the divergence orthogonally to Killing vector fields. Indeed, since there are no Killing vector fields growing slower than linearly at infinities of Ricci-flat ALE metrics, the operator ${\delta}_{g^B_{t }}\delta_{g^B_{t }}^*$ (without tilda) is invertible between the relevant spaces.

\begin{proposition}\label{prop: gluing trees}
    For any $0<\alpha<1$, $0<\beta<2$ and $\varepsilon>0$, there exist $C>0$ and $\delta>0$ such that for any $t_{\max},\|v\|_{C^{2,\alpha}_{\beta,*}}<\delta$, there is a unique solution $(\Bar{g}^B_{t ,v },u^B )$ to 
\begin{itemize}
    \item $\|g^B_{t }-\Bar{g}^B_{t ,v }\|_{C^{2,\alpha}_{\beta,*}(g^D_{t })}<\varepsilon,$ 
    \item $g^B_{t }+v ^B-\Bar{g}^B_{t ,v ^B}\perp_{L^2(g^D_{t })} \tilde{\mathbf{O}}(g^\textbf{D}_{t })$,
    \item $\mathbf{\Phi}_{g^B_{t }}(\bar{g}^B_{t ,v }) + \pi_{\bar{g}^B_{t ,v }}u^B  = 0$, or equivalently: $\E(\bar{g}^B_{t ,v }) + \pi_{\bar{g}^B_{t ,v }}u^B  = 0,\text{ and}$
$\delta_{g^B_{t }}(\bar{g}^B_{t ,v }) = 0.$
\end{itemize}
where we denote $\tilde{\mathbf{O}}(g^B_{t })$ the set of elements of $\tilde{\mathbf{O}}(g^D_{t })$ supported in the bubble tree region associated to $g^B_{t }$, and where $u^B \in \tilde{\mathbf{O}}(g^B_{t })$. We have the following estimates: 
\begin{equation}\label{eq:control metric}
    \|g^B_{t }-\bar{g}^B_{t ,v }\|_{C^{2,\alpha}_{\beta,*}(g^D_{t })}\leqslant C (t^{\frac{2-\beta}4} + \|v ^B\|_{C^{2,\alpha}_{\beta,*}}^2),
\end{equation}
as well as $\E(\bar{g}_B)=\cO(r^{-6}) = o(r^{-4})$ at infinity,
hence in particular,
\begin{equation}\label{eq:decay pi uB}
    \pi_{\bar{g}_B}u_B =\cO(r^{-6})= o(r^{-4}).
\end{equation}
\end{proposition}
\begin{proof}
    For ease of notation, we will denote $g^B = g^B_{t }$, $\bar{g}_B=\bar{g}^B_{t ,v }$. The proof is an application of the implicit function theorem between Banach spaces. For each element $v\in \tilde{\mathbf{O}}(g^B)$ (in particular $v ^B$), one searches for $(g(v),u(v))$ with $g(v)=g^B+v+ h(v)$, such that 
$$\mathbf{\Phi}_{g^B}(g(v)) + \pi_{g(v)}u(v) = 0$$
where $(h(v),u(v))\in (C^{2,\alpha}_{\beta,**}\cap \tilde{\mathbf{O}}(g^B)^\perp\cap \ker{\delta_{g^B}})\times \tilde{\mathbf{O}}(g^B)$ where $C^{2,\alpha}_{\beta,**}$ is defined in Section \ref{function spaces}.

The implicit function theorem is applied to the set of equations parametrized by $v$:
$$ \Psi_{g^B}(v,h,u) = \mathbf{\Phi}_{g^B}(g^B +v +h) + \pi_{g^B +v +h}u. $$
$\Psi_{g^B}$ is a smooth map in its arguments between the Banach spaces:
$$\left(\tilde{\mathbf{O}}(g^B)\cap C^{2,\alpha}_{\beta,**}\right)\times \left(C^{2,\alpha}_{\beta,**}\cap \tilde{\mathbf{O}}(g^B)^\perp\cap \ker{\delta_{g^B}}\right)\times \left(\tilde{\mathbf{O}}(g^B)\cap r_B^{-2}C^{\alpha}_{\beta}\right)\to (1+r_B)^{-4}C^{\alpha}_{\beta}.$$

First note that as in the proof of Lemma \ref{lem: div P obst}, one has for some $C>0$ independent of $t$,
\begin{equation}\label{eq: Psi at gB}
    \|\Psi_{g^B}(0,0,0)\|_{r_D^{-2}C^\alpha_{\beta}} \leqslant C t_{\max}^{\frac{2-\beta}4}.
\end{equation}
Denote by $\pi_{\tilde{\mathbf{O}}(g^B)^\perp}$ the $L^2(g^B)$-orthogonal projection on $\tilde{\mathbf{O}}(g^B)^\perp$. We know from Lemma \ref{lem:inverting barP with terms at infty} that the linearization of $\pi_{\tilde{\mathbf{O}}(g^B)^\perp}\Psi_{g^B}$ at $(0,0,0)$, is invertible in the direction $\{0\}\times(C^{2,\alpha}_{\beta,**}\cap \tilde{\mathbf{O}}(g^B)^\perp\cap \ker{\delta_{g^B}})\times\{0\}$ with an inverse whose norm is independent of $t$. Consequently, the linearization of $\Psi_{g^B}$ is invertible in the last two components and the implicit function theorem applies. The resulting $(h(v),u(v))$ is controlled as:
$$ \|h(v)\|_{C^{2,\alpha}_{\beta,**}} + \|u(v)\|_{r_B^{-2}C^{\alpha}_{\beta}} \leqslant C \Big(t^{\frac{2-\beta}4}+ \|v\|_{C^{2,\alpha}_{\beta,*}}^2\Big). $$
By Remark \ref{rem:gauge **}, $h(v)$ has no $r^{-2}$ term in its expansion and $\nabla^lh(v)= \mathcal{O}(r^{-4-l})$. Consequently, we find $$\E\big(\bar{g}^B_{t ,v ^B}\big) =\E\left(g^B_{t ,v ^B} + v ^B +h(v ^B)\right) = \mathcal{O}(r^{-6}).$$
\end{proof}

At infinity, we assume that $g_B$ is in a set of coordinates where 
$$ g_B = g_{e} + H^4 + \mathcal{O}(r_b^{-5}),$$
where $H^4$ is a homogeneous harmonic symmetric $2$-tensor with $|H^4| = \mathcal{O}(r^{-4})$ with $\tr_{e} H^4 = 0$, $\delta_{e} H^4 = 0$ and $H^4(\partial_{r_b},\partial_{r_b}) = 0$. The existence of such coordinates is proven in \cite[Lemma 11.12]{do3}

{
Let us start by decomposing $\tilde{\mathbf{O}}(g^B)$ the refined space of obstructions from Section \ref{sec: refined obst}. Denote $(N,g_{b}) = (N_{j_o},g_{b_{j_{o}}})$, where $j_o = j_{p^{-1}(p_o)}$, the core orbifold of the bubble tree as in Definition \ref{def:tree of sing at po}. 

An infinitesimal Ricci-flat deformation of $g_b$, $\mathbf{o}_1\in\mathbf{O}(g_b)$, plays a distinct role:
$$\mathbf{o}_1 := \Hess_{g_b}u - \frac14(\Delta_{g_b}u) g_b,$$
where $u$ is the unique solution of $\Delta_{g_b}u = 8$, $u=r_b^2+o(1)$, see \cite{bh,ozuthese}. On $(N,g_{b})$,  we decompose $L^2(g_b)$-orthogonally
$$\mathbf{O}(g_b) = \mathbf{O}_0(g_b) \overset{}{\oplus} \mathbb{R} \mathbf{o}_1, \qquad \text{ and define }\qquad\tilde{\mathbf{O}}_0(g^B) := \pr(\mathbf{O}_0(g_b)) \oplus \bigoplus_{j\in J_{p_o}\backslash\{j_{o}\}}\pr(\mathbf{O}(g_{b_j}))$$ and denote  $(\tilde{\mathbf{o}}_k)_{k\geqslant 2}$ an $L^2(g^B)$-orthogonal basis of $\tilde{\mathbf{O}}_0(g^B)$. Then, there exist numbers $(\mu_k)_k$ such that 
\begin{equation}\label{eq:decomposition uB}
    u_B = \mu_1\tilde{\mathbf{o}}_1 + \sum_{k\geqslant 2}\mu_k \tilde{\mathbf{o}}_k.
\end{equation}}

\begin{corollary}\label{cor:estimate u_B}
    We have the estimate: 
    \begin{equation}\label{eq: estimation tree obstructions}
        \mu_1\|\tilde{\mathbf{o}}_1\|_{L^2(g_b)} + \sum_{k\geqslant 2}\varepsilon_k(t ,v ^B)  |\mu_k| = 0,
    \end{equation}
    where $\varepsilon_k(t ,v ^B)\to 0$ as $(t ,v ^B)\to 0$.
\end{corollary}
\begin{remark}
    This corollary says that if $\mu_1\neq 0$, then at least one of the other $\mu_k$ is much larger than $\mu_1$. This is an extension to trees of singularities of \cite[Proposition 1.12]{ozu4}, which relied on the real-analyticity of the moduli space of Ricci-flat modulo obstructions metrics close to a single \textit{smooth} Ricci-flat ALE metric. There is no such result for trees of Ricci-flat ALE metrics. 
\end{remark}
\begin{proof}
The proof relies on an application of \cite[Lemma 11.16]{do3} together with \eqref{eq:decomposition uB} which says that if $\bar{g}^B$ is asymptotic to $\mathbb{R}^4/\Gamma$, then $\tilde{\mathbf{o}}_1$ is the only component of $u_B$ inducing some $\frac{dr_b^2}{r_b^4}$ at infinity:
\begin{equation}\label{eq: asympt uB}
    \begin{aligned}
        u_B(\partial_{r_b},\partial_{r_b}) &= \mu_1\tilde{\mathbf{o}}_1(\partial_{r_b},\partial_{r_b}) + \sum_{k\geqslant 2}\mu_k \tilde{\mathbf{o}}_k(\partial_{r_b},\partial_{r_b}) \\
        &= C \mu_1 \frac{|\Gamma|}{|\mathbb{S}^3|} \|\tilde{\mathbf{o}}_1\|^2_{L^2(g_b)} r_b^{-4}+o(r_b^{-4}).
    \end{aligned}
\end{equation} 

 If we had the same expansion for $\pi_{\bar{g}_B}u_B(\partial_{r_b},\partial_{r_b})$, we would conclude instantly from \eqref{eq:decay pi uB} that $\mu_1 = 0$. The rest of the proof consists in controlling the $O(r_b^{-4})$ terms of $\pi_{\bar{g}_B}u_B-u_B$.

From \eqref{eq:decomposition uB}, we have
$\pi_{\bar{g}_B}u_B = \mu_1\pi_{\bar{g}_B}\tilde{\mathbf{o}}_1 + \sum_{k\geqslant 2}\mu_k \pi_{\bar{g}_B}\tilde{\mathbf{o}}_k$. Let us control $\pi_{\bar{g}_B}\tilde{\mathbf{o}}_k-\tilde{\mathbf{o}}_k = -\delta^*_{\bar{g}_B}(\delta_{\bar{g}_B}\delta^*_{\bar{g}_B})^{-1}\delta_{\bar{g}_B}\tilde{\mathbf{o}}_k$.

Each of the $\tilde{\mathbf{o}}_k$ is an element of $C^{k,\alpha}_{\beta,**}$ and $\delta_{\bar{g}_B}\tilde{\mathbf{o}}_k\in (1+r_B)^{-4}r_B^{-1}C^{1,\alpha}_\beta$ since it vanishes in a neighborhood of infinity. Consequently, by Lemma \ref{lem:controle inverse mise jauge ALE}, $(\delta_{\bar{g}_B}\delta^*_{\bar{g}_B})^{-1}\delta_{\bar{g}_B}\tilde{\mathbf{o}}_k\in r_BC^{3,\alpha}_{\beta,**}$, and in particular, $(\delta_{\bar{g}_B}\delta^*_{\bar{g}_B})^{-1}\delta_{\bar{g}_B}\tilde{\mathbf{o}}_k = -Y^3+o(r^{-3})$ at infinity for some homogeneous $Y^3\in\ker \delta_e\delta^*_e$ with $|Y^3|_e\propto r^{-3}$. This is the standard behavior of elliptic operators between weighted Hölder spaces as one crosses an exceptional value. We deduce that
\begin{equation}\label{eq:exp pi ok}
    \pi_{\bar{g}_B}\tilde{\mathbf{o}}_k-\tilde{\mathbf{o}}_k = \delta_{e}^*Y^3 + o(r^{-4})
\end{equation} 
for some harmonic vector field $Y^3$ in $r^{-3}$ at infinity. The first thing to note is that $\pi_g$ is exactly the $L^2(g^B)$-projection of $\tilde{\mathbf{o}}_k$ onto $\ker \delta_{g^B}$, so we deduce that 
$$\|\pi_{\bar{g}_B}\tilde{\mathbf{o}}_k-\tilde{\mathbf{o}}_k\|_{L^2(g^B)} \leqslant C \Big((\max_{j\in J_{p_o}} t_{j})^{\frac{2-\beta}4} + \|v ^B\|_{C^{2,\alpha}_{\beta,*}}^2\Big) \|\tilde{\mathbf{o}}_k\|_{L^2(g^B)}.$$

There remains to determine how large $Y^3$ is compared to the $L^2$-norm of $\tilde{\mathbf{o}}_k$. By definition of $\pi_{\bar{g}^B}$, the term $Y^3$ is the asymptotic term of a solution $Y = Y^3+o(r^{-3})$ of the solution of 
\begin{equation}\label{eq:Y}
    \delta_{\bar{g}^B}\delta_{\bar{g}^B}^* Y =\delta_{\bar{g}^B}(\tilde{\mathbf{o}}_k).
\end{equation}
We use Proposition \ref{prop:duality vect fields
} to control it. If in the bases of Section \ref{sec:duality vect fields} one has the decomposition $Y^3=\sum_i\alpha_i X_{a_1^-,i} + \beta_i X_{c_2^-,i} + \gamma_i X_{b_0^-,i}$, consider $X_1 = \sum_i\alpha_i X_{a_1^+,i} + \beta_i X_{b_2^+,i} + \gamma_i X_{c_0^+,i}$ the \textit{dual} harmonic linear vector field. Then for the solution $\mathcal{X}_1$ of Lemma \ref{lem harmonic vect field linear} of $\delta_{\bar{g}^B}\delta_{\bar{g}^B}^*\mathcal{X}_1=0$ asymptotic to the linear vector field $X_1$ at infinity, we have by \eqref{eq:Y} and \eqref{eq:duality vector fields}
\begin{equation}\label{eq:ibp vect}
 \left|\int_M\delta_{\bar{g}^B}(\tilde{\mathbf{o}}_k)(\mathcal{X}_1)\right| = \left|\int_M \delta_{\bar{g}^B}\delta_{\bar{g}^B}^* Y (\mathcal{X}_1) \right| = \left|B(Y^3,X_1)\right| = \sum_iC_\alpha\alpha_i ^2 + C_\beta\beta_i ^2 + C_\gamma\gamma_i^2=:\|Y^3\|^2.
\end{equation}
Here, the right-hand-side is quadratic in $\|Y^3\|$ while the left-hand-side is linear in $X_1$ (hence in $Y^3$) and multiplied by the tensor $\delta_{\bar{g}^B}(\tilde{\mathbf{o}}_k)$, which is small by \eqref{eq:control operators on tilde o}. 

More concretely, by \eqref{eq:control operators on tilde o} and \eqref{eq:control X1}, the left-hand side of \eqref{eq:ibp vect} is controlled as follows (not the absence of squre in the denominator)
$$\frac{1}{\|r_e^3Y^3\|_{L^\infty(e)}}\left|\int_M\delta_{\bar{g}^B}(\tilde{\mathbf{o}}_k)(\mathcal{X}_1)\right|\to 0\qquad \text{ as } t\to 0.$$
Since on the finite-dimensional space of homogeneous harmonic vector fields decaying like $r^3$, one has $C^{-1}\|Y^3\|^2\leqslant \|r_e^3Y^3\|_{L^\infty(e)}\leqslant C\|Y^3\|^2$, we conclude that $\|r_e^3Y^3\|_{L^\infty(e)}\to 0$ as $t\to 0$.

In particular, $|\delta_{e}^*Y^3(\partial_{r_B},\partial_{r_B})|\leqslant \varepsilon(t ,v ^B) r^{-4}$ where $\varepsilon(t ,v ^B)\to 0$ as $(t ,v ^B)\to 0$. By \eqref{eq:exp pi ok}, 
\begin{equation}\label{eq:dvp pi g ok}
\pi_{\bar{g}_B}\tilde{\mathbf{o}}_k(\partial_{r_B},\partial_{r_B}) - \tilde{\mathbf{o}}_k(\partial_{r_B},\partial_{r_B}) = \varepsilon(t ,v ^B)r_B^{-4} + o(r_B^{-4}),
\end{equation}
 Together with \eqref{eq: asympt uB}, this leads to the desired statement.
\end{proof}

\begin{remark}
    We cannot prevent the presence of new $r^{-4}$ terms since there may be terms $Y^3$ in $r^{-3}$ appearing in the vector field obtained by $(\delta_{\bar{g}_B}\delta_{\bar{g}_B}^*)^{-1}\delta_{g^B}$ when gauge-fixing our tensors $\tilde{\mathbf{o}}_k$ thanks to $\pi_{\bar{g}_B}$. An example of such a tensor $Y^3$ is $ \nabla_{e}(r^{-2}) = -2r^{-3}\partial_r$ since $\frac1{2}\Hess_{e}(r^{-2}) = \frac{3dr^2 - r^2g_{\mathbb{S}^3}}{r^4}$ where $g_{e} = dr^2 + r^2g_{\mathbb{S}^3}$. 
\end{remark}

\subsection{Construction of the approximate metric and estimates}\label{sec:approx metric}

We now construct a good approximation $g^A= g^A_{t,v}$ of the Einstein metric $g $ close to $g^D_{t }$ in order to apply Corollary \ref{cor: quantitative uniqueness desing} and to contradict \eqref{eq:control approx obst} when the orbifold is spherical or hyperbolic. We will perturb our trees of singularities $g^B$ to make them match the orbifold metric $g_o$ better, then estimate the errors.

Close to a singular point $p_o\in M_o$, there exist coordinates in which one has 
\begin{equation}\label{eq:dvp go}
    g_o = g_{e} + H_2 + \mathcal{O}(r^{3}),
\end{equation}
where $H_2$ is a homogeneous quadratic tensor with $\delta_{e}H_2=0$ and $\bar{P}_eH_2 + \lambda g_{e}=0$ if $\E(g_o)+\lambda g_o = 0$. Such coordinates are classically obtained by gauge-fixing normal coordinates. 
\begin{prop}\label{prop:extension quadratic}
    There exists a symmetric $2$-tensor $h_2$ and $u_o\in \tilde{\mathbf{O}}(g^B)$ such that:
    \begin{equation}\label{eq: ext quadratic terms}
        \left\{\begin{aligned}
            &\bar{P}_{\bar{g}^B}h_2 +\lambda \bar{g}^B ={\pi_{\bar{g}^B}}u_o^{[2]}\in {\pi_{\bar{g}^B}}\tilde{\mathbf{O}}(g^B)\\
            &h_2 = H_2 + \mathcal{O}(r_B^{-2+\varepsilon}) \text{ for all $\varepsilon>0$.}\\
        \end{aligned}\right.
    \end{equation}

    Additionally, if $H_2$ comes from a metric $g_o$ which is spherical or hyperbolic with $\Ric(g_o)=\Lambda g_o$, $\Lambda\in \mathbb{R}\backslash\{0\}$, then we have the estimate: for $c\neq 0$, one has
    \begin{equation}\label{eq:est obst tree}
        \left\|u_o^{[2]} - c\,|\Gamma|\,\Lambda \,\tilde{\mathbf{o}}_1 \right\|_{r_B^{-2}C^{\alpha}_\beta}\to 0, \text{ as $(t,v^B)\to 0$.}
    \end{equation}
\end{prop}
\begin{proof}
    Consider $\chi$ a cut-off function supported in a region where $\bar{g}^B$ has ALE coordinates, then $\bar{P}_{\bar{g}^B}(\chi H_2)+\lambda \bar{g}^B\in r_D^{-2}C^{\alpha}_{\beta}$ for $\beta>0$, so there exists a couple of elements $(h',u_o^{[2]})\in (C^{2,\alpha}_{\beta,*}\cap \tilde{\mathbf{O}}(g^B)^\perp)\times \tilde{\mathbf{O}}(g^B)$ such that 
    $\bar{P}_{\bar{g}^B}(\chi H_2)+\lambda \bar{g}^B = \bar{P}_{\bar{g}^B}h' + {\pi_{\bar{g}^B}}u_o^{[2]},$ and defining $h_2 = \chi H_2-h'$, we solve \eqref{eq: ext quadratic terms}.

    Let us now focus on the estimate \eqref{eq:est obst tree}. This relies on the following integration by parts: as in \cite[Proposition 5.9]{ozu2}, one has
    \begin{equation}\label{eq:IBP P H2}
        \int_N \langle \bar{P}_{\bar{g}_B} (\chi H_2),\pi_{\bar{g}^B}\tilde{\mathbf{o}}_k \rangle_{\bar{g}_B} = \int_N \langle \chi H_2,\bar{P}_{\bar{g}_B}(\pi_{\bar{g}^B}\tilde{\mathbf{o}}_k) \rangle_{\bar{g}_B} + Q(H_2, O_k^4),
    \end{equation}
where $O_k^4$ is the homogeneous harmonic term such that $\pi_{\bar{g}^B}\tilde{\mathbf{o}}_k = O^4_k + o(r^{-4})$ at infinity, and $Q(H_2,O_k^4) = \int_{\mathbb{S}^3/\Gamma}(3\langle H_2,O_k^4 \rangle+O_k^4(\nabla_e\tr_eH_2,\partial_r))dv_{\mathbb{S}^3/\Gamma}$, where $\bar{g}^B$ is asymptotic to $\mathbb{R}^4/\Gamma$ at infinity.

Now a combination of \cite[Proposition 11.18]{do3} and \cite[(86)]{ozu2} shows that there exists a constant $c\neq 0$ such that if $H_2$ comes from a spherical or hyperbolic orbifold with $\Ric(g_o)=\Lambda g_o$ (or equivalently $\E(g_o)+\Lambda g_o = 0$ in dimension $4$), then one has 
$$ Q(H_2,O_k^4) = -2c\Lambda \,O_k^4(\partial_r,\partial_r). $$

Now, from \eqref{eq:dvp pi g ok} and \eqref{eq: asympt uB}, we know that if $k \neq 1$, then $r^4O_k^4(\partial_r,\partial_r)\to 0$ as $r\to+\infty$, while $r^4O_1^4(\partial_r,\partial_r)\to \frac{|\Gamma|}{|\mathbb{S}^3|} \|\mathbf{o}_1\|^2_{L^2(g_b)}$ as $r\to+\infty$.
\end{proof}

We can further expand the metrics $g_o$ and $\bar{g}^B$ in divergence-free coordinates respectively at $p_o$ and at infinity to define the homogeneous tensors $H_3$, $H_4$ and $H^4$ on Euclidean space, 
\begin{equation}
\begin{aligned}
    &g_o = g_{e} + H_2+H_3+H_4+\mathcal{O}(r_e^5),\\
    &\bar{g}^B = g_{e}+H^4+\mathcal{O}(r_e^{-5}).
\end{aligned}
\end{equation}
these tensors are homogeneous, $|H_k|_e\propto r_e^k$ and $|H^4|_e\propto r_e^{-4}$, and satisfy the equations:
\begin{equation}\label{eq: homog tensors}
    \begin{aligned}
    &\bar{P}_{e} H_2 + \lambda g_{e} = 0,\\
    &\bar{P}_{e} H_3  = 0,\\
    &\bar{P}_{e} H_4 + Q_{e}(H_2,H_2)+ \lambda H_2 = 0, \text{ and}\\
    &\bar{P}_{e} H^4 = 0,
\end{aligned}
\end{equation}
where $Q_g$ denotes the quadratic terms in the expansion $$h\mapsto \mathbf{\Phi}_g(g+h) = \bar{P}_g(h) + Q_g(h,h) + R_g(h),$$
where $R_g(h)$ contains the remaining cubic and higher order terms, see \cite[Section 14]{biq1}. The equations \eqref{eq: homog tensors} simply correspond to the expansions of $\Ric(g_o) = \Lambda g_o$ at $p_o$ and $\Ric(\bar{g}^B)=0$ at infinity.

The following result is a simplified version of \cite[Sections 2.1.1 and 2.1.2]{ozu3} and \cite[Proposition 11.10] {do3}, and follows \cite[Section 14]{biq1}. Its proof is very similar to that of Proposition \ref{prop:extension quadratic}.

\begin{prop}\label{prop:extension tensors}
    There exist tensors $h_3$, $h_4$ and $h^4$ satisfying the following properties:
    \begin{equation}\label{eq: ext cubic terms}
        \left\{\begin{aligned}
            &\bar{P}_{\bar{g}^B}h_3={\pi_{\bar{g}^B}}u_o^{[3]}\in {\pi_{\bar{g}^B}}\tilde{\mathbf{O}}(g^B)\\
            &h_3 = H_3 +  \mathcal{O}(r_B^{-1+\varepsilon}) \text{ for all $\varepsilon>0$,}\\
        \end{aligned}\right.
    \end{equation}
    \begin{equation}\label{eq: ext quartic terms}
        \left\{\begin{aligned}
            &\bar{P}_{\bar{g}^B}h_4 + Q_{\bar{g}^B}(h_2,h_2)+\lambda h_2 ={\pi_{\bar{g}^B}}u_o^{[4]} \in {\pi_{\bar{g}^B}}\tilde{\mathbf{O}}(g^B)\\
            &h_4 = H_4 + \mathcal{O}(r_B^{\varepsilon}) \text{ for all $\varepsilon>0$,}\\
        \end{aligned}\right.
    \end{equation}
\begin{equation}\label{eq: ext terms infty}
        \left\{\begin{aligned}
            &\bar{P}_{g_o}h^4 =0 \text{ if }g_o \text{ is spherical or hyperbolic}\\
            &h^4 = H^4 + \mathcal{O}(r_b^{-2-\varepsilon}) \text{ for all $\varepsilon>0$,}\\
        \end{aligned}\right.
    \end{equation}
\end{prop}

\begin{remark}
    The above tensors are introduced for two reasons:
\begin{enumerate}
    \item to make the bubble metric and the orbifold metric match to a higher order in the neck region where we interpolate between them, this is why they match the asymptotic expansions of the metrics, and
    \item not to perturb the Einstein equation too much, this is why we have the tensors solve the above equations, see the resulting controls in Corollary \ref{cor: est Phi regions}.
\end{enumerate}
\end{remark}
Let us number $p_o^1,\dots,p_o^N\in S_{o}$ the singular points of $g_o$ to be desingularized and denote $(\bar{g}^B_1, u^B_1),\dots,(\bar{g}^B_{N},u^B_N)$ the Ricci-flat modulo obstructions metrics of Proposition \ref{prop: gluing trees} obtained from the associated perturbed trees of singularities. Up to renumbering the gluing scales $t_1 = t_{p^{-1}(p_o^1)},\,\dots,\, t_N = t_{p^{-1}(p_o^N)}$ to $g_o$, we assume
$$t_1\leqslant\dots\leqslant t_N.$$ We also denote $(h_{2,1},u_{o,1}),\dots, (h_{2,N},u_{o,N})$ and $(h_{3,1},u_{o,1}^{[3]},h_{4,1},u_{o,1}^{[4]},h^4_1),\dots,(h_{3,N},u_{o,N}^{[3]},h_{4,N},u_{o,N}^{[4]},h^4_N)$ the $2$-tensors constructed in \eqref{eq: ext quadratic terms} and Proposition \ref{prop:extension tensors}. Consider a cut-off function $\chi_{o,j}$ supported on of $\{r_o<2\varepsilon_0\}\cap B(p_j,1)$, equal to $1$ on $\{r_o<\varepsilon_0\}\cap B(p_j,1)$ and with derivatives controlled as $|\nabla_{g_o}^l\chi_{o,j}|\leqslant C_l r_o^{-l}$ for $\varepsilon_0>0$ of Definition \ref{orb Ein} and $C_l>0$. 

We have the following estimates.
\begin{corollary}\label{cor: est Phi regions}
    Define the metrics and obstructions perturbed to $4^{th}$ order
    \begin{equation}\label{eq:approx on Mo}
        g_o^{[4]}:= g_o + \sum_{j=1}^N t_j^2\,\chi_{o,j}h^4_j,\quad \text{ and }
    \end{equation}
\begin{equation}\label{eq:approx on bargB}
        \bar{g}^{B,[4]}_j:= \bar{g}^{B}_j+t_j \,h_{2,j}+t_j^\frac32\, h_{3,j}+t_j^2\, h_{4,j}, \qquad \bar{w}^{B,[4]}_j:= u_B - t_j u_{o,j}^{[2]} -t_j^\frac32 u_{o,j}^{[3]}-t_j^2 u_{o,j}^{[4]}.
    \end{equation}
    They satisfy the controls: for all $l\in\mathbb{N}$, there exists $C_l>0$ depending on the above tensors such that
    \begin{equation}\label{eq:control Ric go4}
        r_B^{2+l}\left|\nabla^l\left(\mathbf{\Phi}_{g^B_j}(\bar{g}^{B,[4]}_j) + \pi_{g^B_j}\bar{w}^{B,[4]}_j \right)\right|_{g^B_j} \leqslant C_l\,\,t_j^\frac{5}{2}\,r_B^5,
    \end{equation}
    \begin{equation}\label{eq:control Ric bargB4}
        r_o^{2+l}\left|\nabla^l\left(\mathbf{\Phi}_{g_o}(g_o^{[4]})+\lambda g_o^{[4]}\right)\right|_{g_o} \leqslant C_l \left(t_j^4r_o^{-8} + t_j^2\, 1_{\{\varepsilon_0<r_o<2\varepsilon_0\}}\right) \text{ on } B(p_j,3\varepsilon_0).
    \end{equation}
    where $1_{A}$ is the indicator function of a set $A$.
\end{corollary}
\begin{proof}
    
Let us drop the index $j$ for ease of notation. The expansion of $g\mapsto\mathbf{\Phi}_{\bar{g}^{B}}(g)$ close to $g = \bar{g}^B$ in the direction $h^{[4]}:=\bar{g}^{B,[4]}-\bar{g}^{B} =t\,h_2+t^\frac32\,h_3+t^2h_4$ yields
$$\mathbf{\Phi}_{\bar{g}^{B}}(\bar{g}^{B,[4]}) = \mathbf{\Phi}_{\bar{g}^{B}}(\bar{g}^{B}) + \bar{P}_{\bar{g}^{B}}(h^{[4]}) + Q_{\bar{g}^{B}}(h^{[4]},\,h^{[4]}) + R_{\bar{g}^{B}}(h^{[4]}), $$
and using the definitions of the tensors in \ the worst terms come from $r_B^{2+l}|\nabla^l_{\bar{g}^{B}}Q_{\bar{g}^{B}}(t\,h_2,t^\frac32\,h_3)| \leqslant t^{\frac52}r_B^5$. 

We similarly control $g_o^{[4]}$, where the worst terms come from:
$r_o^{2+l}|\nabla^l_{g_o}Q_{g_o}(t^2\,h^4\,,\,t^2\,h^4)| \leqslant t^4r_o^{-8}$
and the cut-off of $t^2h^4$ by $\chi_{o}$ which  satisfies $\varepsilon_0^{l}|\nabla_{g_o}^l\chi_{o}|\leqslant C_l$ by definition.
\end{proof}

We define the \textit{approximate Einstein modulo obstruction} $(g^A,w^A)$ iteratively as follows, using the notations of Definition 2.6.
\begin{defn}[Approximate Einstein modulo obstruction metric]\label{defn:approximate metric}
     We again assume that we ordered the scale so that $t_1\leqslant\dots\leqslant t_N$.
\begin{enumerate}
    \item Define $(g^A_0,w^A_0) = (g_o,0)$,
    \item For $j\geqslant 0$, define (using the notation $\#$ of Definition \ref{def naive desing}) $(g^A_{j+1},w^A_{j+1})$ by
    \begin{equation}
    \left\{\begin{aligned}
        g^A_{j+1} &:= \left(g^A_{j}+t_{j+1}^2\,\chi_{o,j}\,h^4_{t_{j+1}}\right)\,\,\#_{p_{j+1},t_{j+1}} \,\,g^{B,[4]}_j\\
        w^A_{j+1} &:= w^A_j\,\,\#_{p_{j+1},t_{j+1}}\,\,\left(u^B_{j+1}+u_o^{j+1}\right).
    \end{aligned}\right.
\end{equation}
\item Denote $g^A = g^A_N$.
\end{enumerate}
On $M_o^{16t}$, one has $g^A = g_o^{[4]}$.
\end{defn}

We finally measure how good an approximation our metrics $g^A_{t ,v }$ are.

\begin{prop}\label{prop:estimate approx}
    For $0<\beta<1$, we have 
    \begin{equation}\label{eq:control phi  approx}
        \left\|\mathbf{\Phi}_{g^D_{t }}(g^A) + \pi_{g^A}w^A \right\|_{r_D^{-2}C^\alpha_\beta}\leqslant C t_N^{\frac{5-\beta}4} = o(t_N).
    \end{equation}
\end{prop}
\begin{remark}
    The key point here are that the control is in $t_N$ rather than in $t_{\max}$. It does not see the relative scales deep in the trees of singularities, and this is the reason why we need to perturb these trees of Ricci-flat ALE metrics into a \emph{single} Ricci-flat modulo obstruction metric in Section \ref{sec:gluing trees}.
\end{remark}
\begin{proof}
    Thanks to Corollary \ref{cor: est Phi regions}, there only remains to control the transition regions in order to prove \eqref{eq:control phi  approx}. In such a region with $t_j^\frac14<r_D<2t_j^\frac14$, $g^A$ interpolates between the following metrics satisfying for all $\varepsilon>0$,
    \begin{equation}
        \Phi^*_jg_o^{[4]} = g_{e} + H_2+ H_3+H_4 + t_j^2 H^4 + \mathcal{O}(t_j^2r_o^{-2-\varepsilon} + r_o^5) \qquad \text{ and }
    \end{equation}
    \begin{equation}
t_j\phi_{t_j^{-1/2}}\Psi^*_j(\bar{g}^{B,[4]}) =  g_{e} + t_j^2H^4+H_2+H_3+H_4+ \mathcal{O}(t_j^2r_o^{-2} + t_j^{\frac52}r_o^{-5}).
    \end{equation}
    In particular, the difference between the metrics, $\mathcal{O}(t_j^2r_o^{-2+\varepsilon} + t^{\frac52}r_o^{-5} + t_j^2r_o^{-2-\varepsilon} + r_o^5) = \mathcal{O}(t_j^\frac54)$ in the region $t_j^\frac14<r_D<2t_j^\frac14$ for small enough $\varepsilon>0$. Together with the analogous controls of the derivative, and the controls of the cut-off functions, this yield the stated result.
\end{proof}

\subsection{Isolation of hyperbolic and spherical orbifolds}\label{sec: isolation hy sph orb}

We finally show that there cannot be any other Einstein orbifold $GH$-close to a spherical or hyperbolic orbifold. The argument is by contradiction and consists in proving that $w^A$ cannot be small, contradicting \eqref{eq:control approx obst}.

\begin{proof}[Proof of Theorem \ref{thm:isolation orbifolds}]

Let $(M_o,g_o)$ be a spherical or hyperbolic orbifold, and assume that, a sequence of Einstein \textit{orbifolds} $(M_i,g_i)_{i\in\mathbb{N}}$ satisfies
$ d_{GH}\big((M_i,g_i),(M_o,g_o)\big) \to 0.$
Since $(M_o,g_o)$ is a spherical or hyperbolic orbifold, we have $\mathbf{O}(g_o) = \{0\}$, and it is isolated as an Einstein orbifold on $M_o$. Hence, up to taking a subsequence, we assume $M_i=M\neq M_o$, and there is indeed a nontrivial bubble-tree degeneration. 

By Theorem \ref{thm: exhaustion neighb orb}, for all $i$, $(g_i,0)$ is isometric to an Einstein (modulo obstruction) perturbation $\hat{g}_{t_i,v_i}$ of a naïve gluing $g^D_{t_i}$ with $v_i\in \tilde{\mathbf{O}}(g^D_{t_i})$ the metric and elements of the approximate kernel defined in Section \ref{sec:compactness dim 4}. Up to taking yet another subsequence, we can order the $N$ singular points of $M_o$ as $p_1,...,p_N$ such that
$$0<t_{1,i}\leqslant t_{2,i} \leqslant... \leqslant t_{N,i},$$
where $t_{k,i}$ is the scale of the gluing at $p_k$ in the construction of $g^A_{t_i,v_i}$ in Definition \ref{defn:approximate metric}.

    Thanks to the control of Proposition \ref{prop:estimate approx} and \eqref{eq:control approx obst}, we deduce that for any $0<\beta<1$ we have $\|w^A\|_{L^2}\leqslant C t_{N,i}^{\frac{5-\beta}4} = o(t_{N,i})$, hence denoting $\mu_k$, and $\tilde{\mathbf{o}}_k$ the elements of Section \ref{sec:gluing trees} for the last metric $\bar{g}^B_{t_i,v_i,N}$ glued in the construction of Definition \ref{defn:approximate metric}, we have the following decomposition in the orthonormal basis $(\tilde{\mathbf{o}}_k)_k$ of $\tilde{\mathbf{O}}(g^D_t)$ 
\begin{align*}
    w^A|_{N_N^{16t}}&= u_{B_N,i} - t_{N,i} u_{o,N}^{[2]} -t_{N,i}^\frac32 u_{o,N}^{[3]}-t_{N,i}^2 u_{o,N}^{[4]}\\
    &= \big(\mu_1-t_{N,i}\lambda_1 + o(t_{N,i})\big) \tilde{\mathbf{o}}_1 + \sum_k\big(\mu_k+o(t_{N,i})\big)\tilde{\mathbf{o}}_k.
\end{align*}
Now, in order to have $\|w^A|_{N_N^{16t}}\|_{L^2}=o(t_{N,i})$, since $\|\tilde{\mathbf{o}}_1\|_{L^2(g^D)}>0$ is bounded away from $0$ independently on $(t_i,v_i)$, this imposes 
    \begin{equation}\label{eq: obst 1}
        \mu_1 = t_{N,i}\lambda_1 + o(t_{N,i}), \text{ and}
    \end{equation}
    \begin{equation}\label{eq: obst k}
        \mu_k = o(t_{N,i}), \text{ for } k\geqslant2
    \end{equation}
    The first obstruction \eqref{eq: obst 1} and \eqref{eq: estimation tree obstructions} imply that 
    $$\sum_{k\geqslant 0}\varepsilon_k(t_i,v_i^B)  |\mu_k| = t_{N,i}\lambda_1 + o(t_{N,i}).$$
    Since from \eqref{eq:est obst tree}, we have $\lambda_1 \neq 0$ for hyperbolic and spherical orbifolds, this contradicts \eqref{eq: obst k}.
\end{proof}

\begin{proof}[Proof of Corollary \ref{cor:gap thms}]
    
 The proofs of the gap theorems have a similar classical strategy by contradiction:
\begin{enumerate}
 \item take a sequence of Einstein orbifolds $(M_i,g_i)$ with $\Ric(g_i) = 3g_i$, $\Vol(M_i,g_i)>v>0$ and assume one of the conditions (1)-(5) with $\varepsilon_0$ replaced by $\frac{1}{i}$,
 \item take a sublimit $(M_o,g_o)$ by compactness, and show that it must be a spherical orbifold, and
 \item conclude by Theorem \ref{thm:isolation orbifolds} that for large enough $i$, all of the $(M_i,g_i)$ are isometric to $(M_o,g_o)$.
\end{enumerate}
\end{proof}

\part{Regularity of Einstein $5$-manifolds}

In this part, we now use our $4$-dimensional isolation result to prove new regularity properties for $5$-manifolds with bounded Ricci curvature.

\section{Uniqueness of tangent cones and the refined rectifiability}

In this section, we consider the following setting: Let $(M^5_i, g_i, p_i) \xrightarrow{pGH} (X^5,d,p)$ be a sequence of noncollapsed pointed manifolds with uniformly bounded Ricci curvature
\begin{align*}
    |\Ric_i | \leqslant 4, \qquad \text{ and }\qquad
    \Vol(B_1(p_i) > v >0.
\end{align*}

Our goal is to analyze the singular structure of the noncollapsed limit $(X^5,d,p)$, in particular to establish the uniqueness of tangent cones along the top stratum $\cS^1\setminus\cS^0$ and to refine the rectifiability of the singular set $\cS(X)$. We also prove regularity results for the asymptotic geometry of Ricci-flat manifolds with Euclidean volume growth.

\subsection{Preliminaries}

We recall the standard stratification framework for $n$-dimensional noncollapsed Ricci-bounded limit spaces and several quantitative tools that will be used throughout the paper. 

Let $(M^n_i, g_i, p_i) \xrightarrow{pGH} (X^n,d,p)$ be a sequence of noncollapsed pointed manifolds with uniformly bounded Ricci curvature $|\Ric_i | \leqslant n-1,$ and $\Vol(B_1(p_i)) > v >0$. We call such $X$ as a \textit{Ricci-bounded limit space}.

\begin{definition}
    A metric space $Y$ is called \emph{$k$-symmetric} at $y$ with respect to $\cL^k$ if there exists some pointed isometry $\iota: (\RR^k \times C(Z), (0,z)) \to (Y,y)$ for some metric cone $C(Z)$ over a compact metric space $Z$ with vertex $z$ and $\cL^k = \iota(\RR^k \times \{z\})$.
\end{definition}

Following \cite{cc97}, we introduce the stratification and define the $k^{th}$-stratum to be 
\begin{equation*}
    \cS^k \equiv \{ x\in X ~|~ \text{ no tangent cone at } x \text{ is } (k+1) \text{-symmetric}. \}
\end{equation*}

\begin{remark}
    We will later refine the stratification in dimension $n=5$ using the uniqueness of the splitting tangent cones; see \eqref{e:refined stratification}. 
\end{remark}

The \emph{regular set} $\cR(X)$ consists of points whose tangent cones are uniquely $\mathbb{R}^n$, and the \emph{singular set} is $\cS(X)=X\setminus \cR(X)$. Combining the pioneering work \cite{col, cc97} and the resolution of codimension four conjecture by Cheeger-Naber \cite{cn}, we have 
\begin{theorem}[\cite{cn}]\label{t:codim four}
    The singular set $\cS(X^n)$ satisfies $\cS = \cS^{n-4}$ and $\dim(\cS) \leqslant n-4$. 
\end{theorem}

In \cite{cjn} and \cite{jn}, they proved a structure theorem for the singular set $\cS$:

\begin{theorem}\cite{cjn,jn}\label{t:cjn}
Let $X^n$ be a noncollapsed Ricci-bouned limit space. Then:
\begin{enumerate}
    \item For \emph{$\cH^{n-4}$-a.e.} $x \in \cS(X)$, the tangent cone at $x$ is \emph{unique} and isometric to the cone $\RR^{n-4} \times C(\mathbb{S}^3/\Gamma)$.
    \item The singular set $\cS$ is \emph{$(n-4)$-rectifiable}: there exists a countable collection of $\cH^{n-4}$-measurable subsets $Z_i \subset Z$ and bi-Lipschitz maps $\phi_i : Z_i \to \RR^{n-4}$ such that $\cH^{n-4}(\cS \setminus \bigcup_i Z_i)=0$.
\end{enumerate}
\end{theorem}

Both the uniqueness of tangent cones and the bi-Lipschitz structure hold for the top stratum $\cS^{n-4}$ modulo an \emph{$\cH^{n-4}$-measure zero subset}. In dimension five we will strengthen this in the next two sections, by proving uniqueness and bi-Lipschitz structure on the entire top stratum, without discarding any measure-zero subset.

Next we recall the quantitative symmetry:

\begin{definition}
    Given a metric space $X$ and $r>0, \varepsilon > 0$, we define a ball $B_r(x) \subset X$ as \emph{$(k, \varepsilon)$-symmetric} with respect to $\cL_{x,r}$ if $d_{GH}(B_r(x), B_r(0,x')) < \varepsilon r$ for some $k$-symmetric metric cone $\RR^k \times C(Z)$ with vertex $(0,x')$, and that $\cL_{x,r} = \iota(\RR^k \times \{x'\} \cap B_r(0,x'))$ where $\iota : B_r(x') \to B_r(x)$ is the $\varepsilon r$-Gromov-Hausdorff map.
\end{definition}

One of the fundamental tools to study the structure of noncollapsed spaces with uniform Ricci curvature lower bound is the Bishop-Gromov monotonicity: for any $(M^n,g)$ with $\Ric_M \ge -\kappa g$, the volume ratio defined as follows is non-increasing in $r$:
\begin{equation*}
    \cV_x(r) = \cV^{\kappa}_x(r) \equiv \frac{\Vol(B_r(x))}{\Vol_{-\kappa}(B_r(0))},
\end{equation*}
where $\Vol_{-\kappa}(B_r(0))$ is the volume of the ball with radius $r$ in the space form with constant curvature $-\kappa$. We define the volume density at $x$ as $\lim_{r\to 0}\cV_x(r)$. When the volume ratio is constant, the space is isometric to the model space. In \cite{cc}, they proved the almost rigidity:
\begin{theorem}\label{t:almost volume cone}
    For any $\varepsilon>0$ and $\delta\leqslant \delta(n,\varepsilon)$, the following holds. If $(M^n,g)$ is a Riemannian manifold with $\Ric_M \ge -\delta g$, and $\cV_x(2) \ge (1-\delta) \cV_x(1) $, then $B_2(x)$ is $(0,\varepsilon)$-symmetric.
\end{theorem}

When multiple $0$-symmetries occur, higher symmetry follows from the quantitative cone-splitting principle of \cite{chn13}.
\begin{lemma}\label{l:cone splitting}
    Let $\varepsilon,\tau>0$ and $\delta\leqslant \delta(\varepsilon,\tau)$. Assume $\Ric_M \ge -\delta g$ and
    \begin{enumerate}
        \item $B_2(p)$ is $(k,\delta)$-symmetric with respect to $\cL_{x,2} \subset B_2(p)$.
        \item there exists some $x \in B_1(p) \setminus B_{\tau}(\cL_{x,2})$ such that $B_2(x)$ is $(0,\delta)$-symmetric.
    \end{enumerate}
    Then $B_1(p)$ is $(k+1,\varepsilon)$-symmetric.
\end{lemma}

Combining Lemma~\ref{l:cone splitting} with the codimension-four theorem \ref{t:codim four} yields the $\varepsilon$-regularity theorem of \cite{cn}.

\begin{theorem}[\cite{cn}]\label{t:eps_regularity}
    For all $n\in\mathbb{N},\,\,v>0$, there exists $\varepsilon_0(n,v)>0$ such that if $\Vol(B_1(p)) > v >0$, $|\Ric_{M^n}|\leqslant n-1$ and $B_2(p)$ is $(n-3,\varepsilon)$-symmetric with $\varepsilon \leqslant \varepsilon_0$, then the harmonic radius $r_h(p)$ at $p$ is greater than 1. If $M^n$ is Einstein, then we have the curvature bound $\sup_{B_1(p)} |\Rm| \leqslant 1$.  
\end{theorem}

According to the $\varepsilon$-regularity, if a ball has enough symmetry, then it is regular enough. Consult \cite{and90} or \cite{jn} for the definition of harmonic radius.

Finally, we recall the notion of \textit{$\varepsilon$-splitting function}, which is crucial to establish the bi-Lipschitz structure for the singular set.

\begin{definition}
    Let $B_r(p) \subset M$. We say $u: B_r(p) \to \RR$ is a \textit{$\varepsilon$-splitting function} if the following holds:
    \begin{enumerate}
        \item $\Delta u = 0$;
        \item $\fint_{B_r} \big| |\nabla u|^2 - 1 \big| \leqslant \varepsilon$;
        \item $\sup_{B_r(p)} |\nabla u| \leqslant 1+\varepsilon$;
        \item $r^2 \fint_{B_r(p)} |\nabla^2 u|^2 \leqslant \varepsilon$.
    \end{enumerate}
\end{definition}

The existence of $\varepsilon$-splitting function is equivalent to the metric splitting behavior of the ball. See \cite{cc} and also \cite{cjn}.

\begin{theorem}\label{t:almost splitting}
    Let $(M^n,g)$ be a smooth Riemannian manifold with $\Ric_{M^n} \ge -\delta$. For $\varepsilon>0$, if $\delta \leqslant \delta(n,\varepsilon)$, then we have the following
    \begin{enumerate}
        \item If $d_{GH}(B_2(p), B_2(0,x')) \leqslant \delta$, with $(0,x') \in \RR \times X$ for a metric space $X$,  then there exists an $\varepsilon$-splitting function on $B_1(p)$;
        \item If there exists a $\delta$-splitting function on $B_2(p)$, then $d_{GH}(B_1(p),B_1(0,x')) \leqslant \varepsilon$ for $(0,x') \in \RR \times X$. Moreover, $(u,v) : B_1(p) \to \RR \times X$ gives an $\varepsilon$-Gromov-Hausdorff map for some $v:B_1(p) \to X$.
    \end{enumerate}
\end{theorem}

We also record the convergence of limiting harmonic functions. The result holds under more general assumptions. See \cite{cjn,jn} and also \cite{ah18,aht18,che99,din02,gms15,mn19,zz19}.

\begin{proposition}\label{p:convergence splitting function}
    Let $(M^n_i,g_i,p_i,\mu_i) \to (X,d,p,\mu)$ satisfy $\Ric_{M^n_i} \ge n-1$ and $\Vol(B_1(p_i)) > v$ and $\mu_i = |\Vol(B_1(p_i)|^{-1}\Vol$. Assume that $u_i: B_r(p_i) \to \RR$ is harmonic with $\sup_{B_r(p_i)}|\nabla u_i| \leqslant C$ for some $r>0$ and $C>0$. Then passing to a subsequence there exists some a harmonic function $u: B_r(p) \to \RR$ satisfying 
    \begin{enumerate}
        \item $u_i \to u$ uniformly on compact subsets of $B_r$. 
        \item $|\nabla u(x)| \leqslant \liminf_i ||\nabla u||_{L^{\infty}(B_s(x_i))}$ for each $x_i \to x \in X$ with $B_{2s}(x_i) \subset B_r(p_i)$. 
        \item $u_i \to u$ in $W^{1,p}$-sense for any $1<p<\infty$.
    \end{enumerate}
\end{proposition}

Using the $W^{1,2}$ convergence of splitting functions, we can also define the $\varepsilon$-splitting function on the limit space $(X,d,p)$. We say $u$ is a $\varepsilon$-splitting function on $B_r(p) \subset X$ if there exist $B_r(p_i) \subset M_i \to B_r(p)$ and $\varepsilon_i$-splitting functions on $B_r(p_i)$ converging uniformly to $u$ with $\varepsilon_i \to \varepsilon$. Therefore, Theorem \ref{t:almost splitting} also holds for the limit space $(X,d,p)$.

\subsection{Uniqueness of tangent cones}\label{subs:uniqueness}

Now we focus on dimension $n=5$. Let $X^5$ be a noncollapsed Ricci-bounded limit space. Then the singular set $\cS(X^5) = \cS^1$. We will prove the uniqueness of tangent cones at the points in $\cS^1 \setminus \cS^0$ and improve the regularity estimates around those points. Similar results hold for tangent cones at infinity.

\begin{proof}[Proof of Theorem \ref{thm:isolation 1sym}]
    Suppose we have the noncollapsing convergence
    \begin{equation*}
    (M_i, g_i, p_i) \xrightarrow{pGH} (C(Y), d_{C(Y)}, y) \quad \text{ and }  \quad (N_i, \tilde{g}_i, q_i) \xrightarrow{pGH} (C(Z), d_{C(Z)}, z).
\end{equation*}
with both satisfying $|\Ric_i| \to 0$. Since we assume the cone $C(Z)$ splits off a line: $C(Z) \cong \RR \times C(Z')$, by codimension four theorem \ref{t:codim four}, we conclude that $Z'$ is indeed a smooth Einstein $3$-manifold and thus $Z' \cong \mathbb{S}^3/\Gamma$ for some $\Gamma \subset O(4)$. This implies that $Z$ is a sine suspension over $\mathbb{S}^3/\Gamma$ and thus a spherical Einstein $4$-orbifold with isolated singularities. Also by codimension four theorem \ref{t:codim four}, we have $Y$ is an Einstein $4$-orbifold. Since we assume $d_{GH}(Y,Z) < \varepsilon_0$, by the gap theorem \ref{thm:isolation orbifolds}, we conclude that $Y \cong Z$ provided $\varepsilon_0$ is small enough. This completes the proof.
\end{proof}

As a consequence, we prove the uniqueness of the tangent cone that splits off a line. In fact, this is a direct consequence of Theorem~ \ref{thm:isolation 1sym} and the fact that the space of the cross sections of the tangent cones at some fixed point is connected.  

\begin{proof}[Proof of Theorem \ref{t:local uniqueness} and Theorem \ref{t:uniqueness at infinity}]
    We prove the local uniqueness theorem \ref{t:local uniqueness}, and the proof for the uniqueness of tangent cones at infinity \ref{t:uniqueness at infinity} proceeds verbatim. Pick any $x \in \cS^1 \setminus \cS^0$. By definition there exists one tangent cone $C(Z_0)$ at $x$ that splits off a line. Let $C(Z_1)$ be a tangent cone under another rescaling sequence. Let $\cC_x \equiv \{ Y: C(Y)$ is a tangent cone at $x \}$. Since $\cC_x$ is known to be connected, there exist finitely many compact metric spaces $\{Y_i\}_{i=0}^N$ such that
    \begin{enumerate}
        \item $Y_0 \cong Z_0$ and $Y_N \cong Z_1$;
        \item $d_{GH}(Y_i, Y_{i+1}) \leqslant \varepsilon_0/2$ for any $0\leqslant i \leqslant N-1$, where $\varepsilon_0$ is the gap in Theorem \ref{thm:isolation 1sym}.
     \end{enumerate}

By iterating Theorem \ref{thm:isolation 1sym}, we have $Y_i \cong Y_{i+1}$ for each $i$ and thus $Z_0 \cong Z_1$. This proves the uniqueness of the tangent cones at $x$. The second result follows since the $0$-stratum $\cS^0$ is countable. 
\end{proof}

Next we prove the curvature estimates around each point $x \in \cS^1 \setminus \cS^0$. 

Fix any $x \in \cS^1 \setminus \cS^0$. By uniqueness Theorem \ref{t:local uniqueness}, for each $r \leqslant r_x$ for some small $r_x$, we have $B_{r}(x)$ is $(1,\varepsilon)$-symmetric with respect to $\cL_{r} = \cL_{x,r}$, with $\varepsilon \downarrow 0$ as $r \downarrow 0$. Fix some $\tau>0$. By cone-splitting Lemma \ref{l:cone splitting} and the $\varepsilon$-regularity Theorem \ref{t:eps_regularity}, we conclude that if $\varepsilon \leqslant \varepsilon_0(v,\tau)$ is small enough, then there exists some constant $\rho \ll 1$ such that the harmonic radius at $y$ has a lower bound:
\begin{equation}\label{e:harmonic radius lower bound}
    r_h(y) \ge \rho r \text{ for any } y \in B_{r}(x) \setminus B_{\tau r}(\cL_r).
\end{equation}

If we assume that the sequence $M_i$ is Einstein, then by $\varepsilon$-regularity there exists $c(r) \downarrow 0$ as $r\downarrow 0$ such that
\begin{equation}\label{e:curvature scale estimate}
    \sup_{B_{\rho r}(y)} |\Rm| \leqslant c(r) \cdot r^{-2}. 
\end{equation}

Then we can define a \emph{wedge region} $\cW(x)$ around $x$ to be 
\begin{equation}\label{e:wedge region}
    \cW(x) \equiv \cW_{\tau}(x) \equiv  \bigcup_{0<r \leqslant r_x} \{ y \in B_{r}(x) \setminus B_{\tau r}(\cL_r)  \}.
\end{equation}

\begin{proof}[Proof of Corollary \ref{c:local curvature est}]
    Pick a point $y \in \cW(x)$, set $r \equiv 2 d(x,y)$, and use \eqref{e:harmonic radius lower bound} and \eqref{e:curvature scale estimate}.
\end{proof}

\begin{remark}
    We can also take union on $\tau$ to define the rounded wedged region $\hat\cW(x) $ at $x$
    \begin{equation}
    \hat\cW(x) \equiv \bigcup_i B_{r_i}(x) \setminus B_{\tau_i r_i}(\cL_{r_i}).
\end{equation}
Similar estimates hold on the rounded wedge region, replacing $d(x,y)$ by the distance from $y$ to the splitting axis $\cL_r$. 
\end{remark}

We now consider the asymptotic setting. Let $(M,g,p)$ be a pointed 5-manifold with Euclidean volume growth, $\Vol(B_R(p)) \ge v R^5$, and satisfy that $\sup_{B_{2R}(p) \setminus B_R(p)} |\Ric| = o(R^{-2})$. If one tangent cone at infinity splits off a line, then it is unique by Theorem \ref{t:uniqueness at infinity}.

Fix a point $p\in M$. Next we construct the wedge region around infinity $\cW_p(\infty)$: for each $R \ge R_p$ for some large $R_p$, we have $B_R(p)$ is $(1,\varepsilon)$-symmetric with respect to $\cL_R = \cL_{p,R}$ with $\varepsilon\downarrow 0$ as $R \uparrow \infty$. Fix some $\tau>0$. Similarly the \textit{wedge region at infinity} $\cW(\infty) \equiv \cW_{p,\tau}(\infty)$ is defined as
\begin{equation}\label{e:wedge region at inf}
    \cW(\infty) \equiv \bigcup_{R \ge R_p} \{ y \in B_{R}(p) \setminus B_{\tau R}(\cL_R) \}.
\end{equation}

\begin{proof}[Proof of Corollary \ref{c:infinity est}]
    Fix some $p\in M$. For $y \in \cW(\infty)$ constructed as above, we set $R = 2 d(p,y)$. Then the proof is completed by \eqref{e:harmonic radius lower bound} and \eqref{e:curvature scale estimate}.
\end{proof}

Similarly, the rounded wedge region is defined as $\hat\cW(\infty) \equiv \bigcup_{\tau>0} \cW_{p,\tau}(\infty)$, where similar estimates hold. 

Finally we provide a proof of the higher dimensional Theorem \ref{thm:higher dimension unique}, similar to $5$-dimensional case.
\begin{proof}[Proof of Theorem \ref{thm:higher dimension unique}]
    Similar to the proof of \ref{thm:isolation 1sym}, by assumption, every tangent cone at $x$ can be written as $\RR^{n-5} \times C(Y)$ where $Y$ is an Einstein 4-orbifold with isolated singularities. Moreover, let $\RR^{n-5}\times C(Y_0)$ be the tangent cone at $x$ that splits off $\RR^{n-4}$. This implies that $Y_0$ is a sine suspension over $\mathbb{S}^3/\Gamma$ for some $\Gamma \subset O(4)$. Then the proof is completed by Theorem \ref{thm:isolation orbifolds} and the fact that the space of the cross sections of tangent cones at $x$ is connected.
\end{proof}

\subsection{Refined rectifiability}\label{subs:rectifiable}

Let $X^5$ be a noncollapsed Ricci-bounded limit space. Recall that in \cite{cjn,jn}, it is proven that the top stratum $\cS^1 \setminus \cS^0$ is contained in a countable union of bi-Lipschitz images from subsets of $\RR$, modulo some $\cH^1$-measure zero set. In this section, building on the uniqueness Theorem, we will improve the bi-Lipschitz structure by showing that there is no need to discard any measure zero subset, thereby establishing Theorem \ref{thm:lips curve sing}. For related work on uniqueness and the singular set we refer to \cite{cm16,nv17,HJ23,HJ24, HJ25,flb,hkm}.

For convenience, we assume that for some small $\varepsilon_0>0$
\begin{equation}\label{e:Ric small}
    |\Ric| \leqslant \varepsilon_0.
\end{equation}
By rescaling this assumption is always satisfied, without affecting the results in this section. First we note that by uniqueness Theorem \ref{t:local uniqueness}, there is a constant $N(v)>0$ and a natural stratification of the singular set:
\begin{equation}\label{e:refined stratification}
    \cS = \cS^0 \sqcup \bigsqcup_{2 \leqslant |\Gamma| \leqslant N} \cS^1_{\Gamma}, 
\end{equation}
where $\cS_{\Gamma}^1$ is the set of points whose unique tangent cone is $\RR \times \RR^4/\Gamma$, $\Gamma\subset O(4)$.

First we prove a crucial uniform symmetry propagation result, coming from the gap Theorem \ref{thm:isolation orbifolds}. 

\begin{proposition}\label{p:uniform symmetry}
    For any $v>0, \eta>0$ and $\varepsilon \leqslant \varepsilon_0(v,\eta)$, the following holds. If for some $x \in B_1(p) \cap \cS^1_{\Gamma}$ we have $d_{GH}(B_1(x), B_1(0,x')) \leqslant \varepsilon$ with $(0,x')$ the vertex of the cone $\RR \times C(\mathbb{S}^3/\Gamma)$, then $d_{GH}(B_r(x), B_r(0,x')) \leqslant \eta r$ for any $r \leqslant 1$. 
\end{proposition}

\begin{proof}
We argue by contradiction. First note that for each $v>0$, by volume convergence \cite{col}, there are only finitely many possible $\Gamma$ such that $\RR \times C(\mathbb{S}^3/\Gamma)$ could arise as a tangent cone at some point in $B_1(p)$. Hence in the following argument we can fix the group $\Gamma$.

    We suppose that there exist $\nu >0$, $\eta_0>0$, sequences $x_i \in B_1(p_i) \cap \cS^1_{\Gamma}(X_i)$ and $r_i \leqslant 1$ such that
    \begin{equation*}
        d_{GH}(B_1(x_i),B_1(0,x')) \leqslant 1/i \quad \text{ while } \quad d_{GH}(B_{r_i}(x_i), B_{r_i}(0,x')) > \eta_0 r_i.
    \end{equation*}
    
If $\limsup_i r_i \geqslant r_0 > 0$, then it contradicts the fact that $d_{GH}(B_1(x_i),B_1(0,x')) \leqslant 1/i$ when $i$ is large enough. Hence it suffices to consider the case $r_i \to 0$. Since $d_{GH}(B_1(x), B_1(0,x')) < 1/i$, then by volume convergence theorem \cite{col} we have $|\Vol(B_1(x_i)) - \Vol(B_1(0,x'))| \to 0$ as $i \to \infty$. Since $x_i \in \cS^1_{\Gamma}$, we have that the volume density at $x$ is equal to $|\Gamma|^{-1}$. Therefore, the volume ratio satisfies a pinched condition: for any $s\in (0,1]$
\begin{equation}\label{e:Rect;vol pinch}
    |\cV_{x_i}^{-\varepsilon_0}(s) - |\Gamma|^{-1} | \leqslant c(\varepsilon_0) \varepsilon_i, \text{ with } \varepsilon_i \to 0 \text{ as } i \to \infty.  
\end{equation}
Here $\varepsilon_0$ is the constant in \eqref{e:Ric small}.

Let $\delta_0$ be the gap in Theorem \ref{thm:isolation orbifolds}. Since each $x_i \in \cS^1_{\Gamma}$, there exists a sequence $s_i \to 0$ such that 
\begin{equation*}
    d_{GH}(B_{s_i}(x_i),B_{s_i}(0,x')) \leqslant \frac{\delta_0}{10}s_i.
\end{equation*}

Without loss of generality, we may assume $r_i \ge s_i$ for each $i$. Consider $\alpha_{i}(t) := (\Bar{B}_t(x_i), t^{-1}d_i)$ for $t \in [s_i, r_i]$. By Gromov's precompactness theorem, we know that each $\alpha_i([s_i, r_i])$ is contained in the compact space $(\cM,d_{GH})$, where $\cM$ is the space of compact Ricci limit spaces of dimension $5$, diameter upper bound 1, volume lower bound $v$,  and Ricci curvature lower bound $-4$. Hence there exists some uniform constant $N(v) < \infty$ such that there exists a finite set of scales $\{t_{i,j}\}_{1\leqslant j \leqslant N}$ satisfying
    \begin{enumerate}
        \item $t_{i,1} = s_i$ and $t_{i,N} = r_i$ for each $i$.
        \item For each $i$ and each $1\leqslant j \leqslant N-1$ we have $d_{GH}(\alpha_i(t_{i,j}), \alpha_i(t_{i,j+1})) \leqslant \delta_0/10$.
    \end{enumerate}

Then let $i \to \infty$. By passing to subsequences, for each $1\leqslant j \leqslant N$, we have $\alpha_i(t_{i,j})$ converges to a metric space $W_j$. By \eqref{e:Rect;vol pinch} and Lemma \ref{t:almost volume cone}, we conclude that $W_j$ is a metric cone and thus we write $W_j= C(Z_j)$. Since $s_i, r_i \to 0$, we have $|\Ric(\alpha_i(t_{i,j}))| \to 0$ as $i \to \infty$. By the codimension four theorem \ref{t:codim four}, each $Z_j$ is an Einstein orbifold. Hence, by our construction we have
\begin{enumerate}
        \item $d_{GH}(Z_1, \mathbb{S}^4/\Gamma) \leqslant \frac{\delta_0}{10}$;
        \item $d_{GH}(Z_N,\mathbb{S}^4/\Gamma) \ge \eta_0$;
        \item For each $1\leqslant j \leqslant N$ we have $d_{GH}(Z_j, Z_{j+1}) \leqslant \frac{\delta_0}{10}$.
    \end{enumerate}

Since each $Z_j$ is an Einstein 4-orbifold, by Theorem \ref{thm:isolation orbifolds} we have $Z_j \cong \mathbb{S}^4/\Gamma$. This contradicts (2). Hence, the proof is completed.    
\end{proof}

Next we will apply the uniform symmetry result to construct a local bi-Lipschitz map around each point in $\cS^1_{\Gamma}$. Another key ingredient is the neck decomposition developed in \cite{jn,cjn} (see also \cite{nv19}).  

We define the constant $\Bar{V} = 1/|\Gamma|$, the volume density at those points in $\cS^1_{\Gamma}$. Given $x\in X$, $r>0$ and $\xi>0$, we define the volume-pinched set $\cP_{r,\xi}(x)$ as
\begin{equation*}
    \cP_{r,\xi}(x) \equiv \{ y \in B_{4r}(x) : \cV_y(\xi r) \le \Bar{V} + \xi \}.
\end{equation*}

Most constants are chosen compatibly with the neck decomposition in \cite[Section 10]{cjn}, generally satisfying
\begin{equation*}
    0<\xi < \delta < \gamma < \varepsilon < \tau \ll 1. 
\end{equation*}

\begin{proposition}\label{p:local Lipschitz}
    For any $p \in \cS^1_{\Gamma}$, there exists some function $\phi$ defined on $B_{r}(p)$ for some radius $r>0$ such that $\phi$ is bi-Lipschitz on $B_{r}(x) \cap \cS^1_{\Gamma}$: for some constant $C>1$ and any $x,y \in B_{r}(p) \cap \cS^1_{\Gamma}$,
    \begin{equation}\label{e:Recf;biLip}
        C^{-1} \cdot d(x,y) \leqslant |\phi(x) - \phi(y)| \leqslant C \cdot d(x,y).
    \end{equation}
\end{proposition}

\begin{proof}
    Fix $p \in \cS^1_{\Gamma}$. After rescaling, by Proposition \ref{p:uniform symmetry}, there exists $R_0 \gg 1$ and $\eta>0$ such that
    \begin{equation}\label{e:bilip assumption}
        d_{GH}(B_{r}(x),B_r(0,x')) \le \eta r \text{ for any } x\in B_4(p)\cap \cS^1_{\Gamma} \text{ and any } 0<r\le R_0
    \end{equation} 
with $(0,x')$ the vertex of the cone $\RR \times \RR^4/\Gamma$. Moreover, we can assume $\cV_z(\xi^{-1}) \ge \Bar{V} - \xi$ for any $z\in B_4(p)$ if $\eta$ is chosen sufficiently small. 

We claim that if $d_{GH}(B_{r}(x),B_r(0,x')) \le \eta r$ with 
$\eta \le \eta(v,\varepsilon,\gamma,\xi)$, then
\begin{equation}\label{eq: size volume pinched}
\Vol\bigl(B_{\gamma r}(\cP_{r,\xi}(x))\bigr) > \varepsilon \gamma^4 r_d^5.    
\end{equation}
Rescaling by $r_d$ and arguing by contradiction, we obtain a sequence of balls $B_1(x_i)$
converging to $B_1(0,x') \subset \RR \times \RR^4/\Gamma$ for which the above volume bound fails for fixed $v,\varepsilon,\gamma,\xi$. But
$(-4,4)\times\{x'\} \subset \cP_{1,\xi}(0,x')$
which satisfies $\Vol\bigl(B_{\gamma}(\cP_{1,\xi}(0,x'))\bigr) \ge c\,\gamma^4$ for some universal constant $c>0$. By volume convergence, this contradicts the failure of the bound
whenever $\varepsilon < c/10$, proving the claim.
    
    Hence we can apply the decomposition proposition for 1-symmetric balls by \cite[Proposition 10.5]{cjn} (see also \cite{jn}), to obtain a decomposition of $B_1(p)$
    \begin{equation*}
        B_1(p) \subset \Big( \cC_0 \cup \cN \cap B_1(x) \Big) \cup \bigcup_b B_{r_b}(x_b) \cup \bigcup_{d}B_{r_d}(x_d),
    \end{equation*}
where
\begin{enumerate}
    \item $\cN = B_2(p) \setminus \Big( \cC_0 \cup \bigcup_b B_{r_b}(x_b) \cup \bigcup_d B_{r_d}(x_d) \Big)$ is a $(\delta,\tau)$-neck region;
    \item each $B_{2r_b}(x_b)$ satisfies $r_h(x_b)>2r_b$, where $r_b$ is the harmonic radius;
    \item each $B_{2r_d}(x_d)$ satisfies $\Vol(B_{\gamma r_d}\cP_{r_d,\xi}(x_d))< \varepsilon^2 \gamma^4 r_d^5$ or $\cP_{r_d,\xi}(x_d))= \emptyset$.
\end{enumerate}

The center set of the neck region is defined to be $\cC=\cC_+ \cup \cC_0$ where $\cC_+ = \bigcup_b \{x_b\} \cup \bigcup_d\{x_d\}$ with the corresponding radius. And if $x\in \cC_0$, then $r_x = 0$. 
Recall that the associated measure $\mu$ to the neck region $\cN$ is defined by $\mu = \sum_{x\in \cC_+} r_x \delta_x + \cH^1|_{\cC_0}$. By \cite[Theorem 2.9]{cjn}, for any $x \in \cC$ and $r\ge r_x$ with $B_{2r}(x) \subset B_2(p)$, $\mu$ is Ahlfors regular: for some universal constant $A$
\begin{equation}\label{e:ahlfors}
    A^{-1} \cdot r \le \mu(B_r(x)) \le A \cdot r.
\end{equation}

We claim that $\cC_0 \cap B_2(p) = \cS^1_{\Gamma} \cap B_2(p)$. In fact, by cone-splitting \ref{l:cone splitting} and $\epsilon$-regularity \ref{t:eps_regularity}, the region $\cN \cup \bigcup_b B_{r_b}(x_b)$ is regular. It suffices to check that $B_{2r_d}(x_d)\cap \cS^1_{\Gamma}=\emptyset$. Suppose there exists some $x\in B_{2r_d}(x_d)\cap \cS^1_{\Gamma}$. Then by \eqref{e:bilip assumption} and \eqref{eq: size volume pinched}, we have $B_{r_d/2}(x)$ is $(1,\eta)$-symmetric. Therefore $\Vol(B_{\gamma r_d}\cP_{r_d,\xi}(x_d)) > \Vol(B_{\gamma r_d/2}\cP_{r_d/2,\xi}(x))>\varepsilon^2 \gamma^4r_d^5$, contradicting the definition of the $d$-balls.  This implies the claim.

Hence it suffices to find a function that is bi-Lipschitz on $\cC_0$. By \cite[Proposition 9.3]{cjn}, there exists some $\delta'$-splitting function $u$ defined on $B_4(p)$ such that for any $x,y\in \cC$
\begin{equation*}
    \frac{1}{2}\cdot d(x,y)^{2} \le |u(x) - u(y) | \le 2\cdot d(x,y).
\end{equation*}

In particular, $u$ is 1-1 from $\cC$ to some subset in $\RR$. We can now define the function
\begin{equation*}
    \phi(x) \equiv \mu\Big( \{z\in B_2(p) \cap \cC :u(z) \le u(x)\} \Big).
\end{equation*}

We will prove the bi-Lipschitz property for such $\phi$: for some universal constant $C>1$, and any $x,y\in \cC_0$
\begin{equation*}
    C^{-1} \cdot d(x,y) \le |\phi(x) - \phi(y) | \le C \cdot d(x,y).
\end{equation*}

We fix two points $x,y \in \cC_0$ with $d(x,y)=s$. Assume that $u(x)>u(y)$.
Let $\Omega \equiv \{z\in B_2(p) \cap \cC: u(y) < u(z) \le u(x) \}$. Then $|\phi(x) - \phi(y)| = \mu(\Omega)$.

By \cite[Theorem 7.2]{cjn}, there exists some number $T=T_{x,4s}>0$ such that $Tu$ is a $\delta''$-splitting function on $B_{4s}(x)$. By Theorem \ref{t:almost splitting}, $B_{2s}(x)$ is $(1,\delta''')$-symmetric with respect to $\cL_{x,2s}$ and $(Tu,v)$ is a $\delta'''$-GH map for some map $v :B_{2s}(p) \to \RR^4/\Gamma$. Hence it is easy to see that $\Omega \equiv \{z\in B_{4s}(x) \cap \cC : Tu(y) < Tu(z) \le Tu(x)\}$. This proves the upper bound for $\mu(\Omega)$ using \eqref{e:ahlfors} since in particular $\Omega \subset B_{4s}(x)$
\begin{equation}\label{e:bilip upper}
    \mu(\Omega) \le \mu(B_{4s}(x)) \le (4A)\cdot s.
\end{equation}

On the other hand, by the Reifenberg property of $\cC$ (See \cite[Definition 3.1 (n3)]{jn}), there exists some point $z\in B_{2s}(x) \cap \cC$ such that $Tu|_{B_{s/10}(z)} \subset \big(Tu(y), Tu(x) \big) \subset \RR$. Moreover, by the Lipschitz property of the radius function on $\cC$, we have $r_z \le s/100$ since $r_p = r_q =0$. Therefore,  we can apply \eqref{e:ahlfors} to obtain
\begin{equation}\label{e:bilip lower}
    \mu(\Omega) \ge \mu(B_{s/10}(z)) \ge (10A)^{-1} \cdot s.
\end{equation}

Combining \eqref{e:bilip upper} and \eqref{e:bilip lower} we complete the proof of the bi-Lipschitz property of $\phi$.

\end{proof}

\begin{proof}[Proof of Theorem \ref{thm:lips curve sing}]
    We can decompose the singular set as
    \begin{equation*}
        \cS \cap B_1(p) = \bigcup_{i \in I} (\cS^1_{\Gamma_i} \cap B_1(p)) \cup (\cS^0 \cap B_1(p)).
    \end{equation*}

Since we have the uniform lower bound of $\Vol(B_1(x))$ for $x\in B_1(p)$, the index set $I$ is finite. For each $x \in \cS^1_{\Gamma_i} \cap B_1(p)$, by Proposition \ref{p:local Lipschitz}, there exists some radius $r$ such that $B_{r}(x) \cap \cS^1_{\Gamma_i}$ is bi-Lipschitz to a subset of $\RR$. Since $X$ is separable, there exist countably many subsets $Z_j \in \cS^1$ such that each $Z_j$ is bi-Lipschitz to some subset of $\RR$ with $\cS^1 = \bigcup_{j \in J}Z_j$. The proof is completed by noting that $\cS^0$ is also countable.   
\end{proof}

We include another corollary of the uniqueness theorem here: if a curve has constant volume density, then all the tangent cones on the curve are the same. This will be useful to prove the structure theorem for curves of singularities \ref{mthm: orbifold regularity} in the next section. 

\begin{proposition}\label{p:tangent cone is continuous under constant volume ratio}
    Let $\{x_j\}$ be a sequence in $B_1(p)\cap \cS_{\Gamma}^1$ that converges to $y$. Suppose that the volume ratio of the tangent cones at $y$ be the same as the cone $\RR \times \RR^4/\Gamma$. Then $y \in \cS^1_{\Gamma}$, that is, any tangent cone at $y$ is uniquely $\RR \times \RR^4/\Gamma$.
\end{proposition}

\begin{proof}
    It is trivial when $\Gamma = \{id\}$. Hence we assume $\Gamma \neq \{id\}$. Let $C(Y)$ be a tangent cone at $y$ with vertex $y^*$ from a rescaling sequence $r_i \downarrow 0$. By assumption we have $\Vol(B_r(y^*)) = \Vol(B_r(0,x'))$ for any $r>0$ where $(0,x')$ is the vertex of $\RR \times \RR^4/\Gamma$. By picking large $i$ we can assume $|\Vol(B_{r_i}(y)) - \Vol(B_{r_i}(0,x'))| \leqslant \varepsilon_i r_i^5$ where $\varepsilon_i \downarrow 0$ as $i \uparrow \infty$. Since the sequence $x_j$ converges to $y$, then for each $i$ there exists $x_i$ such that $|\Vol(B_{r_i}(x_i)) - \Vol(B_{r_i}(y))| \leqslant \varepsilon_i r_i^5$ and $d_{GH}(B_{r_i}(x_i), B_{r_i}(y)) \leqslant \varepsilon_i r_i$. On the other hand, since each $x_i \in \cS^1_{\Gamma}$, there exists some $s_i < r_i$ such that $d_{GH}(B_{s_i}(x_i), B_{s_i}(0,x')) < \varepsilon_i s_i$. 

    Consider $\alpha_{i}(t) \equiv (\Bar{B}_t(x_i), t^{-1}d)$ for $t \in [s_i, r_i]$. Since $|\Vol(B_{r_i}(x_i)) - \Vol(B_{r_i}(0,x'))| < 2\varepsilon_i  r_i^5$, By Lemma \ref{t:almost volume cone}, we have $\alpha_i(t)$ is $\delta_i$-close to a cone for any $t \in [s_i,r_i]$ with $\delta_i \downarrow 0$ as $i \uparrow \infty$. Following the proof of Proposition \ref{p:uniform symmetry}, by Gromov's precompactness theorem and Theorem \ref{t:codim four}, we have that each $\alpha_i(t)$ converges to the unit ball in $C(Z_t)$ where each $Z_t$ is an Einstein 4-orbifold. Moreover, we have that $\alpha_i(s_i) \to B_1(0,x') \subset \RR \times \RR^4/\Gamma$ and $\alpha_i(r_i) \to B_1(y^*) \subset C(Y)$. By the isolation Theorem \ref{thm:isolation orbifolds}, this implies that $\alpha_i(t) \to B_1(0,x')$ for any $t\in [s_i,r_i]$ and in particular $C(Y) \cong \RR \times \RR^4/\Gamma$. The proof is completed by the uniqueness Theorem \ref{t:local uniqueness}.
\end{proof}

\begin{corollary}\label{c:curve constant volume ratio}
    Let $\gamma$ be a continuous curve in $B_1(p) \subset X$. Assume that the volume density is constant along the curve. Then there exists some $\Gamma \subset 
    O(4)$ such that all of the tangent cones along the curve are isometric to $\RR \times \RR^4/\Gamma$. 
\end{corollary}

\begin{proof}
    Suppose that the curve has the Euclidean volume ratio. Then the whole curve is in the regular set by volume convergence Theorem \cite{col} and all the tangent cones must be Euclidean. Therefore, we can assume $\gamma \subset \cS$. We can decompose 
    \begin{equation*}
        \gamma = \cS^0(\gamma) \sqcup \bigsqcup_{i=1}^N \cS^1_{\Gamma_i}(\gamma).
    \end{equation*}

    By Proposition \ref{p:tangent cone is continuous under constant volume ratio}, each $\cS^1_{\Gamma_i}(\gamma)$ is closed relative to $\gamma$. This implies that $\cS^0(\gamma)$ is open relative to $\gamma$. Since we have the dimension estimate $\dim(\cS^0) \leqslant 0$, it implies that $\cS^0(\gamma) = \emptyset$. Since each $\cS^1_{\Gamma_i}(\gamma)$ is closed relative to $\gamma$, we have $\gamma = \cS^1_{\Gamma_i}$ for some $i$.
\end{proof}

\section{Orbifold structure for curves of singularities}\label{s:orbifold structure}

In this section, we consider a sequence $(M_i,g_i,p_i)_i$ of \emph{Einstein} $5$-metrics with 
$$\Vol(B_1(p_i))> v>0, \qquad \text{ and}\qquad|\Ric(g_i)|\leqslant 4$$
for uniform constants $v>0$. Then, one can take a subsequence converging in the Gromov-Hausdorff sense to a singular space $(X,d,p)$. 

In the previous section we analyzed the regularity of points in $\cS^1 \setminus \cS^0$.  
We now consider the case when the singularities exhaust an entire curve.  
In this setting we show that the curve of singularities in $\cS^1_{\Gamma}$ consists of orbifold singularities lying along geodesics, and that the curvature remains uniformly bounded in a neighborhood of the curve; see Theorem \ref{mthm: orbifold regularity}.

We consider a Lipschitz curve $\gamma: [-2,2] \to \cS^1_{\Gamma}$ with $\gamma(0) = x_0$ and $\Gamma \neq \{id\}$. In the following we let $(0,x')$ be cone tip of $\RR \times \RR^4/\Gamma$. By Proposition \ref{p:local Lipschitz}, we can assume $\gamma$ is $(1+10^{-10})$-bi-Lipschitz by rescaling: 
\begin{equation}\label{e:glue;curve lip}
        \Big| d(\gamma(t),\gamma(s)) - |t-s|\Big| \leqslant 10^{-10} |t-s|.
    \end{equation}

We recall the definition of 5-dimensional orbifolds with curve singularities:
\begin{defn}[$5$-orbifold with curve singularities]\label{orb Ein curve}
    We will say that a metric space $(M_o,g_o)$ is a $5$--dimensional orbifold
    with singularities along curves if there exist $\varepsilon_0>0$ and a countable
    collection of embedded smooth curves $(\gamma_k)_k$ in $M_o$, called
    \emph{singular curves}, such that the following properties hold:
    \begin{enumerate}
        \item the space $\bigl(M_o\setminus\cup_k \gamma_k,\,g_o\bigr)$ is a smooth
        manifold of dimension $5$,
        \item for each singular curve $\gamma_k$ of $M_o$ and every point
        $p\in\gamma_k$, there exist
        \begin{itemize}
            \item an open neighborhood $U_{k,p}\subset M_o$ of $p$,
            \item a finite subgroup $\Gamma_k\subset O(4)$ acting freely on $\mathbb{S}^3$,
            \item a diffeomorphism
            $\Phi_{k,p} \colon B_e(0,\varepsilon_0)\subset
                \mathbb{R}\times\bigl(\mathbb{R}^4/\Gamma_k\bigr) \longrightarrow
                U_{k,p}\subset M_o,$ such that:
        \begin{enumerate}
            \item[(a)] $\Phi_{k,p}\bigl( \mathbb{R}\times\{0\} \cap B_e(0,\varepsilon_0)\bigr)
            = \gamma_k\cap U_{k,p}$,
            \item[(b)] the pull-back metric $\Phi_{k,p}^*g_o$ lifts, on the covering
            $\mathbb{R}\times\mathbb{R}^4$, to a Riemannian metric.
        \end{enumerate}
        \end{itemize}
    \end{enumerate}
    
We say that the orbifold has regularity $C^{k,\alpha}$, (resp. is real-analytic) if all of the metrics $\Phi_{k,p}^*g_o$ are.
\end{defn}

The techniques in this section can be applied to higher dimensional case up to minor modifications, around a codimension four Lipschitz submanifold of singular set where all the tangent cones are $\RR^{n-4} \times \RR^4/\Gamma$ for the same $\Gamma$. For the gluing part \ref{s:C0 chart}, it is more technically complicated in higher dimension so for simplicity we prove it in dimension $n=5$. However, for the regularity upgrading part \ref{s:curvature estimate} we will work in all dimensions to keep track of the dimensional constants from Sobolev inequality.

\subsection{Continuous metric extension}\label{s:C0 chart}

In the first step, we establish the $C^0$ extension of the metric near the singular curve $\gamma$. Our approach follows the strategy in \cite{DS} where the codimension $4$ singularities of $6$-dimensional Kähler-Einstein metrics reduces to the isolated singularity case in dimension four. In the case of isolated singularities, one typically constructs a diffeomorphism between an annulus surrounding the singular point and a corresponding annulus in the tangent cone. These diffeomorphisms are then glued together along the radial direction. The uniqueness of tangent cone then ensures that the error decays as the radial scale tends to zero. 

Several key new difficulties have to be addressed in the case of a higher-dimensional singular set.  At each fixed radial scale, one must first glue together the various local blocks along the directions parallel to the singular set. As the radial scale tends to zero, the number of such blocks increases, making uniform estimates essential to prevent error accumulation. Moreover, the gluing procedure must be carefully implemented to avoid shrinking the annuli in proportion to the number of blocks. We will overcome these difficulties in Lemma \ref{l:uniform symmetry along curves} and Lemma \ref{l:Horizontal Gluing}.

We first state a key uniform symmetry estimate used to control errors in the gluing procedure. The key point is that the symmetry scale $r_0$ is independent of the point $x$ on the curve. 

\begin{lemma}\label{l:uniform symmetry along curves}
    Let $\eta>0$ be given, $x\in \mathcal{S}^1_\Gamma$ and $B_r(0,x')$ metric balls centered at the tip of the cone $\mathbb{R}\times \mathbb{R}^4/\Gamma$. Then 
    there exists some $r_0(\eta)$ such that the following estimate holds \begin{equation}\label{e:curve uniform symmetry}
        d_{GH}(B_r(x), B_r(0,x')) \leqslant \eta r, \text{ for any } x \in \gamma \text{ and any } r \leqslant r_0.
    \end{equation}

Moreover, for each $r_0>0$, there exists $\eta$, with $\eta \to 0$ as $r_0\to 0$, such that \eqref{e:curve uniform symmetry} holds.
\end{lemma}

\begin{proof}
    Let $\eta>0$ be given. Then by Proposition \ref{p:uniform symmetry}, there exists some $\varepsilon_0(\eta)$ such that:
    \begin{equation*}
        \text{whenever } d_{GH}(B_{r_x}(x), B_{r_x}(0,x')) \leqslant \varepsilon_0 r_x \text{ for some } r_x>0, \text{ then } d_{GH}(B_r(x),B_r(0,x')) \leqslant \eta r \text{ for any } r\leqslant r_x. \qquad
    \end{equation*}
    Therefore we can define 
    \begin{equation*}
        r_x \equiv \sup \big\{ r: d_{GH}(B_s(x), B_s(0,x')) < \tfrac{\varepsilon_0 s}{2} \text{ for any } s \leqslant r \big\}>0.  
    \end{equation*}

    Then for any $y \in B_{\varepsilon_0 r_x/2}(x) \cap \gamma$, we have
    \begin{equation*}
        d_{GH}(B_{r_x}(y), B_{r_x}(0,x')) \leqslant d_{GH}(B_{r_x}(y), B_{r_x}(x)) + d_{GH}(B_{r_x}(x), B_{r_x}(0,x')) \leqslant \frac{\varepsilon_0 r_x}{2} + \frac{\varepsilon_0 r_x}{2} = \varepsilon_0 r_x.
    \end{equation*}
    Hence we have $d_{GH}(B_r(y),B_r(0,x')) \leqslant \eta r$ for any $r\leqslant r_x$. 

    Since $\gamma$ is a closed and thus compact subset, there exist finitely many $i$ such that 
    \begin{equation*}
        \gamma = \bigcup_{x \in \gamma} B_{\varepsilon_0 r_x /2}(x) \cap \gamma = \bigcup_{i=1}^M B_{\varepsilon_0 r_{x_i} /2}(x_i) \cap \gamma.
    \end{equation*}

    The proof is finished if we choose $r_0 = \min\{r_{x_1}, ..., r_{x_M}\}$. Moreover, note that if $\eta \to 0$, then $\varepsilon_0(\eta) \to 0$ and thus $r_x \to 0$. This proves the second claim.
    \end{proof}

Let $g_0$ denote the standard flat metric in $\RR \times \RR^4/\Gamma$ with vertex $(0,x')$. 
We define the annuli as
\begin{equation*}
    (b_1, b_2) \times  A_{a_1, a_2} \equiv \{(t, z) \in \RR \times \RR^4/\Gamma: b_1 < t < b_2,  a_1 < |z| < a_2 \}.
\end{equation*}
For convenience, we define $|x|_{\Gamma} \equiv \|z \|$ for any $x = (t,z)$.

Suppose $B_{200r}(x)$ be $(1,\eta)$-symmetric with respect to $\cL_{x,200r}$ for some $\eta < 1/10$, i.e. $$d_{GH}(B_{200r}(x), B_{200r}(0,x')) \leqslant 200\eta r$$ and $\cL_{x,200r}$ is the Gromov-Hausdorff correspondence of $(\RR \times \{0\}) \cap B_{200r}(0,x') \subset \RR \times \RR^4/\Gamma$. Then by $\varepsilon$-regularity due to Anderson \cite{and}, Colding \cite{col}, Cheeger-Naber \cite{cn}, and the uniqueness Theorem \ref{t:local uniqueness}, for any $\delta>0$, if $\eta\leqslant \eta_0(\delta)$, there exists some smooth embedding $\phi_{\eta} : [-20r, 20r] \times  {\Bar{A}}_{r, 100r} \to B_{200r}(x) \setminus B_{2\eta}(\cL_{x,200r}) 
$ satisfying
\begin{enumerate}
    \item $|r^{-2}\phi_{\eta}^* g - g_0|_{C^4} < \delta$,
    \item For any $y \in (-20r, 20r) \times {\Bar{A}}_{r, 100r}$ we have
    \begin{equation}\label{e:glue;gamma distance}
        (1-\delta) r |y|_{\Gamma} \leqslant  d( \phi_{\eta}(y), \cL_{x,200r}) \leqslant (1+\delta) r |y|_{\Gamma}.
    \end{equation}
\end{enumerate}

The following approximation by isometry lemma is standard and proven like \cite[Lemma 5.8]{DS}.

\begin{lemma}\label{l:isom appro}
    Let $\delta>0$ be given small. Then there exists some $\delta'$, with $\delta' \to 0 $ as $\delta \to 0$,  such that for any smooth map $\phi: [11,19] \times \Bar{A}_{1,100} \to \RR \times \RR^4/\Gamma$ with $|\phi^* g_0 - g_0|_{C^4} \leqslant \delta$, then there exists some isometry $P_i \in \operatorname{Isom}(\RR) \times \operatorname{Isom}(\RR^4/\Gamma)$ such that $|P_i \circ \phi -  \mathrm{id}|_{C^3} < \delta'$ in a smaller region. 
\end{lemma}

\begin{proof}
    The proof follows from contradiction arguments.
\end{proof}

Now we glue the local diffeomorphisms together along the direction parallel to the singular curve $\gamma$.  
The subtle point is that the domains $(-1,1) \times \bar{A}_{10r,90r}$ are not scale invariant; however, all estimates remain uniform in $r$, independent of the parameter $t$. Crucially, $\operatorname{Isom}(\RR) \times \operatorname{Isom}(\RR^4/\Gamma)$ is much smaller than $\operatorname{Isom}(\RR^5)$. The isometry $P$ from Lemma \ref{l:isom appro} can only be a translation or reflection in $\RR$, combined with a rotation respecting $\Gamma$-action on $\RR^4$.

\begin{lemma}\label{l:Horizontal Gluing}
For any $r$ small, there exists a smooth embedding $H_r:(-1,1) \times \bar{A}_{10r,90r} \to B_2(x_0)$ such that
\begin{align}\label{e:Horizontal gluing}
    (1- \delta) |y|_{\Gamma} \leqslant d(H_r(y), \gamma) \leqslant (1+\delta) |y|_{\Gamma}, \text{ and } |r^{-2} H_r^{*} g  - g_0|_{C^2} \leqslant \delta(r).
\end{align}
And $\delta(r) \to 0$ as $r \to 0$. 
\end{lemma}

\begin{proof}
    Let $r \in (0, 10^{-2})$ be given small. Let $\eta=\eta(r)>0$ be as in Lemma \ref{l:uniform symmetry along curves} and $\phi_{\eta,x}$ be the embedding defined in previous discussion for each $x \in \gamma$.

    Consider the point set $\{ x_i\}$ where $ x_i = \gamma(30 i r )$ for integers $i \in [-N(r), N(r)]$ with $N(r) \leqslant C/r$ and their associated maps $\phi_i \equiv \phi_{\eta, x_i}$. Define
    $A_i \equiv \phi_i( (-20r, 20r) \times {\Bar{A}}_{r, 100r}) \subset B_2(x_0).$ 

    If we choose $r$ small and thus $\eta$ small enough, then by \eqref{e:glue;curve lip} and $\varepsilon$-regularity Theorem \ref{t:eps_regularity}, for $x_1 = \gamma(30r)\in B_{100r}(x_0)$ we have
    \begin{align*}
        d_{H}(\cL_{x_0,200r} \cap B_{100r}(x_0), \cL_{x_1, 200 r}\cap B_{100 r}(x_0)) \leqslant  10^{-10} r.
    \end{align*}

    Similar inequality holds for general $x_i$ and $x_{i+1}$. Hence by \eqref{e:glue;gamma distance} we can conclude that
    \begin{enumerate}
        \item $\bigcup_i A_i \subset \big\{ x \in X: 0.9r< d(x, \gamma) < 101r\big\}.$ 
        \item For any $|i - j| > 1$, we have $A_i \cap A_j = \emptyset$.
        \item Let $D_i \equiv A_i \cap A_{i+1}$. Then we have $(11 r,19r) \times \bar{A}_{2r,99r} \subset \phi_i^{-1}(D_i)$ and $(-19r, -11r) \times \Bar{A}_{2r,99r} \subset \phi^{-1}_{i+1}(D_i)$ for any such $i$. 
    \end{enumerate}

In the following we assume $r=1$ for simplicity. First we glue $A_i$ and $A_{i+1}$ together.

 Consider the map $\Psi_i\equiv \phi_{i+1}^{-1} \circ \phi_i : [11,19] \times \Bar{A}_{2,99} \to [-20,20] \times \Bar{A}_{1,100}$. Then $\Psi_i$ satisfies 
 \begin{align*}
     (1-3\delta)|x|_{\Gamma} \leqslant |\Psi_i(x)|_{\Gamma}  \leqslant (1+3\delta) |x|_{\Gamma} \quad \text{ and } \quad |\Psi_i^* g_0 - g_0|_{C^4} \leqslant C_1 \delta,
 \end{align*}
 with constant $\delta$ as in \eqref{e:glue;gamma distance}.
By Lemma \ref{l:isom appro}, there exists some isometry $P\in  \operatorname{Isom}(\RR) \times \operatorname{Isom}(\RR^4/\Gamma)$ such that $|P \circ \Psi_i - id|_{C^3} \leqslant C_2(\delta)$ on $(12, 18) \times A_{3,98}$. Then up to some reflection in $t$-direction, we can pick $P_i = (T_{+30}, R_i)$, where $T_{+30}(t,z) = (t+30, z)$ is a translation along $t$-direction and that $R_i \in \operatorname{Isom}(\RR^4/\Gamma)$ is a rotation,  such that $|P_i \circ \Psi_i -id|_{C^3} \leqslant C_2(\delta)$.

Now we choose a smooth cut-off function $\chi \in C^{\infty}(im(\phi_i))$ using the map $\phi_i$ satisfying 
\begin{enumerate}
    \item $0\leqslant \chi \leqslant 1$ in $im(\phi_i)= A_i$.
    \item $\chi(\phi_i(t,z)) = 1$ for $z \in \bar{A}_{1+\theta(t), 100 - \theta(t)}$ and $t \in [-20, 14]$.
    \item $\chi(\phi_i(t,z)) = 0$ for $t \geqslant 16$.
    \item $\chi(\phi_i(t,z)) = 0$ for $|z| \notin \bar{A}_{2, 99}$ for $t \geqslant 9$.
\end{enumerate}
Here $\theta(t)$ is a non-decreasing function with $\theta(t) \equiv 0$ for $t \leqslant 5$ and $\theta(t) \equiv 2$ for $t \geqslant 9$. We can trivially extend $\chi$ by assuming $\chi \equiv 0$ so that $\chi$ is defined smoothly in $A_i \cup A_{i+1}$ satisfying $|\chi|_{C_g^4} \leqslant C_3$.

Consider the pinched annulus $\Omega_i$ defined as a smooth connected region $\Omega_i \equiv \Omega_{i,1} \cup \Omega_{i,2} \subset A_i \cup A_{i+1}$ where 
\begin{align*}
    \Omega_{i,1} &\equiv \{ \phi_i(t,z) : -20 \leqslant t \leqslant 20, \text{ and }  z \in \bar{A}_{1+2\theta(t), 100 -2\theta(t)} \};\\
    \Omega_{i,2} &\equiv \{ \phi_{i+1}(t,z) : -20 \leqslant t \leqslant 20, \text{ and }  z \in \bar{A}_{1+\theta(-t), 100 -\theta(-t)} \}.
\end{align*}

We construct a map defined on $\Omega_i \subset A_i \cup A_{i+1}$ by
\begin{equation}
    h_i \equiv (1 - \chi) \cdot P_i \circ \phi^{-1}_{i+1}  + \chi \cdot \phi_{i}^{-1}.
\end{equation}

Note that $h_i$ satisfies that
\begin{enumerate}
    \item $h_i(x) = P_i \circ \phi_{i+1}^{-1}(x)$ if $x \in  \Omega_i \setminus A_i$,
    \item $h_i(x) =  \phi_{i}^{-1}(x)$ if $x \in \Omega_i \setminus A_{i+1}$,
    \item $(1-C_4(\delta))|h_i(x)|_{\Gamma} \leqslant d(x, \gamma)  \leqslant (1+C_4(\delta))|h_i(x)|_{\Gamma}$,
    \item $|h_i^* g_0 - g|_{C^2_g} \leqslant C_5(\delta)$.
\end{enumerate}

Recall that $P_i = (T_{+30},R_i)$ with $R_i$ a rotation of $\RR^4/\Gamma$. Hence $h_i$ is a diffeomorphism onto $h_i(\Omega_i)$ where
\begin{enumerate}
    \item $\Big( (-20,5) \cup (25,50) \Big) \times \bar{A}_{1,100} \subset h_i(\Omega_i)$.
    \item $(-20, 50) \times \bar{A}_{10,90} \subset h_i(\Omega_i)$.
\end{enumerate}

We can then glue $\Omega_i$ and $A_{i-1}$. Since the left end of $\Omega_i$ is the same as $A_i$, and the pinched region $D_i$ is disjoint with $A_{i-1}$, we can glue these two using the same arguments with the uniform estimate: there exists a smooth embedding $H_{i-1}: \Tilde{\Omega}_{i-1} \to \RR \times \RR^4/\Gamma$ such that
\begin{enumerate}
    \item $\Big((-50, -25) \cup (25, 50) \Big)\times \bar{A}_{1,100} \subset H_{i-1}(\Tilde{\Omega}_i)$,
    \item $(-50, 50) \times \bar{A}_{10,90} \subset H_{i-1}(\Tilde{\Omega}_i)$,
    \item $(1-C_4(\delta))|H_{i-1}(x)|_{\Gamma} \leqslant d(x, \gamma)  \leqslant (1+C_4(\delta))|H_{i-1}(x)|_{\Gamma}$,
    \item $|H_{i-1}^* g_0 - g|_{C^2_g} \leqslant C_5(\delta)$.
\end{enumerate}

The same holds for the extension to the right end. Note that at each gluing, we only modify the overlapping region but keep the non-overlapping ends unchanged. Therefore, after at most finitely many times gluing, the estimates hold uniformly and we obtain a diffeomorphism satisfying the required properties.  
\end{proof}

Hence we can glue all $t$-directional blocks together with uniform estimates. Then we can glue these annuli along the $\Gamma$-direction.

\begin{prop}\label{prop:C^0 metric}
    For any small $r$, there exists a diffeomorphism $F: B^*_r(0) \subset \big(\RR \times \RR^4/\Gamma \setminus \RR \times \{0\}\big) \to B_2(x_0) \setminus \gamma$ such that $F^*g$ extends continuously to a metric tensor over the universal cover $\Hat{B}_r(0) \subset \RR \times \RR^4$. 
\end{prop}

\begin{proof}
    For any small $r$, by Lemma \ref{l:Horizontal Gluing}, there exists a smooth diffeomorphism $H_r$ satisfying properties \eqref{e:Horizontal gluing} with error $\delta(r) \downarrow 0$ as $r\downarrow 0$. Then we apply similar gluing arguments as in Lemma \ref{l:Horizontal Gluing} to the map $V_0\equiv H_{r_0}$ with some fixed small $r_0$, and the map $H_{2^{-1}r_0}$, to obtain a new diffeomorphism $V_1: (-1,1) \times \bar{A}_{5r_0, 90r_0} \to B_2(x_0) \setminus \gamma$ with estimates
    \begin{align*}
        (1- \delta') |y|_{\Gamma} \leqslant d(V_1(y), \gamma) \leqslant (1+\delta') |y|_{\Gamma}, \text{ and } |V_1^{*} g  - g_0|_{C^2} \leqslant \delta'.
    \end{align*}
    Then by induction we can glue the maps $V_i$ and $H_{2^{-{i}}r_0}$ together to obtain $V_{i+1}$ satisfying the same estimates above with $\delta'(i) \to 0$ as $i \to \infty$. Therefore we get a map $F: B^*_{r_0}(0) \to B_2(x_0)$, which satisfies 
    \begin{align*}
        |F^*g - g_0|_{C^0(-1,1) \times \bar{A}_{2^{-i}r_0,90r_0}} \leqslant \delta'(i).
    \end{align*}
    Since $\delta'(i) \to 0$ as $i \to \infty$, $F^*g$ extends continuously in $\Hat{B}_{r_0}(0)$. 
\end{proof}

\subsection{$L^{\infty}$ curvature bound}\label{s:curvature estimate}

In this section we prove that the curvature is bounded around the singular curve $\gamma:[-2,2] \to \cS^1_{\Gamma}$ with $\gamma(0)=x_0$. By Lemma \ref{l:uniform symmetry along curves} and rescaling, we can simply assume that there exists some small $\eta>0$ such that
\begin{equation}
        d_{GH}(B_r(x), B_r(0,x')) \leqslant \eta r, \text{ for any } x \in \gamma([-1,1]) \text{ and any } r \leqslant 1.
    \end{equation}

The main theorem of this section is the following uniform $L^{\infty}$ curvature bound:  
\begin{theorem}\label{t:bound on Rm}
If $\eta$ is small enough, then for any integer $k\ge0$, there exists some constant $C_k(\eta)$ such that
    \begin{equation}
        \sup_{B_1(x_0)}|\nabla^k\Rm| \leqslant C_k(\eta).
    \end{equation}
\end{theorem}

By Proposition \ref{prop:C^0 metric}, we have a $C^0$ coordinate chart for the ball $B_1(x_0)$ and the metric is $C^0$-close to the standard metric in $\RR \times \RR^4/\Gamma$. Since we will focus on the curvature estimate in this section, without loss of generality we study the universal cover of the ball $B_1(x_0)$. The sobolev inequalities hold in $B_1(x_0)$ thanks to \cite{cro}. To keep track of the dimensional constants, we will work in any dimension $n\ge 5$ in this section. 

We define $B \equiv B_1^x(0^{n-4}) \times B_1^y(0^4)$ and the ball minus a segment $B^* \equiv B \setminus (B_1^x(0^{n-4}) \times \{0\})$. Also, we let $\nabla_x, \nabla_y$ denote the gradient along $x, y$ direction respectively. 

In the following we consider the equation which holds in $B^*$: 
\begin{equation}\label{e: sub-equ of u}
    \Delta u \geqslant - f u. 
\end{equation}
It will be applied to $f= |\Rm|$ and both $u =|\nabla^k\Rm|$ or $u =|\Rm|^{1-\alpha}$ for $k\in\mathbb{N}$ and some dimensional $\alpha$. 

We will adapt Moser's iteration argument to prove that $|\Rm|$ is bounded. To achieve this, we first prove the following Caccioppoli type inequality, motivated by \cite{sib85}.

\begin{lemma}\label{l:Caccioppoli}
    Let $u \in C^1(B^*)$ and $u \geqslant 0$ satisfying the equation \eqref{e: sub-equ of u} in the sense of distributions. Suppose $u \in L^{2q}$ for some $q \geqslant 1 + \varepsilon >1$ and $|f(x,y) | \leqslant \delta |y|^{-2}$ for some $\delta \leqslant \delta_0(\varepsilon,q)$. Then for any $\eta \in C^{\infty}_c(B)$ we have
    \begin{align}\label{e:Cacchio-1}
        \int_B \eta^2 |\nabla u^q|^2 \leqslant C(\varepsilon) q \int_B |\nabla \eta|^2 u^{2q}.
    \end{align}
\end{lemma}

\begin{proof}
    By integration by parts, the following holds for any function $\chi \in C_c^{\infty}(B^*)$ with compact support in $B^*$
    \begin{equation}\label{e:[Curv. Est] Int by part}
        \int_B \nabla \chi \cdot \nabla u \leqslant \int_B \chi f u.
    \end{equation}

We define the following test function for some large constant $K>0$
\begin{equation}
    F(t) = \frac{t^{q}}{(1 + \frac{t}{K})^{q/2}}
\end{equation}
whose motivation is that if $u\in L^{2q}$, then $F(u)\in L^4$ since we temper the large values of the function.

Note that the function $F$ is $C^{\infty}$ for $t>0$. We define $G(t) \equiv F(t) F'(t)$. Then the following properties hold by direct computation:
\begin{align}
    F(t) &\leqslant  K^{q/2} t^{q/2};\\
    tG(t) &\leqslant qF(t)^2;\label{e:[Curv. Est] tG < F}\\
    G'(t) &\geqslant c(\varepsilon) F'(t)^2. \label{e:[Curv. Est] G' > c F'}
\end{align}

Now we pick $\eta, \bar{\eta}_{\varepsilon} \in C_c^{\infty}(B)$ and further let $\bar{\eta}_{\varepsilon}$ vanish in $B_1^x \times B_{\varepsilon}^y$. Choose the cutoff function $\chi \equiv (\eta \bar{\eta}_{\varepsilon})^2 G(u)$ and by \eqref{e:[Curv. Est] Int by part} and \eqref{e:[Curv. Est] tG < F} we have
\begin{align}\label{eq:control nablaG nablau}
    \int_{B} \nabla \Big( (\eta \bar{\eta}_{\varepsilon})^2 G(u)\Big) \cdot \nabla u \leqslant \int_{B} (\eta \bar{\eta}_{\varepsilon})^2 G(u) f u \leqslant q \int_{B} (\eta \bar{\eta}_{\varepsilon})^2  |f| |F(u)|^2.
\end{align}

Note that $\nabla F(u) = F'(u) \nabla u$ and $\nabla G(u)= G'(u) \nabla u$. Hence by \eqref{e:[Curv. Est] G' > c F'} we have $\nabla G(u) \cdot \nabla u = |\nabla u|^2 G'(u) \geqslant c(\varepsilon) F'(u)^2 |\nabla u|^2 = c(\varepsilon)|\nabla F(u)|^2$. Hence
\begin{align}
    \int_{B} \nabla \Big( (\eta \bar{\eta}_{\varepsilon})^2 G(u)\Big) \cdot \nabla u &= \int_B (\eta \bar{\eta}_{\varepsilon})^2 \nabla G(u) \cdot \nabla u + \int_B G(u) \nabla (\eta \bar{\eta}_{\varepsilon})^2 \cdot \nabla u \\
    &\geqslant c(\varepsilon) \int_B (\eta \bar{\eta}_{\varepsilon})^2 |\nabla F(u)|^2- \int_B    |2 \eta \bar{\eta}_{\varepsilon} F'(u) \nabla u| \cdot  |  F(u) \nabla(\eta \bar{\eta}_{\varepsilon}) |   \\
    &= c(\varepsilon) \int_B (\eta \bar{\eta}_{\varepsilon})^2 |\nabla F(u)|^2- \int_B    |2 \eta \bar{\eta}_{\varepsilon} \nabla F(u)| \cdot  |  F(u) \nabla(\eta \bar{\eta}_{\varepsilon}) |
\end{align}

Combining these and \eqref{eq:control nablaG nablau} we obtain
\begin{align}\label{e:[Curv. Est] I_1 + I_2 inequality}
    c(\varepsilon) \int_B (\eta \bar{\eta}_{\varepsilon})^2 |\nabla F(u)|^2 \leqslant \int_{B} |2 \eta \bar{\eta}_{\varepsilon} \nabla F(u)| \cdot |\nabla(\eta \bar{\eta}_{\varepsilon}) F(u)|  + q \int_B (\eta \bar{\eta}_{\varepsilon})^2 |f| F(u)^2 \equiv I_1 + I_2.
\end{align}

By Cauchy-Schwarz inequality, for some $c'(\varepsilon)$ we have
\begin{equation}
    I_1 \leqslant \frac{c(\varepsilon)}{2} \int_B (\eta \bar{\eta}_{\varepsilon})^2 |\nabla F(u)|^2 + c'(\varepsilon) \int_B |\nabla (\eta \bar{\eta}_{\varepsilon})|^2 F(u)^2.
\end{equation}

To estimate $I_2$, we apply Hardy inequality to get
\begin{equation}\label{e:[curvatuer bound] I_2}
    I_2 \leqslant  \delta q \int_{B_1^x} \int_{B_1^y} 
    (\eta \bar{\eta}_{\varepsilon})^2 F(u)^2 |y|^{-2} dy dx \leqslant  4\delta q \int_{B_1^x} \int_{B_1^y} |\nabla_y   (\eta \bar{\eta}_{\varepsilon} \cdot F(u))|^2(x,y) dy dx. 
\end{equation}

Note that $|\nabla_y F(u)|^2 \leqslant |\nabla F(u)|^2$. Hence by choosing $\delta \leqslant \delta_0(\varepsilon,q)$ small, we can absorb all terms involving $|\nabla F|$ into the left hand side of \eqref{e:[Curv. Est] I_1 + I_2 inequality} and thus obtain
\begin{equation}
    \int_{B} (\eta \bar{\eta}_{\varepsilon})^2 |\nabla F(u)|^2 \leqslant C(\varepsilon)q \int_B |\nabla (\eta \bar{\eta}_{\varepsilon})|^2 F^2(u).
\end{equation}
Here is the only place where we require $\delta$ to be small depending on $q$. 

Note that we can choose $\bar{\eta}_{\varepsilon}$ such that $\eta \cdot \bar{\eta}_{\varepsilon} \to \eta$ and $\int_{B \cap \supp(\eta)} | \nabla \bar{\eta}_{\varepsilon}|^4 \to 0$ as $\varepsilon \to 0$ since $B\setminus \supp(\Bar{\eta}_{\varepsilon}) \subset B_1^x \times B_{2\varepsilon}^y \to B_1^x \times \{0\} \subset \RR^{n-4} \times \{0\}$. Combining $F(u) \leqslant K^{q/2} u^{q/2}$ and $u \in L^{2q}$, we have $F(u) \in L^4$ and thus
\begin{equation}
    \int_{B} |\nabla \bar{\eta}_{\varepsilon}|^2 \eta^2 F^2(u) \leqslant C(K) \Bigg(\int_{B} |\nabla \bar{\eta}_{\varepsilon}|^4 \Bigg)^{1/2} \cdot \Bigg( \int_B F(u)^4 \Bigg)^{1/2} \to 0, \text{ as } \varepsilon \to 0.
\end{equation}

Combining all, we can now conclude that
\begin{equation}
    \int_{B} \eta^2 |\nabla F(u)|^2 \leqslant C(\varepsilon)q \int_B |\nabla \eta|^2 F^2(u).
\end{equation}
Finally we let $K \to \infty$. Note that $F(u) \to u^{q}$. The proof is finished by Fatou's Lemma.    \end{proof}

As a consequence, we get higher integrability of $u$ whenever $\delta$ is small enough.
\begin{lemma}\label{l:Higher integrability for u}
    Let $u \in C^1(B^*)$ and $u \geqslant 0$ satisfying the equation \eqref{e: sub-equ of u} in the sense of distributions. Suppose $u \in L^{2q}$ for some $q \geqslant 1 + \varepsilon >1$. For any $N>1$, if $|f(x,y) | \leqslant \delta |y|^{-2}$ for some $\delta \leqslant \delta_0(n,\varepsilon,N)$, we have 
    \begin{align}
        ||u^q||_{L^{2N}(B_{1/2})} \leqslant C(n,\varepsilon,N) ||u^q||_{L^2(B_1)}
    \end{align}
\end{lemma}

\begin{proof}
    We will prove it using Sobolev inequality and iteration argument. We assume $q<N$. Otherwise, there is nothing to prove.

    Let $v_0 = u^q$. Define $v_k = v_{k-1}^{\frac{n}{n-2}} = u^{q \cdot (\frac{n}{n-2})^k}$. We further define the radii $r_k = \frac1{2}  +\frac1{2^{k+1}}$. Then we can define the cutoff functions $\eta_k \equiv 1$ in $B_{r_{k+1}}$ and vanishing outside $B_{r_{k}}$.

    By Sobolev inequality in $B_1 \subset \RR^n$ and \eqref{e:Cacchio-1}, with $2^* \equiv \frac{2n}{n-2}$ we have
    \begin{align*}
        ||v_0||_{L^{2^*}(B_{r_1})} \leqslant ||\eta_0 v_0||_{L^{2^*}(B_1)} \leqslant C ||\nabla (\eta_0 v_0)||_{L^2(B_1)} \leqslant C( ||\eta_0 \nabla v_0||_{L^2(B_1)} + ||v_0 \nabla \eta_0||_{L^2(B_1)})      \leqslant C_0(n,\varepsilon,N)||v_0 \nabla \eta_0||_{L^2(B_1)}.
    \end{align*}

 This implies that $v_0\in L^{2^*}(B_{r_1})$, equivalently $u \in L^{\frac{2qn}{n-2}}(B_{r_1})$. If $\frac{qn}{n-2} \geqslant N$, we finish the proof. Otherwise, we can apply \eqref{e:Cacchio-1} and Sobolev inequality again to $\eta_1 v_1$ to obtain
 \begin{align}
     ||\eta_1 v_1||_{L^{2^*}(B_{r_2})} \leqslant C(n,\varepsilon,N) ||v_1 \nabla \eta_1||_{L^2(B_{r_2})} \leqslant C_1(n,\varepsilon,N) ||v_0||_{L^{2^*}(B_{r_2})}.
 \end{align}

 Next we iterate this step for finitely many times until $N \equiv q\cdot(\frac{n}{n-2})^M > \frac{n}{2}$. Then by induction
\begin{align}
    ||\eta_M v_M||_{L^{2^*}(B_{r_{M+1}})} \leqslant C_M(n,\varepsilon,N) ||v_0||_{L^2(B_{r_{M+1}})}.
\end{align}

Since $B_{1/2} \subset B_{r_{M+1}} \subset B_1$ and $v_M = u^{q\cdot (\frac{n}{n-2})^M } = u^N$, it implies that  $u^q \in L^{2N}(B_{1/2})$ with $ ||u^q||_{L^{2N}(B_{1/2})} \leqslant C_M(n,\varepsilon,N) ||u^q||_{L^2(B_1)}$. This finishes the proof.
\end{proof}

Next we will use Moser iteration argument to obtain $L^{\infty}$-estimate for the following equation for some $q_0=q_0(n)>1$ in $B_1$
\begin{equation}\label{e:[Moser] u^{1+q}}
    \Delta u \geqslant - C u^{1+q_0}.
\end{equation}

Note that in Lemma \ref{l:Caccioppoli}, we obtain a Caccioppoli type inequality when $\delta$ is small but depends on $q$.  Here we can revise the
proof to obtain improved estimates with small $\delta$ independent of $q$, which is crucial to run the iteration argument towards $q=\infty$. The idea is as follows. By Lemma \ref{l:Higher integrability for u}, whenever $\delta(n,\varepsilon)$ is chosen small, we can obtain $L^N$-norm for $u^{q}$ with $N>n/2$, and thus for $f=u^{q_0}$ in the case of \eqref{e:[Moser] u^{1+q}}. Then in the proof of Lemma \ref{l:Caccioppoli} we replace the estimate $|f| \leqslant \delta|y|^{-2}$ by $L^N$-norm of $f=u^{q_0}$. Note that the $L^N$-estimate of $f$ does not depend on $\delta$. Hence we can get similar Caccioppoli type inequality for all $q$. Also, since we assume the bound $|f(x,y)| = |u^{q_0}(x,y)| \leqslant \delta |y|^{-2}$, we have $u \in L^{2q}$ for some $q(q_0)>1$ particularly.

\begin{lemma}\label{l:Cacchio-2}
    Let $u \in C^1(B_1^*)$ and $u \geqslant 0$ satisfying the equation \eqref{e:[Moser] u^{1+q}} in the sense of distributions. Suppose $u \in L^{2+2\varepsilon}$ for some $\varepsilon>0$, and $|u^{q_0}(x,y) | \leqslant \delta |y|^{-2}$ with $q_0$ in \eqref{e:[Moser] u^{1+q}} and some $\delta \leqslant \mathbf{\delta_0(n,q_0)}$. Then for any $q>1$, we have $u \in L^{2q}$. Moreover, there exists some $\alpha(n)>0$ such that for any $\eta \in C^{\infty}_c(B_{1/2})$,     \begin{align}\label{e:Caccip lemma 2}
    \int_{B_{1/2}} \eta^2 |\nabla u^q|^2 \leqslant C(n,\varepsilon,||u^{1+\varepsilon}||_{L^{2}(B_1)}) q^{\alpha(n)} \int_{B_{1/2}} (|\nabla \eta|^2 + \eta^2) u^{2q}.
    \end{align}
\end{lemma}

\begin{proof}
We can let $f=Cu^{q_0}$ in this case. Then by Lemma \ref{l:Higher integrability for u}, we have $u \in L^{2q}$ for any $q>1$. Moreover, if we choose $\delta\leqslant \delta_0(n,\varepsilon)$, we know $f=Cu^{q_0} \in L^{2n}(B_{1/2})$ and in fact $\|f\|_{L^{2n}(B_{1/2})} \leqslant C(n,\varepsilon)\|u^{1+\varepsilon}\|_{L^2(B_1)}$.

Then we can go back to the proof of Lemma \ref{l:Caccioppoli}. Now we know $u \in L^{2q}$ and we want to prove \eqref{e:Caccip lemma 2}. Following the same steps as in the proof of Lemma \ref{l:Caccioppoli}, note that the only place where we use the estimate for $f$ is the estimate of $I_2$, see \eqref{e:[curvatuer bound] I_2}. Now using the $L^{2n}$-norm of $f$, we can compute $I_2$ over $B_{1/2}$ in \eqref{e:[Curv. Est] I_1 + I_2 inequality} as follows:
\begin{align}
    I_2 = q \int_{B_{1/2}} (\eta \bar{\eta}_{\varepsilon})^2 |f| F(u)^2 &\leqslant q \Big(\int_{B_{1/2}} |f |^{2n} \Big)^{\frac1{2n}} \cdot \Big( \int_{B_{1/2}} |(\eta \bar{\eta}_{\varepsilon})F(u)|^{\frac{4n}{2n-1}}\Big)^{1- \frac1{2n}}.
\end{align}

Since $2^* \equiv \frac{2n}{n-2} > \frac{4n}{2n-1} > 2$, then by interpolation inequality and Sobolev inequality we have
\begin{align}
    \Big( \int_{B_1} |(\eta \bar{\eta}_{\varepsilon})F(u)|^{\frac{4n}{2n-1}}\Big)^{1- \frac1{2n}} &\leqslant \theta ||(\eta \bar{\eta}_{\varepsilon})F(u)||_{L^{2^*}(B_1)}^2 + C(n) \theta^{-\alpha} ||(\eta \bar{\eta}_{\varepsilon})F(u)||_{L^2(B_1)}^2 \\
    &\leqslant C_0 \theta ||\nabla (\eta \bar{\eta}_{\varepsilon} F(u))||_{L^{2}(B_1)}^2 + C(n) \theta^{-\alpha} ||(\eta \bar{\eta}_{\varepsilon})F(u)||_{L^2(B_1)}^2,
\end{align}
where $\alpha$ is a positive constant depending on $n$ and $\theta>0$ to be chosen. This implies
\begin{align}
    I_2 \leqslant q \cdot ||f||_{L^{2n}(B_{1/2})} \Big( C_0 \theta ||\nabla (\eta \bar{\eta}_{\varepsilon} F(u))||_{L^{2}(B_1)}^2 + C(n) \theta^{-\alpha} ||(\eta \bar{\eta}_{\varepsilon})F(u)||_{L^2(B_1)}^2 \Big).
\end{align}

Choose $\theta < \Big(10C_0 q ||f||_{L^{2n}(B_1)} \Big)^{-1} c(\varepsilon)$ where $c(\varepsilon)$ is the constant in \eqref{e:[Curv. Est] G' > c F'}. Then we 
\begin{align}
    I_2 \leqslant \frac{c(\varepsilon)}{10} ||\nabla (\eta \bar{\eta}_{\varepsilon} F(u))||_{L^{2}(B_1)}^2 + C(n,\varepsilon,||f||_{L^{2n}(B_{1/2})})q^{\alpha}||(\eta \bar{\eta}_{\varepsilon})F(u)||_{L^2(B_1)}^2.
\end{align}

Then we can run the same arguments as in Lemma \ref{l:Caccioppoli}. Since $||f||_{L^{2n}(B_{1/2})} \leqslant C(n,\varepsilon) ||u^{1+\varepsilon}||_{L^2(B_1)}$, we conclude the proof by obtaining
\begin{align}
    \int_{B_{1/2}} \eta^2 |\nabla u^q|^2 \leqslant C(n,\varepsilon,||u^{1+\varepsilon}||_{L^{2}(B_1)}) q^{\alpha(n)} \int_{B_{1/2}} (|\nabla \eta|^2 + \eta^2) u^{2q}.
\end{align}
\end{proof}

Finally we can run the iteration argument to conclude the $L^{\infty}$-estimate.

\begin{theorem}\label{t:Local Moser}
    Let $u \in C^1(B_1^*)$ and $u \geqslant 0$ satisfying the equation \eqref{e:[Moser] u^{1+q}} in the sense of distributions. Suppose $|u^{q_0}(x,y) | \leqslant \delta |y|^{-2}$ for some $\delta \leqslant \mathbf{\delta_0(n,q_0)}$. Then
    \begin{align}
        ||u||_{L^{\infty}(B_{1/4})} \leqslant C(n,\delta,q_0).
    \end{align}
\end{theorem}

\begin{proof}
Since $|u^{q_0}(x,y) | \leqslant \delta |y|^{-2}$, we have $\|u\|_{L^{2+2\varepsilon}(B_1)} \leqslant C(n,\varepsilon,\delta,q_0)$ for some $\varepsilon \in (0,q_0-1)$. Let $q_* = 1+\varepsilon$ for simplicity.

    We consider the similar setting as in the previous Lemma \ref{l:Cacchio-2}. Namely, let $v_0 = u^{q_*}$ and $v_k = v_{k-1}^{\tau} = (u^{q_*})^{\tau^k}$ with $ \tau =\frac{n}{n-2}$. We further define the radii $r_k = \frac14  +\frac1{4^{k+1}}$. Then we can define the cutoff functions $\eta_k \equiv 1$ in $B_{r_{k+1}}$ and vanishing outside $B_{r_{k}}$.
    
According to Lemma \ref{l:Cacchio-2}, since $v_k = u^{q_*\cdot \tau^k}$ and $|\nabla \eta_k|^2 \leqslant C 10^k$, we have
    \begin{align}
        \int_{B_{1/2}} \eta_k^2 |\nabla v_k|^2 \leqslant C \cdot(q_* \tau^k)^{\alpha}\int_{B_{1/2}} (|\nabla \eta_k|^2 + \eta^2) v_k^2 \leqslant C_1  (10\tau^{\alpha})^k \int_{B_{r_k}} v_k^2.
    \end{align}
where $C_1$ depends on $n, q_*, ||u^{q_*}||_{L^{2}(B_1)}$.

Then by Sobolev inequality and the inequality above, we have
    \begin{align}
        \Big(\int_{B_{r_{k+1}}} v_k^{2\tau} \Big)^{1/\tau} \leqslant \Big( \int_{B_{1/2}} |\eta_k v_k|^{2\tau} \Big)^{1/\tau} \leqslant C \int_{B_{1/2}} |\nabla(\eta_k v_k)|^2 \leqslant C \int_{B_{1/2}} |\nabla\eta_k|^2v_k^2 + \eta_k^2 |\nabla v_k|^2 \leqslant C_2 (10 \tau^{\alpha}+1)^k \int_{B_{r_{k}}} v_k^2,
    \end{align}
where $C_2$ depends on $n, q_*, ||u^{q_*}||_{L^{2}(B_1)}$.

Therefore, we can now estimate $||u^{q_*}||_{L^{2\tau^{k+1}(B_{r_{k+1}})}}$ as follows
\begin{align}
    ||u^{q_*}||_{L^{2\tau^{k+1}}(B_{r_{k+1}})} &= \Big(\int_{B_{r_{k+1}}} v_k^{2\tau} \Big)^{\frac1{2\tau^{k+1}}} \\
    &\leqslant \Big( C_2 (10 \tau^{\alpha}+1)^k \int_{B_{r_{k}}} v_k^2 \Big)^{\frac1{2\tau^k}} \\
    &\leqslant C_3^{\frac{k}{\tau^k}} ||u^{q_*}||_{L^{2\tau^k}(B_{r_{k}})},
\end{align}
where $C_3$ depends on $n,q_*,||u^{q_*}||_{L^2(B_1)}$.

Then by induction we can conclude that
\begin{align}
    ||u^{q_*}||_{L^{2\tau^{k+1}}(B_{r_{k+1}})} \leqslant C_3^{\sum_{j=1}^k \frac{j}{\tau^j}} ||u^{q_*}||_{L^2(B_{1/2})}.
\end{align}

Let $k \to \infty$, then $\tau^k \to \infty$ and $\sum_{j=1}^{\infty} \frac{j}{\tau^j} < C(n)$. Therefore,
\begin{align}
    ||u^{q_*}||_{L^{\infty}(B_{1/4})} \leqslant C(n,q_*,||u^{q_*}||_{L^2(B_1)}).
\end{align}
    This finishes the proof if we choose $\varepsilon = \frac1{2}(q_0-1)$.
\end{proof}

Finally we prove the curvature bound Theorem \ref{t:bound on Rm}.

\begin{proof}[Proof of Theorem \ref{t:bound on Rm}]
Recall that for Einstein metrics we have the following equations: there exist constants $C_k(n)>0$ and $\alpha(n)>0$ such that for any $k\in\mathbb{N}$
\begin{itemize}
    \item $\Delta |\nabla^k\Rm| \geqslant - C_k|\Rm|\,|\nabla^k\Rm|$,
    \item $\Delta |\Rm|^{1-\alpha} \geqslant - C_0|\Rm|^{2-\alpha}$.
\end{itemize}
These inequalities are classical consequences of the fact that on an Einstein metric, the curvature tensor is harmonic (seen as a $\Lambda^2$ valued $2$-form) and a Weitzenböck formula; see \cite{bkn} for example.

Let $u = |\Rm|^{1-\alpha}$ and then $u$ satisfies the equation \eqref{e:[Moser] u^{1+q}} for $q_0 = (1 -\alpha)^{-1}$. By Theorem \ref{c:local curvature est}, we have $u^{q_0}(x,y) = |\Rm|(x,y) \leqslant \delta |y|^{-2}$. Therefore, the proof of the bound for $|\Rm|$ is now completed by Theorem \ref{t:Local Moser}. The estimate for higher order derivatives will be standard; see for example \cite{DS}.
\end{proof}

\subsection{Analytic orbifold coordinates}

Recall that in the previous section \ref{s:C0 chart}, we construct the $C^0$ coordinates around curves of singularities by gluing techniques. In this section, we will improve the regularity of the coordinate charts using the curvature estimate $|\nabla^k \Rm| \leqslant C$ following \cite{DS}.

First let us recall our setting: Let $(X,p)$ be a 5-dimensional noncollapsed limit space. Let $\gamma:[-2, 2 ] \subset \RR \to X$ be a 1-Lipschitz curve with $\gamma(0) = p$ parametrized by arc length. Suppose each $\gamma(t)$ have the same tangent cone $\RR \times \RR^4/\Gamma$ with $\Gamma \neq id$. Then by Proposition \ref{prop:C^0 metric} we know $B_1(p) \setminus \gamma$ is a smooth manifold with metric $g$. 

Let $\cC^\gamma_j \equiv \{ x \in X : d(x, \gamma) = j^{-1} \}$ be the cylindrical hypersurface surrounding $\gamma$. The regularity of such cylinders is not good enough for higher order derivative estimates. So we need to construct a better set of cylinders. We define $\cC_0 \equiv (-1,1) \times \mathbb{S}^3 \subset \RR \times \RR^4$ to be the standard cylinder in $\RR^5$.

According to Lemma \ref{l:Horizontal Gluing}, we have the following immediate corollary.

\begin{lemma}\label{l:coordinate, appro cylinder}
    There exists a sequence of smooth embeddings $f_j: \cC_0 = (-1,1) \times \mathbb{S}^3 \to X$ such that the following holds:
    \begin{enumerate}
        \item $d_{GH} (\cC_{f_j}, \cC^\gamma_j) \cap B_1(p)) \leqslant \varepsilon_j$, where $\cC_{f_j} \equiv f_j(\cC_0)$;
        \item $\big| j^2 f_j^*g - h_0  \big|_{C^4_{h_0}} \leqslant \varepsilon_j$;
        \item $\big| j^{-1} A_{\cC_j} - id   \big|_{C^3_{h_0}} \leqslant \varepsilon_j$. 
    \end{enumerate}
\end{lemma}

Then we can construct the analytic coordinate around the curve. 
\begin{theorem}\label{t:analytic chart}
    There exists an analytic diffeomorphism $F : B^* \to B_1(p) \setminus \gamma$ such that $F^*g$ extends to an analytic metric on the whole $B$. 
\end{theorem}

\begin{proof}

First we define the maps $F_j : \cC_0 \times [j^{-1}, 1]$ as follows:
\begin{equation*}
    F_j(z,r) = \exp_{f_{j}(z)} \Big((r-j^{-1}) N(f_j(z)) \Big).
\end{equation*}
where $z=(t,\theta) \in \cC_0 = (-1,1) \times \mathbb{S}^3$ and $N(f_j(z))$ is the unit outward normal vector of $\cC_{f_j}$ at $f_j(z)$. 

We want to show that each $F_j$ is a diffeomorphism from $\cC_0 \times [j^{-1}, 1]$ to its image. To prove this, let $J(r)$ a Jacobi field along the geodesic $\alpha_z(r) \equiv F_j(z,r)$. Since we have $|\Rm| \leqslant C$ in $B_1(p) \setminus \gamma$ by Theorem \ref{t:bound on Rm}, by Rauch comparison theorem there exists some constant $C>0$ and 
\begin{equation*}
    C^{-1} |J(j^{-1})|_g \leqslant |J(r)|_g \leqslant C j |J(j^{-1})|_g.
\end{equation*}
Hence $F_j$ has no critical point in $\cC_0 \times [j^{-1}, 1]$.

We can write the metric $F_j^*g =  dr^2 + r^2 h_j(z,r)$. Let $h_0$ be the standard metric on $ \cC_0 =  (-1,1) \times \mathbb{S}^3$ and we will prove $h_j \to h_0$ as $j \to \infty$ in $C^{1,\alpha}$. 

Given $z\in (-1,1)\times \mathbb{S}^3$ and a unit tangential vector $\xi$ at $z$. Consider the Jacobi field $J(r)$ along the geodesic $\alpha_z(r) = F_j(z,r)$ satisfying
\begin{enumerate}
    \item $J(j^{-1}) =  df_j(\xi)$;
    \item $J'(j^{-1}) = A_{\cC_j}(J(j^{-1}))$.
\end{enumerate}

This implies that $J(r) = dF_j|_{(z,r)}(\xi)$. Moreover, by Lemma \ref{l:coordinate, appro cylinder} we have
\begin{enumerate}
    \item $\big| |J(j^{-1})|_g - j^{-1}    \big| \leqslant j^{-1} \varepsilon_j$;
    \item $\big| |J'(j^{-1})|_{g} - j |J(j^{-1})|_g  \big| \leqslant 2\varepsilon_j$.
\end{enumerate}

Let $\{e_1(r), ..., e_5(r) = \alpha'(r)\}$ be the parallel orthonormal frame along $\alpha_z(r)$ and write $J(r) = \sum_{k=1}^5 J_k(r) e_k(r)$. Then by the Jacobi field equation we have
\begin{equation*}
    J_k''(r) + \sum_{l} R_{k5l5}|_{\alpha_z(r)}J_l(r) = 0,
\end{equation*}
where $R_{k5l5} = \langle \Rm(e_k,e_5)e_l , e_5 \rangle$. Then
\begin{equation*}
\begin{split}
        J_k(r) &= J_k(j^{-1}) + J'_k(j^{-1}) \cdot (r- j^{-1}) - \int_{j^{-1}}^r \int_{j^{-1}}^s R_{k5l5}|_{\alpha_z(t)}J_k(t) dt ds.
\end{split}
\end{equation*}

By Theorem \ref{t:bound on Rm}, we have $|\Rm| \leqslant C_1$. This implies that for any $r \in [j^{-1},1]$ we have
\begin{equation*}
    \Big| |J(r)|_g - r \Big| \leqslant C_2 (j^{-1} + \varepsilon_j r + r^3).
\end{equation*}
Hence, at the metric level,
\begin{equation*}
    \Big| h_j - h_0 \Big|_{C^0_{h_0}} \leqslant C_2 (j^{-1} r^{-1} + \varepsilon_j  + r^2).
\end{equation*}

By differentiating the Jacobi equation and the bound $|\nabla^k \Rm| \leqslant C(k)$ we obtain the high order derivative estimates
\begin{equation*}
\begin{split}
    r^{-1} \Big|\nabla^{g_0}( h_j - h_0) \Big|_{C^0_{h_0}} \leqslant C_3 (j^{-1} r^{-2} + \varepsilon_j r^{-1}  + r); \\
    r^{-2}\Big|\nabla^{g_0} \nabla^{g_0} ( h_j - h_0) \Big|_{C^0_{h_0}} \leqslant C_4 (j^{-1} r^{-2} + \varepsilon_j r^{-1} + 1). 
\end{split}
\end{equation*}

For any fixed $r \ll 1$ , we let $j\to \infty$. Then the metric $F^*_j$ converges in $C^3$ to some map on $[r ,1] \times \cC_0$. Then we let $r \downarrow 0$ and thus we obtain a limit map $F : (0,1] \times \cC_0$ such that the metric $F^*g$ satisfies
\begin{equation*}
    r^{-2} \Big| F^*g - g_0\Big|_{C^0_{g_0}} + r^{-1} \Big| F^*g - g_0\Big|_{C^1_{g_0}} + \Big| F^* g - g_0\Big|_{C^2_{g_0}} \leqslant C_5.
\end{equation*}

This implies that $F^*g$ extends to a $C^{1,1}$ metric on the whole ball $B_1$. It is classical that if an Einstein metric is $C^{1,\alpha}$ in some coordinates, then they can be upgraded to a better set of coordinates where the metric is real-analytic. This can be proven thanks to harmonic coordinates, or coordinates in Bianchi gauge or divergence-free gauge with respect to a background flat metric.
\end{proof}

\subsection{Structure for curves of singularities}

We now prove Theorem  \ref{mthm: orbifold regularity} for curves of singularities.

\begin{proof}[Proof of Theorem \ref{mthm: orbifold regularity}]
Let $\gamma: I \to \cS(X)$ be a singular curve and $\cV(t)$ be the volume density at the point $\gamma(t)$ for $t\in I$. Since $\cV$ is lower semi-continuous with values in a finite set, then there exists a nowhere dense subset $\cI \subset I$ such that
\begin{enumerate}
    \item $I = \cI \cup \bigcup_{j=1}^{\infty} (s_j, s_{j+1})$;
    \item $\cV$ is constant on each interval $(s_j,s_{j+1})$. 
\end{enumerate}

By Corollary \ref{c:curve constant volume ratio}, for each $j$ there exists some $\Gamma_j$ such that the tangent cones on each segment $\gamma((s_j,s_{j+1}))$ are all isometric to $\RR\times \RR^4/\Gamma_j$. Combining Theorem \ref{t:bound on Rm} and \ref{t:analytic chart}, we can prove (3). To prove (2), note that the metric is invariant under the nontrivial group $\Gamma_j$ acting by isometry and fixing a curve. It is a classical consequence that this curve must be a geodesic.    
\end{proof}

\subsection{Geodesics on limits of Einstein $5$-manifolds}

We conclude by classifying all geodesics in the limit space $X$. We first recall the H\"older continuity estimate along geodesic in \cite{con12} and also \cite{deng25}.

%In a seminal work \cite{con12}, Colding and Naber proved the H\"older continuity of tangent cones along any \emph{limit geodesic} that arises as the limit of geodesics in the sequence. More recently Deng \cite{deng25} generalized the estimate to \textit{any} geodesic in $X$, that is a curve locally minimizing the distance $d$.

\begin{theorem}\label{t:continuous tangent cone}
    Let $\sigma:[0,l] \to X$ be a geodesic in the limit space $X$. Then there exists $\alpha,C>0$ such that given any $\delta>0$ with $\delta l <s<t<l -\delta l$ and $Y_{\sigma(s)},Y_{\sigma(t)}$ tangent cones from the same sequence of rescalings, then we have
    \begin{equation*}
        d_{GH}(B_1(Y_{\sigma(s)}), B_1(Y_{\sigma(t)})) < \frac{C}{\delta l}|s-t|^{\alpha}.
    \end{equation*}
\end{theorem}

As an immediate consequence, if one point of the geodesic lies in the regular set $\cR$, then the entire interior of the geodesic lies in $\cR$ \cite{con12}. In particular the regular set $\cR$ is convex. Using Theorem \ref{t:continuous tangent cone} together with our isolation Theorem \ref{thm:isolation 1sym}, we can now classify all geodesics in $X$: 

\begin{proof}[Proof of Theorem \ref{t:limit geodesic}]
    It suffices to consider the case where the entire $\sigma((0,l)) \in \cS$. By the stratification Theorem \ref{t:new stratification}, there must be some $\Gamma$ and $t_0 \in (0,l)$ such that $\sigma(t_0) \in \cS^1_{\Gamma}$ since $\cS^0$ is only countable. Then by Theorem \ref{t:continuous tangent cone} and isolation Theorem \ref{thm:isolation 1sym}, there exists some $\theta_0>0$ such that the tangent cone at $\sigma(t)$ is isometric to $\RR \times \RR^4/\Gamma$ for any $t\in (t_0-\theta_0, t_0 + \theta_0)$. By connectedness this implies that any interior point has the unique tangent cone $\RR \times \RR^4/\Gamma$. Therefore we can use Theorem \ref{thm:lips curve sing} to prove that the interior of this geodesic is removable in the sense of Theorem \ref{thm:lips curve sing}.    
\end{proof}

Theorem \ref{t:limit geodesic} therefore provides a complete classification of the interior of all geodesics. We say that a geodesic is of \emph{$\Gamma$-type} if the tangent cones at all interior points are $\RR \times \RR^4/\Gamma$ (we allow $\Gamma = \{id\}$). We now turn to the behavior of the end points. We prove that the tangent cones at end points can only have worse singularities. Similar results can be formulated for any curves thanks to Theorem \ref{thm:lips curve sing}.

\begin{proposition}
    Let $\sigma:[0,l) \to X$ be a geodesic of $\Gamma$-type in $X$. Let $C(Z)$ be a tangent cone at $\sigma(0)$. Then $Z$ has at least one singularity of the type $\RR^4/\Gamma'$ with $|\Gamma'|\ge |\Gamma|$.
\end{proposition}

\begin{proof}
    Let $C(Z)$ be the tangent cone at $\sigma(0)$ from the rescaling sequence $s_i$. By passing to a subsequence, we have $\sigma(s_i) \to z$ for some point $z \in Z$. Since $\sigma$ is of $\Gamma$-type, by Bishop-Gromov monotonicity and Colding's volume convergence theorem, the volume density at $z \in Z\subset C(Z)$ cannot be larger than that of $\RR^4/\Gamma$. This completes the proof.
\end{proof}

\appendix

\part{Appendix: Analysis in weighted spaces}\label{function spaces}

   \section{Function spaces}\label{sec:function spaces}
    
    For a tensor $s$, a point $x$, $\alpha>0$ and a Riemannian manifold $(M,g)$. The Hölder seminorm is defined as
$$ [s]_{C^\alpha(g)}(x):= \sup_{\{y\in T_xM,|y|< \textup{inj}_g(x)\}} \Big| \frac{s(x)-s(\exp^g_x(y))}{|y|^\alpha} \Big|_g.$$

For orbifolds, we consider a norm which is bounded for tensors decaying at the singular points.
\begin{defn}[Weighted Hölder norms on an orbifold]\label{norme orbifold}
    Let $\beta\in \mathbb{R}$, $k\in\mathbb{N}$, $0<\alpha<1$ and $(M_o,g_o)$ an orbifold and $S_o$ a subset of its singularities. Then, for any tensor $s$ on $M_o$, we define
    \begin{align*}
        \| s \|_{C^{k,\alpha}_{\beta}(g_o)} &:= \sup_{M_o}r_o^{-\beta}\Big(\sum_{i=0}^k r_o^{i}|\nabla_{g_o}^i s|_{g_o} + r_o^{k+\alpha}[\nabla_{g_o}^ks]_{C^\alpha(g_o)}\Big).
    \end{align*}
\end{defn}

For ALE manifolds, we will consider a norm which is bounded for tensors decaying at infinity.

\begin{defn}[Weighted Hölder norms on an ALE orbifold]\label{norme ALE}
Let $\beta\in \mathbb{R}$, $k\in\mathbb{N}$, $0<\alpha<1$ and $(N,g_b)$ be an ALE manifold and $S_b$ a subset of its singularities. Then, for all tensor $s$ on $N$, we define
   \begin{align*}
       \| s \|_{C^{k,\alpha}_{\beta}(g_b)}:= \sup_{N}r_b^\beta\Big( \sum_{i=0}^kr_b^{i}|\nabla_{g_b}^i s|_{g_b} + r_b^{k+\alpha}[\nabla_{g_b}^ks]_{C^\alpha({g_b})}\Big).
   \end{align*}
\end{defn}

On $M$, using the partition of unity 
$1= \chi_{M_o^t} + \sum_j \chi_{N_j^t}$ from Definition \ref{def cutoffs all}, we define a global norm.
\begin{defn}[Weighted Hölder norm on a naïve desingularization]\label{norme a poids M}
		Let $\beta\in \mathbb{R}$, $k\in\mathbb{N}$ and $0<\alpha<1$. We define for $s\in TM^{\otimes l_+}\otimes T^*M^{\otimes l_-}$ a tensor $(l_+,l_-)\in \mathbb{N}^2$, with $l:= l_+-l_-$ the associated conformal weight,
		$$ \|s\|_{C^{k,\alpha}_{\beta}(g^D)}:= \| \chi_{M_o^t} s \|_{C^{k,\alpha}_{\beta}(g_o)} + \sum_j T_j^{\frac{l}{2}}\|\chi_{N_{j}^t}s\|_{C^{k,\alpha}_{\beta}(g_{b_j})}.$$
\end{defn}

 With the notations of Definition \ref{orb Ein}, denote for each singular point $k$, $A_k(t,\varepsilon_0) := (\Phi_k)_*A_{{e}}(\varepsilon_0^{-1}\sqrt{t_k},\varepsilon_0)$ and $B_k(\varepsilon_0):=(\Phi_k)_*B_{{e}}(0,\varepsilon_0)$, as well as cut-off functions $\chi_{A_k(t,\varepsilon_0)}$ and $\chi_{B_k(\varepsilon_0)}$ respectively supported in $A_k(t,\varepsilon)$ and $B_k(\varepsilon_0)$, and equal to $1$ on $A_k(t,2\varepsilon_0)$ and $B_k(\varepsilon_0/2)$.

\begin{defn}[$C^{k,\alpha}_{\beta,*}$-norm on $2$-tensors]\label{def: Ckalphabeta*}
    Let $ h $ be a $2$-tensor on $(M,g^D)$, $(M_o,g_o)$ or $(N,g_b)$. We define
    $$\|h\|_{C^{k,\alpha}_{\beta,*}}:= \inf_{h_*,H_k} \|h_*\|_{C^{k,\alpha}_{\beta}} + \sum_k |H_k|_{{e}},$$
    where the infimum is taken on the $(h_*,(H_k)_k)$ satisfying $h= h_*+\sum_k \chi_{A_k(t,\varepsilon_0)}H_k$ for $(M,g^D)$ or $h= h_*+\sum_k \chi_{B_k(\varepsilon_0)}H_k$ for $(M_o,g_o)$ or $(N,g_b)$, where each $H_k$ is a constant traceless symmetric $2$-tensors on $\mathbb{R}^4\slash\Gamma_k$.
\end{defn}

\begin{defn}[$rC^{k,\alpha}_{\beta,*}$-norm of a vector fields]
    Let $ X $ a vector field on $(M,g^D)$ (respectively $(M_o,g_o)$ or $(N,g_b)$). We define its $rC^{k,\alpha}_{\beta,*}$-norm, where $r$ is the function $r_D$ (respectively $r_o$ or $r_b$) by
    $$\|X\|_{rC^{k,\alpha}_{\beta,*}}:= \inf_{X_*,X_k}\|X_*\|_{rC^{k,\alpha}_{\beta}} + \sum_k \|X_k\|_{rC^0_0(e)},$$
    where the infimum is taken among the couples $(X_*,X_k)$ satisfying $X = X_* + \sum_k\chi_{\mathcal{A}_k(t,\varepsilon)}X_k$ (respectively $X = X_* + \sum_k\chi_{B_o(\varepsilon)}X_k$ or $X = X_* + \chi_{B_b(\varepsilon)}X_k$).
\end{defn}

{ A last function space is adapted to the asymptotics of $2$-tensors and vector fields at ALE ends.

\begin{defn}[$C^{2,\alpha}_{\beta,**}$ norm on a tree of ALE orbifold]\label{norm orbifold ALE}
    Let $(N,g^B)$ a naïve gluing of a tree of Ricci-flat ALE orbifolds. Let $h$ be a $2$-tensor on $N$, and assume that $h = H^2 + H^4 + \mathcal{O}(r_b^{-4-\beta})$ for $\beta>0$, where $H^2$ and $H^4$ are homogeneous harmonic symmetric $2$-tensors with $|H^j| \propto r^{-j}$ on $\mathbb{R}^4/\Gamma$. For $\chi$, a cut-off function supported in $N\backslash K$ of Definition \ref{def orb ale}. We define the $C^{2,\alpha}_{\beta,**}$-norm of $h$ as:
    $$\|h\|_{C^{2,\alpha}_{\beta,**}}:= \|r^2H^2\|_{L^\infty(e)} +\|r^4H^4\|_{L^\infty(e)} + \big\|(1+r_b)^4(h-\chi (H^2+H^4))\big\|_{C^{2,\alpha}_{\beta,*}}.$$
\end{defn}

\begin{remark}\label{rem:gauge **}
    If for a metric $g$ such that $\|g - g^B\|_{C^{2,\alpha}_{\varepsilon}} \leqslant 1$, one has $\delta_{g} h = 0$ and $h\in C^{2,\alpha}_{\beta,**}$, then there is no term $H^2$ in $r^{-2}$ in the decomposition of $h$ since these terms are not divergence-free by \cite[Lemma 2.50]{ozuthese}. There is no term in $r^{-3}$ because of the action of $\Gamma\neq \{\mathrm{id}\}$.
\end{remark}

\begin{defn}[$r_bC^{3,\alpha}_{\beta,**}$-norm on a tree of ALE orbifold]
    Let $(N,g^B)$ a naïve gluing of a tree of Ricci-flat ALE orbifolds. Let $X$ be a vector field on $N$, and assume that $X = Y^3 + \mathcal{O}(r_b^{-3-\beta})$ for $\beta>0$ and $Y^3$ a homogeneous element of the kernel of $\delta_e\delta_e^*$ with $|Y^3|\sim r_e^{-3}$. We define its $r_bC^{3,\alpha}_{\beta,**}$-norm by
    $$\|X\|_{r_bC^{3,\alpha}_{\beta,**}}:= \sup \|r_e^3Y^3\|_{L^\infty(e)} + \big\|(1+r_b)^4(X-\chi(\varepsilon r_b) Y^3)\big\|_{r_bC^{3,\alpha}_{\beta,*}}.$$
\end{defn}}

\begin{remark}
    The norms $ C^{2,\alpha}_{\beta,**}$ and $ rC^{3,\alpha}_{\beta,**}$ only differ from the norms $ C^{2,\alpha}_{\beta,*}$ and $ rC^{3,\alpha}_{\beta,*}$ in the way they weigh a neighborhood of infinity. They are the same in the rest of the manifold.
\end{remark}

The reason why we consider these two function spaces is because we have the following mapping properties for the linearization $\bar{P}$ of $\mathbf{\Phi}$. The following statement corrects \cite[Lemma 5.5]{ozu2}.

\begin{lem}\label{lem:inverting barP with terms at infty}[{Corrected statement of \cite[Lemma 5.5]{ozu2}}]
    Let $(N,g^B)$ be a tree of ALE orbifolds. Then, there exists $C>0$ such that we have, for any $h\perp \tilde{\mathbf{O}}(g^B)$,
    \begin{equation}
    \|h\|_{C^{2,\alpha}_{\beta,**}}\leqslant C \big\|(1+r_b)^4\bar{P}_{g^B}h\big\|_{r_D^{-2}C^{\alpha}_\beta}.\label{inverse ALE term 4}
    \end{equation}
\end{lem}

We also have a similar statement for $\delta\delta^*$.

\begin{lem}[{\cite[Lemma 5.6]{ozu2}}]\label{lem:controle inverse mise jauge ALE}
    Let $(N,g^B)$ be a tree of ALE orbifold. Then, there exists $C>0$ such that for any vector field $X$ on $N$, we have
    \begin{equation}
        \|X\|_{r_BC^{3,\alpha}_{\beta,**}}\leqslant C \big\|(1+r_B)^4\delta_{g^B}\delta_{g^B}^*X\big\|_{r_B^{-1}C^{1,\alpha}_\beta}.\label{inverse delta ALE term 4}
    \end{equation}
\end{lem}

\begin{remark}
    In \cite{ozu2}, the $H^2$ term was omitted in the definition of the $C^{2,\alpha}_{\beta,**}$ norm, see \cite[Definition 5.3]{ozu2}. We a priori need to include these $H^2$ terms even though they will instantly be killed by our gauge condition, see Remark \ref{rem:gauge **}. In particular in their current forms, \cite[Lemma 5.5 and Corollary 5.7]{ozu2} should either include a $H^2$ term, or impose that $h$ is in gauge. 
    
    This does not affect the results of \cite{ozu2}: \cite[Corollary 5.7]{ozu2} is only applied to gluings of hyperkähler ALE metrics where there is no obstruction, and the $H^2$ term do not appear since they are not divergence-free. Concretely, in \cite{ozu2}, Lemma 5.5 and Corollary 5.7 are mentioned in Lemma 5.15, but only applied to the obstructions $\mathbf{o}_j$ which are in gauge, and they are mentioned in Propositions 6.15 and 6.16 where they are applied to hyperkähler metrics where there are no obstructions and the tensors are in gauge. 
\end{remark}

\section{Duality for vector fields}\label{sec:duality vect fields}

We will have to carefully understand the duality between elements of the kernel of $\delta_{e}\delta_{e}^*$, where we identify $1$-forms and vector fields through $X\sim g(X,\cdot)$. From \cite[Section 3.1]{ozu2}, which essentially follows \cite[Section 2]{ct}, the homogeneous $1$-forms in $\ker\delta_{e}\delta_{e}^*$ are of the form
\begin{enumerate}
    \item $r^{a^\pm_j}\psi$ with $a^\pm_j:= \pm (1+j)$, $j\in \mathbb{N}^*$, where $\psi$ is an eigenvector the Hodge Laplacian,
    \item $r^{b^{\pm}_j}d_{\mathbb{S}^3\slash\Gamma}\phi \;+\; b^\pm_j r^{b^\pm_j-1}\phi dr$, $b_j^\pm = -1 \pm (1+j)$ or 
    \item $2r^{c^{\pm}_j+2}d_{\mathbb{S}^3\slash\Gamma}\phi \;+\; c^\mp_j r^{c^\pm_j+1}\phi dr$, with $c_j^\pm = -1 \pm (1+j)$, $j\in \mathbb{N}$ and $\phi$ an eigenfunction of the Laplacian.
\end{enumerate}
Consider for each of the above tensors, $(\phi_i^j)_i$ and $(\psi_i^j)_i$ orthonormal bases of eigentensors of the eigenspaces above. Denote $X_{a_j^\pm,i}$, $X_{b_j^\pm,i}$ and $X_{c_j^\pm,i}$ the associated vector fields in $\ker \delta_e\delta_e^*$.

Among these tensors, those whose norm is in $r$ correspond to: $a_1^+$, $b_2^+$ and $c^+_0$. Those with a norm in $r^{-3}$ correspond to
and those whose norm grows like $r^{-3}$ are: $a_1^-$, $b_0^-$ and $c^-_2$.

An application of \cite[Lemma 3.4]{ozu2} yields the following result.
\begin{lem}\label{lem harmonic vect field linear}
    Let $(N,g^B)$ be a tree of ALE orbifold. Then, for any linear combination $X_1:=\alpha X_{a_1^+} + \beta X_{b_2^+} + \gamma X_{c_0^+}$ for $\alpha,\beta,\gamma\in\mathbb{R}$, there exists a unique vector field $\mathcal{X}_1$ on $N$ such that: for a cut-off $\chi$ supported in a neighborhood of infinity of $N$ where $g^B$ has ALE coordinates,
    \begin{equation}
        \left\{\begin{aligned}
            &\delta_{g^B}\delta_{g^B}^*\mathcal{X}_1 = 0,\\
            &\mathcal{X}_1-\chi X_1\in r_BC^{3,\alpha}_{\delta}, \text{ for } 0<\delta<4
        \end{aligned}\right.
    \end{equation}
    We additionally have:
    \begin{equation}\label{eq:control X1}
        \|\mathcal{X}_1-\chi X_1\|_{r_BC^{3,\alpha}_{\delta}}\leqslant C \|r^{-1}X_1\|_{L^\infty(g_{e})}
    \end{equation}
\end{lem}

There is a duality between  tensors in $\ker \delta_e\delta^*_e$ whose norms are growing in $r$ and in $r^{-3}$. Namely, when integrating by parts $\int_M\delta\delta^*X(Y)-\delta\delta^*Y(X)$ on an ALE manifold $M$, one gets a boundary term $$B(X,Y):=\lim_{s\to\infty}\int_{\{r=s\}} \left(\delta^*X(Y,\partial_r)-\delta^*Y(X,\partial_r)\right).$$
It turns out that if $Y = Y^3 + o(r^{-3})$ and $X = X_1 + o(r)$ for $|Y^3|\propto r^{-3}$ and $|X_1|\propto r$, then $B(X,Y) = B(X_1,Y^3)$ defines an important coupling. Explicit computations show that for three positive constants $C_a$, $C_b$ and $C_c$, one has: for any $\alpha_i,\beta_i,\gamma_i,\alpha'_k,\beta'_k,\gamma'_k\in\mathbb{R}$,
\begin{equation}\label{eq:duality vector fields}
    \sum_{i,k}B\left(\alpha_i X_{a_1^+,i} + \beta_i X_{b_2^+,i} + \gamma_i X_{c_0^+,i}\,,\,\, \alpha'_k X_{a_1^-,k} + \beta'_k X_{c_2^-,k} + \gamma'_k X_{b_0^-,k}\right) = \sum_{i}C_a \alpha_i\alpha'_i + C_b\beta_i\beta'_i + C_c\gamma_i\gamma'_i.
\end{equation}
In other words, the bilinear form $B$ essentially diagonalizes in the above basis. This is a tool to precisely estimate how large the terms in $r^{-3}$ are in the expansion of vector fields in $r_BC^{3,\alpha}_{\delta,**}$.
\begin{proposition}\label{prop:duality vect fields
}
    Let $Y = Y^3 + o(r^{-3})\in r_BC^{3,\alpha}_{\delta,**}$ for $\delta>0$, with a decomposition $Y^3=\alpha'_k X_{a_1^-,k} + \beta'_k X_{c_2^-,k} + \gamma'_k X_{b_0^-,k}$, consider $X_1 = \alpha_i X_{a_1^+,i} + \beta_i X_{b_2^+,i} + \gamma_i X_{c_0^+,i}$ and the associated $\mathcal{X}_1$ from Lemma \ref{lem harmonic vect field linear}. Then, one has 
$$\int_M\delta\delta^*Y(\mathcal{X}_1)=-\sum_{i}C_a \alpha_i\alpha'_i + C_b\beta_i\beta'_i + C_c\gamma_i\gamma'_i.$$

\vspace{-0.3cm}

\end{proposition}
\begin{proof}
    This is a direct consequence of the following integration by parts: using \eqref{eq:duality vector fields} in the last line,
\begin{align*}
    -\int_M\delta\delta^*Y(\mathcal{X}_1)&=\int_M\delta\delta^*\mathcal{X}_1(Y)-\delta\delta^*Y(\mathcal{X}_1)= B(Y^3,X_1)\\
    &=\sum_{i}C_a \alpha_i\alpha'_i + C_b\beta_i\beta'_i + C_c\gamma_i\gamma'_i.
\end{align*}
\end{proof}

\end{document}